\documentclass{surv-l}

\usepackage{amssymb}

\usepackage{graphicx}

\usepackage[cmtip,all]{xy}

\usepackage{}

\newtheorem{theorem}{Theorem}[chapter]
\newtheorem{lemma}[theorem]{Lemma}
\newtheorem{proposition}[theorem]{Proposition}
\newtheorem{corollary}[theorem]{Corollary}
\newtheorem{fact}[theorem]{Fact}

\theoremstyle{definition}
\newtheorem{definition}[theorem]{Definition}
\newtheorem{example}[theorem]{Example}
\newtheorem{notation}[theorem]{Notation}

\theoremstyle{remark}
\newtheorem{remark}[theorem]{Remark}
\newtheorem{question}[theorem]{Question}

\numberwithin{section}{chapter}
\numberwithin{equation}{chapter}

\newcommand{\CC}{\mathbb{C}}
\newcommand \EE {\mathbb{E}}
\newcommand \dieu {\mathbb{D}}
\newcommand{\FF}{\mathbb{F}}
\newcommand{\GG} {\mathbb{G}}
\newcommand{\NN} {\mathbb{N}}
\newcommand{\PP}{\mathbb{P}}
\newcommand{\QQ}{\mathbb{Q}}
\newcommand{\RR}{\mathbb{R}}
\newcommand{\ZZ}{\mathbb{Z}}

\newcommand{\cA}{\mathcal{A}}
\newcommand{\cH}{\mathcal{H}}
\newcommand{\cM}{\mathcal{M}}
\newcommand{\cT}{\mathcal{T}}

\newcommand{\cC}{\mathcal{C}}
\newcommand{\CJ} {\mathcal{J}}
\newcommand{\Sh} {\mathrm{Sh}}
\newcommand{\CX} {\mathcal{X}}

\newcommand \cf {\mathfrak{f}}
\newcommand \hh {\mathfrak{h}}

\newcommand{\ang}[1]{\langle #1 \rangle}
\newcommand{\ooo}{\mathrm{ord}}
\newcommand{\sss}{\mathrm{ss}}
\newcommand \QM {{\QQ[\mu_m]}}

\makeindex

\begin{document}

\frontmatter

\title{The Torelli locus and Newton polygons}

\author{Rachel Pries}
\address{Math department, Colorado State University, Fort Collins, Colorado}
\curraddr{}
\email{pries@colostate.edu}
\thanks{
I would like to thank:\\
the organizers of the Arizona Winter School for inviting me to speak in 2024;\\
Jeff Achter and Valentijn Karemaker for their support and mathematical insights;\\
Frans Oort and Douglas Ulmer for their inspiring mentorship;\\
Steven Groen for helping lead the project groups for the AWS lecture series;\\
the collaborators and colleagues who helped me improve this manuscript;\\
and, finally, the NSF for their partial support (DMS-22-00418).}





\date{}

\maketitle


\setcounter{page}{5}

\tableofcontents


\mainmatter


\chapter{Introduction} \label{C0introduction}

\section{Purpose and goals}

This manuscript is about abelian varieties that are Jacobians of curves.
I started writing it for a lecture series at the Arizona Winter School in 2024 on abelian varieties.
A longer more descriptive title might be
\emph{The Torelli locus in the moduli space of abelian varieties, with
applications to Newton polygons of curves in positive characteristic.}
To elaborate, this manuscript covers two topics:
the first is about the \emph{geometry} of the Torelli locus;
the second is about the \emph{arithmetic} invariants of 
abelian varieties that occur for Jacobians of smooth curves in positive characteristic.

The main content is about (families of) abelian varieties, 
with a focus on those with special properties, including
abelian varieties that are Jacobians of curves, 
that have many automorphisms, or that have unusual Newton polygons in positive characteristic.  
Throughout, there is a consistent interplay between the themes of arithmetic and geometry.
The dimension of an abelian variety places constraints on its arithmetic invariants; conversely, 
arithmetic invariants of an abelian variety place constraints on its geometry.

Jacobians of curves can often be studied in a more explicit and concrete way than one can study a typical abelian variety.  
On the other hand, there are techniques for studying (families of) abelian varieties that do not apply when studying (families of) Jacobians of curves.  This leads to a valuable and rewarding exchange between these topics.
Furthermore, we can study moduli spaces of abelian varieties (or moduli spaces of curves).
These moduli spaces can be refined by adding additional structure.  They 
can also be stratified by a given invariant, where the points of a stratum represent abelian varieties (or curves) 
having that invariant.

One of my goals is to provide a broad perspective within this document.
For simplicity, this manuscript restricts to working over an algebraically closed field, 
but includes material both over $\CC$ (characteristic $0$) and 
over $\bar{\FF}_p$ (characteristic $p$).  
I include a diverse range of open questions, some of which are well-known goals 
within the arithmetic geometry community, and some of which emerged from my own research program.  
Another goal is to cover advanced material in a grounded way.  
The purpose of the examples is to make the concepts and techniques more accessible.
More details about the technical background can be found in the provided references, which in turn connect to 
many other important papers on these topics.
I hope that this manuscript is both broad and accessible enough to encourage many people 
to work on these captivating open questions.

\section{Overview of topics by chapter} \label{Stopics}

\subsection{Situating the content}

Throughout, we work over an algebraically closed field $k$.
In Chapters~\ref{C1Atorelli}, \ref{C2Aboundary}, \ref{C4Aspecial}, and (the first half of) \ref{C3Acovers},
the material applies for $k$ of any characteristic (e.g., $\CC$ or $\bar{\FF}_p$).
The material in Chapters~\ref{Chapter3}, \ref{Chapter4}, \ref{C2Bboundary}, \ref{Chapter4BNPspecial}, and (the second half of) \ref{C3Acovers}
applies when $k$ has positive characteristic (e.g., $\bar{\FF}_p$).

Chapters~\ref{C1Atorelli}, \ref{Chapter3}, and \ref{Chapter4} cover topics about abelian varieties and curves defined over 
an algebraically closed field $k$.
In Chapters~\ref{C2Aboundary} and \ref{C2Bboundary}, we switch to working with moduli spaces, 
whose points parametrize (families of) abelian varieties and curves.
In Chapters~\ref{C4Aspecial}, \ref{Chapter4BNPspecial}, \ref{C3Acovers}, we consider abelian varieties and curves that 
have non-trivial automorphisms, and the moduli spaces for families of these.

Every chapter ends with open questions.  Most likely these problems have widely varying levels of difficulty, and
can be approached with multiple techniques.
In some cases, these questions are well-known, and have been a primary goal of a research community for many decades.
For example, this includes variations of the Schottky problem in Chapter~\ref{C1Atorelli}, conjectures on complete 
subvarieties of moduli spaces in Chapter~\ref{C2Aboundary}, and the Coleman--Oort conjecture in Chapter~\ref{C4Aspecial}.

In general, the open problems in positive characteristic are not as well-known. 
Some highlights are the question of the existence of supersingular curves in Chapter~\ref{Chapter4}, the question of irreducibility of the 
non-ordinary locus of curves in Chapter~\ref{C2Bboundary}, and the question of the generic $a$-number of curves with $p$-rank $0$ 
in Chapter~\ref{C3Acovers}.

Throughout, many references are provided.  It is not possible to include references for everyone's work.
I encourage the reader to view these references as a starting point for discovering the literature.

\subsection{Chapter~\ref{C1Atorelli}: Abelian varieties and Jacobians of curves}

In this chapter, we provide a brief background about abelian varieties and Jacobians of curves.
We give examples of several types of abelian varieties and curves which play an important role in this manuscript.

Let $g$ be a positive integer.
Suppose $C$ is a (smooth, projective, irreducible) curve of genus $g$.
The Jacobian $\mathrm{Jac}(C)$ of $C$ is the moduli space of line bundles of degree $0$ on $C$.
Its group of points is isomorphic to the quotient of the group of divisors of degree zero of $C$ by the subgroup of principal divisors.
Many attributes of $C$ can be characterized by its Jacobian.
An important theorem is that the Jacobian $\mathrm{Jac}(C)$ is a principally polarized (p.p.) abelian variety of dimension $g$.

For $1 \leq g \leq 3$, almost every principally polarized abelian variety is a Jacobian.
For example, a p.p.\ abelian variety of dimension $g=1$ is an elliptic curve.
A p.p.\ abelian surface (resp.\ threefold) is the Jacobian of a smooth curve of genus $2$ (resp.\ genus $3$) 
unless it decomposes as a product, together with the product polarization.
For $g \geq 4$, the situation is very interesting because not every p.p.\ abelian variety is a Jacobian.

We end this chapter with several open variations of the Schottky problem
about whether certain types of abelian varieties are Jacobians of curves.

\subsection{Chapter~\ref{Chapter3}: The $p$-torsion invariants}

Suppose $\mathrm{char}(k)=p$.  In this chapter, we explain several generalizations of the distinction between
ordinary and supersingular elliptic curves.

An elliptic curve $E$ over $k$ is \emph{ordinary} if it has a point of order $p$; 
equivalently, $E$ is ordinary if its Newton polygon has slopes of zero and one.
Otherwise, $E$ is \emph{supersingular}.
There are many results about ordinary and supersingular elliptic curves, 
due to Hasse \cite{Hasse35}, Deuring \cite{deuring}, and Igusa \cite{igusa}.
For a fixed prime $p$, almost all elliptic curves are ordinary and the number of 
isomorphism classes of supersingular elliptic curves is approximately $p/12$.  

Suppose $X$ is a p.p.\ abelian variety $X$ defined over $k$.  To classify the action of Frobenius,
there are combinatorial invariants called the $p$-rank, the Newton polygon, 
the Ekedahl--Oort type, and the $a$-number.
In this document, we call these the \emph{$p$-torsion invariants}.

The $p$-rank is the integer $f$ such that the number of $p$-torsion points on $X$ equals $p^f$.
The Newton polygon is determined by the characteristic polynomial of Frobenius 
on the crystalline cohomology. 
The Ekedahl--Oort type
classifies the structure of the $p$-torsion group scheme $X[p]$ of $X$, along with the perfect pairing $X[p] \times X[p] \to \mu_p$. 
The $a$-number is the number of generators of the local-local part of $X[p]$ as a module under $F$ and $V$.

For every positive integer $g$, the possibilities for the $p$-torsion invariants of a p.p\ abelian variety of dimension $g$ are well understood.
We briefly include one question about the interaction between the Newton polygon and Ekedahl--Oort type, 
and about purity of the stratification by Ekedahl--Oort type.

\subsection{Chapter~\ref{Chapter4}: Existence of curves with given $p$-torsion invariants}

Suppose $\mathrm{char}(k)=p$.
For most integers $g$ and primes $p$, 
it is not known which $a$-numbers, Newton polygons, and Ekedahl--Oort types occur for 
Jacobians of smooth curves of genus $g$ over a field of characteristic $p$. 

The $p$-torsion invariants of a curve are defined to be that of its Jacobian. 
For a curve $C$ defined over a finite field $\FF$, its Newton polygon
keeps track of the number of points on $C$ defined over finite extensions of $\FF$.
The Ekedahl--Oort type of $C$ determines, and is determined by,
the structure of the de Rham cohomology of $C$ 
as a module under Frobenius $F$ and Verschiebung $V$, plus the perfect pairing.
We explain some methods to compute the $p$-torsion invariants for Jacobians of curves.
These involve computations of the Frobenius action on various cohomology groups.

Some Newton polygons and Ekedahl--Oort types have been shown to occur for Jacobians and 
a few Ekedahl--Oort types have been ruled out. 
Historically, some of the first results in this area were, for every prime $p$:
(i) for every $g\geq 2$, there exists an ordinary smooth curve of genus $g$ \cite{miller},
and a non-ordinary smooth curve of genus $g$, 
and (ii) there exists a supersingular curve of genus $2$ \cite{serre82}, or \cite{iko}.

We also describe existence results for Newton polygons and Ekedahl--Oort types for curves of small genus $g=2,3,4$.
In particular, Kudo, Harashita, and Senda proved there exists a supersingular smooth curve of genus $4$ for all primes $p$ \cite{khs20}. 
We describe recent progress for $g=5$. 
Much more is known for Artin--Schreier curves;
we cover some of the (cohomological and computational) results for Artin--Schreier curves.

We highlight Question~\ref{Qopenss} about whether there exists a supersingular smooth curve 
for every genus $g$ and characteristic $p$.  There are also open questions about the number of 
curves over finite fields with specified $p$-torsion invariants.

\subsection{Chapter~\ref{C2Aboundary}: Moduli spaces and the boundary}

For $g \geq 1$, we give some background on the moduli space $\cA_g$ of p.p.\ abelian varieties of dimension $g$, and
the moduli space $\cM_g$ of smooth curves of genus $g$.

The Torelli morphism $\tau_g: \cM_g \to \cA_g$ takes a curve $C$ to its Jacobian $\mathrm{Jac}(C)$.
It is injective on $k$-points, meaning that a curve $C$ over $k$ is uniquely determined by $\mathrm{Jac}(C)$.
The \emph{open Torelli locus} $\cT^\circ_g$ is the image of $\tau_g$; 
its points represent p.p.\ abelian varieties of dimension $g$ that are Jacobians of smooth curves.

A primary goal is to characterize the open Torelli locus $\cT^\circ_g$ within $\cA_g$.
For $g \geq 2$, the dimension of $\cM_g$ is $3g-3$, while the dimension of $\cA_g$ is $g(g+1)/2$.
When $g=1,2,3$, then $\cT^\circ_g$ is open and dense in $\cA_g$, 
reaffirming that almost every p.p.\ abelian variety of dimension $g \leq 3$ is a Jacobian of a smooth curve.  
For $g \geq 4$, the codimension of $\cT^\circ_g$ in $\cA_g$ is positive, and 
this codimension increases with $g$. 

Surprisingly, some facts about smooth curves can be proven using singular curves;
some facts about p.p.\ abelian varieties that are indecomposable can be proven using 
p.p.\ abelian varieties that decompose.
For this reason, it is valuable to consider compactifications of moduli spaces, namely 
the Deligne--Mumford compactification $\overline{\cM}_g$ of $\cM_g$ 
and a toroidal compactification $\tilde{\cA}_g$ of $\cA_g$.
The points of the boundary of $\cA_g$ represent semi-abelian varieties having positive toric rank.
The points of the boundary of $\cM_g$ represent stable singular curves.

A stable curve has \emph{compact type} when its irreducible components are smooth and the dual graph
of the irreducible components is a tree.
The Jacobian of a singular curve of compact type is an abelian variety, although it does decompose 
together with the product polarization.

The \emph{closed Torelli locus} $\cT_g$ is the closure of $\cT^\circ_g$ in $\cA_g$;
its points represent p.p.\ abelian varieties of dimension $g$ that are Jacobians of stable curves of compact type. 
The moduli space $\cM_g^{ct}$ represents stable curves of genus $g$ of compact type.
The Torelli map extends to a map $\tau_g: \cM^{ct}_g \to \cA_g$, whose image is $\cT_g$.
However, $\tau_g$ is not injective on the boundary components of $\cM_g^{ct}$.

This chapter ends with results and open questions about 
complete subspaces of these moduli spaces.
For example, a result of Diaz \cite{diaz84}, reproved by Looijenga \cite{looijenga} in positive characteristic, 
states that a complete subspace $W \subset \overline{\cM}_g$ having codimension 
at most $g$ must intersect the boundary.

\subsection{Chapter~\ref{C2Bboundary}: Intersection of the Torelli locus with $p$-torsion strata} 

Suppose $\mathrm{char}(k)=p$.
We provide some background on 
the stratifications of $\cA_g$ by the $p$-torsion invariants.
These stratifications are well understood, in large part 
because of work of Katz \cite{katzslope} and the purity theorem of de Jong and Oort \cite{oortpurity}.
Many of the open questions in this manuscript for positive characteristic 
can be rephrased as questions about the intersection of the $p$-torsion strata of $\cA_g$ with 
the open Torelli locus of Jacobians of curves. 

More precisely, given a $p$-torsion invariant $\xi$, let $\cA_g[\xi] \subset \cA_g$ denote the stratum of p.p.\ 
abelian varieties that have $p$-torsion invariant $\xi$, with the reduced induced structure.
In most cases, it is not understood how $\cA_g[\xi]$ intersects $\cT^\circ_g$, 
or even whether they intersect at all.  If the intersection is non-trivial, one can ask more refined questions about the geometry 
of the intersection, starting with its dimension, as in Question~\ref{Qmot3}.

Consider the supersingular locus $\cA_g[\sigma_g]$.
Using the geometry of $\cA_3[\sigma_3]$, 
Oort proved that there exists a supersingular curve of genus $3$ for every prime $p$ \cite{O:hypsup}.
For $g=4$, we describe a geometric proof that $\cA_4[\sigma_4] \cap \cT^\circ_4$ is non-empty \cite{Pssg=4};
as an application, for every prime $p$,
this proves that there exists a supersingular curve of genus $4$. 
This proof does not extend to curves of higher genus.

In this chapter, we also
review a theorem of Faber and Van der Geer about the dimensions of the $p$-rank strata in $\cM_g$ \cite{FVdG}. 
Furthermore, we discuss an inductive technique which sometimes yields new results about
the existence of curves with given Newton polygon or Ekedahl--Oort type \cite{priesCurrent}. 

We end the chapter with some open questions about the number of irreducible components of the $p$-rank strata in $\cM_g$
and mass formulas for non-ordinary curves.

\subsection{Chapter~\ref{C4Aspecial}: Curves and abelian varieties with cyclic actions}

In this chapter, we discuss curves and abelian varieties that have an action by a non-trivial 
cyclic group.  Specifically, we consider cyclic degree $m$ covers of the projective line, which are classified by
discrete invariants, namely the degree, the number of branch points, and the inertia type.
The Jacobians of these curves are p.p.\ abelian varieties with an action of $\mu_m$ of a given signature, which is the information
of the dimensions for the eigenspaces of the $\mu_m$-action on the regular $1$-forms.  

Families of covers of this type are represented by moduli spaces called Hurwitz spaces.
These Hurwitz spaces map to Deligne--Mostow Shimura varieties, whose points represent p.p.\ abelian varieties having 
a $\mu_m$-action with this signature.
A family of covers is \emph{special} when the image of the Hurwitz space is open and dense in (a component of) the Shimura variety.

The Legendre family of elliptic curves is one example of a special family, and 
the moduli space of hyperelliptic curves of genus $2$ is another.
The complete list
of families of cyclic covers of $\mathbb{P}^1$ that have special moduli spaces was found by Moonen \cite{moonen}.  
Curves in special families are often easier to analyze because of the additional techniques available on the 
Shimura variety.

We end the chapter with open questions about the Coleman--Oort conjecture.

\subsection{Chapter~\ref{Chapter4BNPspecial}: Newton polygons and cyclic actions}

Let $\mathrm{char}(k)=p$.
In this chapter, we discuss the Newton polygons of abelian varieties (and curves) 
having a cyclic action as in Chapter~\ref{C4Aspecial}.  

Historically, this material generalizes earlier results for Jacobians that have complex multiplication.
Many people computed Newton polygons of Jacobians of (quotients of) Fermat curves.
For these Jacobians, there are several methods to determine the Newton polygon, including computing the action of 
Frobenius or using the Shimura--Taniyama formula.
 
Generalizing the complex multiplication situation, we discuss the Newton polygons of abelian varieties (and curves) 
having a cyclic action as in Chapter~\ref{C4Aspecial}.  There are restrictions on which Newton polygons occur in this 
context which were studied by Kottwitz and others.  The first of these arises from the action of Frobenius on the 
eigenspaces for the $\mu_m$-action.  

The Newton polygon stratification is easier to study on Hurwitz spaces which are special.
As an application, we show there exist supersingular curves of genus $5$, $6$, and $7$, 
under certain (new) congruence conditions on the prime $p$ \cite{LMPT2}.  
 
We end the chapter with some questions that are likely more tractable.
Specifically, these are questions about supersingular curves in special families $\mathcal{F}$ 
of cyclic covers of $\mathbb{P}^1$, 
including the field of definition of the supersingular curves when $\mathrm{dim}(\mathcal{F}) =1$
and the geometry of the supersingular locus when $\mathrm{dim}(\mathcal{F}) =2$.

\subsection{Chapter~\ref{C3Acovers}: Torsion points and monodromy}

This chapter is about the $\ell$-adic monodromy groups of families of abelian varieties and curves, 
for a prime $\ell \not = \mathrm{char}(k)$.
For a curve $C$ of genus $g$, there is a connection 
between the $\ell^{2g}$ points on $\mathrm{Jac}(C)$ that are $\ell$-torsion points
and the unramified $\ZZ/\ell \ZZ$-covers of $C$.  

Consider a family of abelian varieties $X \to W$ of (relative) dimension $g$, with $W$ irreducible. 
Let $w \in W$ be a chosen basepoint. 
Then the fundamental group $\pi_1(W,w)$ of the base acts on 
the $\ell$-torsion of the abelian variety $X_w$, the fiber of $X \to W$ over $w$.  The $\ell$-adic monodromy group measures the image of 
the associated representation of $\pi_1(W,w)$ in the symplectic group $\mathrm{Sp}_{2g}(\ZZ_\ell)$.

In this chapter, we highlight a result of Chai \cite{Chaimonodromy}:
for a prime $\ell \not = p$, it states that the $\ell$-adic monodromy group is big if $W \subset \cA_g$
is an irreducible subspace stable under all $\ell$-adic Hecke correspondences such that $W$ is not contained in the supersingular locus.
We discuss a related result about the $\ell$-adic monodromy groups 
of the $p$-rank strata of $\cM_g$ \cite{AP:mono}.  
This proves that the $p$-torsion and the $\ell$-torsion on Jacobians are independent of each other,
in a way that can be made precise and which has applications for the the generic behavior of curves with a given $p$-rank.

We end the chapter with some open questions about the $a$-number of a generic curve of $p$-rank $0$ and about the
monodromy of the stratum of curves with $p$-rank $0$ and $a$-number $a \geq 2$.

\section{Lectures at the Arizona Winter School}

Here is a brief description of my lectures at the Arizona Winter School in March 2024.
The first half of each lecture included material that applies for abelian varieties over a field $k$ of any characteristic;
the second half focused on material about abelian varieties over a field $k$ of positive characteristic.
Both halves of each lecture included an open question.  
We omit the references in this section; see the relevant parts of Section~\ref{Stopics} for references.

\begin{enumerate}
\item {\bf The Torelli locus and arithmetic invariants}

In the first half of this lecture, I covered parts of Chapters~\ref{C1Atorelli} and \ref{C2Aboundary}, 
focusing on the Torelli locus.
With a dimension count, I showed that the Torelli locus is open and dense inside $\cA_g$ when $1 \leq g \leq 3$, and has 
positive codimension for $g \geq 4$.
In the second half, I described the $p$-torsion invariants of abelian varieties from Chapter~\ref{Chapter3},
and results about $p$-torsion invariants of curves of small genus from Chapter~\ref{Chapter4}.

\bigskip

\item {\bf Boundaries of moduli spaces of curves and abelian varieties}

In the first half of this lecture, I covered material from Chapter~\ref{C2Aboundary}
about the boundary of $\cM_g$ and the clutching morphisms.
I discussed the result of Diaz that a complete subspace 
$W \subset \overline{\cM}_g$ having codimension at most $g$ must intersect the boundary.

In the second half, I described material from Chapter~\ref{C2Bboundary}, 
starting with the purity result for the Newton polygon stratification of $\cA_g$.
As an application, I explained a geometric proof that, for every prime $p$, 
there exist supersingular curves of genus $4$, and explained why 
this proof does not extend to curves of higher genus.
I also described how the boundary can be used to study the $p$-rank stratification of $\cM_g$.  

\bigskip

\item {\bf Special families of abelian varieties}

In this lecture, I covered parts of Chapter~\ref{C4Aspecial} and Chapter~\ref{Chapter4BNPspecial}.
The main themes of the first half were (moduli spaces for) abelian varieties and curves having an action by a cyclic group. 
I highlighted the significance of special families and the Coleman--Oort conjecture.

In the second half, I provided examples of the restrictions on Newton polygons for special families
of cyclic covers of the projective line.
As an application, I explained how this
yields examples of supersingular curves of genus $5,6,7$ under congruence conditions on the prime $p$.

\bigskip

\item {\bf Torsion points and unramified covers}

In this lecture, I covered Chapter~\ref{C3Acovers}
about $\ell$-adic monodromy of families of abelian varieties or curves.
I described Chai's result about big monodromy and how this can be used to prove a big 
monodromy result for the $p$-rank strata of the moduli space of curves.
I ended by analyzing differences in the generic behavior of curves with $p$-rank $0$ and curves that are supersingular.

\end{enumerate}

\chapter{Abelian varieties and Jacobians of curves} \label{C1Atorelli}

The purpose of this chapter is to provide some brief background and intuition about abelian varieties.
We refer to the Prelude of this volume for definitions. 
For more comprehensive information, please read the many excellent textbooks about abelian varieties, 
including Mumford \cite{mumfordAVbook} and Milne \cite{milnechapterAV}. 
See also the notes by \cite{milneAV} and \cite{EvdGMbook}.
Moduli spaces of abelian varieties are covered by Mumford, Fogarty, and Kirwan in \cite{mumfordfogarty}, and by Faltings and Chai in \cite{faltingschaidegeneration}.

Throughout, we work over an algebraically closed field $k$.

\begin{definition} (Prelude Definition 1.3.2) 
An \emph{abelian variety} over $k$ is a smooth connected proper group variety over $k$.
\end{definition}

An equivalent definition from \cite[page 39]{mumfordAVbook} is that an abelian variety is a complete (reduced, irreducible) algebraic variety, 
with a group law whose multiplication and inverse maps are both morphisms of varieties.

\begin{remark} Suppose $X$ is an abelian variety over $k$. 
See the following references for these definitions: group variety, \cite[Section~1]{milnechapterAV};

isogeny, Prelude Section~1.6 or \cite[Section~8]{milnechapterAV};

Picard group of line bundles on $X$,
and more generally the Picard functor of invertible sheaves, Prelude Section~1.6, or \cite{kleimanPicard};

the dual abelian variety $X^*$, Prelude Section~1.7, or \cite[Sections~9,10]{milnechapterAV};

and (principal) polarization on $X$, Prelude Section~1.8 or \cite[Section~13]{milnechapterAV}.
\end{remark}

\section{Overview: Jacobians provide good examples of abelian varieties}

There are some situations where abelian varieties can be constructed in a straight-forward way.
For example, in Section~\ref{Sellipticcurve}, we give a brief description of elliptic curves, which
are abelian varieties of dimension $1$.
In Section~\ref{Scomplexsituation}, we give a brief description of
complex abelian varieties through the lens of complex tori.

However, it can be complicated to study abelian varieties in many situations.
Jacobians of curves provide examples of abelian varieties which are often easier to analyze.
We provide some background on curves in Section~\ref{Sbackgroundcurves}, with
numerous examples included in Section~\ref{Sexamplecurve2}.
Then we define the Jacobian of a (smooth, projective, irreducible) curve in Section~\ref{SJacobianabelvar}.
The main result is Theorem~\ref{Ttorelli0}, which states that 
the Jacobian of a curve of genus $g$ is a principally polarized abelian variety of dimension $g$.
This material is the foundation for the Torelli map. 

As mentioned before, for $g \geq 4$, most abelian varieties of dimension $g$ are not Jacobians of curves.
This leads to a fruitful exchange between the topics of abelian varieties and curves, where 
techniques available for one topic are leveraged to study the other.

\section{Elliptic curves} \label{Sellipticcurve}

We keep this section brief because there are excellent references about elliptic curves, including 
textbooks by Husem\"oller \cite{Husemoller} and Silverman \cite{aec}.

An elliptic curve $E$ is a (smooth, projective, irreducible) curve of genus one, together with a chosen point $\mathcal{O}_E$.
It follows that $E$ is an abelian group, with identity $\mathcal{O}_E$, 
where the addition of points is defined algebraically.
Every elliptic curve can be embedded as a smooth cubic in $\mathbb{P}^2$.

\begin{example} \label{Eelliptic}
Suppose $p = \mathrm{char}(k) \not = 2,3$.
After changing coordinates, $E$ can be described as
the projective curve in $\PP^2$ given by the vanishing of the homogenous equation
$Y^2Z = X^3 + aXZ^2 + bZ^3$, for some $a,b \in k$ such that $4a^3+27b^2 \not = 0$, 
and $\mathcal{O}_E$ is the point $[X:Y:Z]=[0:1:0]$.  Given a line in $\PP^2$, it intersects $E$ in three points (counting with multiplicity)
and these three points sum to the identity in the group law of $E$.
More simply, we can describe $E$ by the affine equation $y^2=f(x)$ where
$f(x) = x^3 + ax +b$ for some $a,b \in k$ such that $f(x)$ has distinct roots.
\end{example}

Over $\CC$, an elliptic curve is a compact complex torus of dimension $1$.

\begin{example} \label{Ecomplextorusdim1}
A compact complex torus of dimension $1$ is isomorphic to $\CC/\Lambda$ where $\Lambda$ is a lattice.
After adjusting by the action of $\CC^*$, we can suppose $\Lambda$ is generated by $1$ and $\tau$, 
where $\tau$ is in the upper half plane $\mathfrak{h}$.
\end{example}

The $j$-invariant of $E: y^2=x^3+ax+b$ is $j=1728 \cdot 4a^3/(4a^3+27b^2)$.
The $j$-invariant $j(\tau)$ of $E = \CC/\Lambda$ is a modular function of weight $0$ for the special linear group 
$\mathrm{SL}_2(\ZZ)$.

\begin{proposition}
Two elliptic curves over $k$ are isomorphic if and only if they have the same $j$-invariant.
\end{proposition}

Given two lattices $\Lambda_i$ generated by $1$ and $\tau_i$ with $\tau_i \in \mathfrak{h}$ for $i=1,2$, 
the two complex tori $\CC/\Lambda_1$ and $\CC/\Lambda_2$ have
$j(\tau_1) = j(\tau_2)$ exactly when
$\tau_1$ and $\tau_2$ are in the same orbit of the action of $\mathrm{SL}_2(\ZZ)$ 
on $\mathfrak{h}$ by fractional linear transformations.

The $j$-invariant can be viewed as the parameter for the moduli space $\cM_{1,1}$ of elliptic curves, or 
the moduli space $\cA_1$ of p.p.\ abelian varieties of dimension $1$, both of which have dimension $1$.
The book by Katz and Mazur \cite{KatzMazur} provides more information about
moduli spaces of elliptic curves.

\begin{remark}
Elliptic curves are abelian varieties of dimension $1$.
Substantially more work is required to study abelian varieties of the next smallest dimension, namely dimension $2$.
From an algebraic perspective, one can consider Jacobians of curves of genus $2$.  
These curves are hyperelliptic;
for $p \not = 2$, each has an affine equation of the form $y^2=f(x)$, where $f(x) \in k[x]$ is a separable polynomial
of degree $5$ or $6$. 
The isomorphism class of a genus two curve is determined by its three Igusa invariants. 
However, it takes some additional work to construct the Jacobian.
Over $\CC$, an abelian surface can be constructed as $\CC^2/\Lambda$, where the lattice $\Lambda$ has 
additional constraints coming from the principal polarization.
\end{remark}

\section{Complex abelian varieties} \label{Scomplexsituation}

In this section, we provide the complex analytic viewpoint on abelian varieties.  
We keep this section brief because there are excellent references about abelian varieties over $\CC$, including 
Swinnerton-Dyer \cite{swinnertondyer}, 
Birkenhake--Lange \cite{birkenhakelange},
Debarre \cite{Debarre}, and Lange \cite{lange}. 

In this section, we work over $k =\CC$.
We denote complex conjugation with an overline.
Let $g \geq 1$ be an integer.  

\subsection{Complex tori}

Suppose $X=V/\Lambda$ is a compact complex torus, 
where $V$ is a complex vector space of dimension $g$ and $\Lambda$ is a lattice in $V$.
Then $X$ is a complex analytic space of dimension $g$. 

Consider a $\ZZ$-basis $\lambda_1, \ldots, \lambda_{2g}$ for $\Lambda$ and a basis 
$e_1, \ldots, e_g$ for $V$.  Writing the former in terms of the latter gives a $g \times 2g$-matrix $\Pi$ called the 
\emph{period matrix}.  

\begin{proposition} \cite[Proposition~1.1.2]{birkenhakelange}
A $g \times 2g$-matrix $\Pi$ is the period matrix of a 
complex torus if and only if the $2g \times 2g$-matrix $\left(\begin{array}{cc}
\Pi \\
\overline{\Pi}
\end{array}\right)$ is invertible.
\end{proposition}

\subsection{Complex abelian varieties}

A complex torus is an abelian variety if and only if it is an algebraic variety.  
This is equivalent to the condition that it has an ample line bundle $\mathcal{L}$,
or equivalently that it can be embedded in projective space.
This condition can be described in several different ways. 
First, here are the Riemann relations.

\begin{theorem} \cite[Theorem~4.2.1]{birkenhakelange} \label{TRiemannrelations}
The complex torus $\CC^g/ \Pi \ZZ^{2g}$ is an abelian variety if and only if
there exists a non-degenerate $2g \times 2g$ alternating matrix $A$ such that the following
\emph{Riemann relations} are true:

(i) $\Pi (A^{-1})^T\Pi = 0$; and 
(ii) $i \Pi (A^{-1})^T\overline{\Pi} >0$.
\end{theorem}

The second description involves Hermitian forms.
A \emph{Hermitian form} on $V$ is a map $H:V \times V \to \CC$ which is $\CC$-linear in the first argument
and such that $H(v, w) = \overline{H(w,v)}$ for all $v,w \in V$.
A Hermitian form is \emph{positive semi-definite} if $H(v, v) \geq 0$ for all $v \in V$;
it is \emph{positive definite} if it is positive semi-definite and $H(v, v) = 0$ if and only if $v = 0$;
it is \emph{non-degenerate} if the condition $H(u,v)=0$ for all $v \in V$ implies $u=0$.

\begin{example}
Continuing with Example~\ref{Ecomplextorusdim1}, for a compact complex torus of dimension one,
the positive definite Hermitian form $H: \CC \times \CC \to \CC$ is given by $H(v,w) = v \cdot \overline{w}/\mathrm{Im}(\tau)$.
\end{example}

\begin{definition}
A \emph{Riemann form} on $X=V/\Lambda$ is a positive definite non-degenerate Hermitian form $H$ on $V$ 
such that the restriction of the imaginary part
$E := \mathrm{Im}(H)$ to $\Lambda$ is integer valued. 
\end{definition}

In fact, $E$ is the first Chern class of $\mathcal{L}$.

\begin{theorem}
A complex torus is isomorphic to an abelian variety $X$ over $\CC$
if and only if it has a Riemann form.
\end{theorem}

We refer to \cite[Section~2.4, Theorem~2.5.5]{birkenhakelange} for a third description in terms of a principal polarization.

\subsection{The Siegel upper half-space} \label{Sppavsymp}

We follow \cite[Chapter~8]{birkenhakelange}.

Suppose $X = V/\Lambda$ is a p.p.\ abelian variety of dimension $g$.  
Let $H$ be a Hermitian form defining its principal polarization.  
Let $\lambda_1, \ldots, \lambda_g, \mu_1, \ldots, \mu_g$ be a symplectic basis of $\Lambda$ for 
the alternating form $E := \mathrm{Im}(H)$; 
this means that $E(\lambda_i, \mu_j) = \delta_{i,j}$, the Kronecker Delta symbol.
With respect to this basis, $E$ is given by the matrix 
$\left(\begin{array}{cc}
0 & I_g \\
-I_g & 0 
\end{array}\right)$.
The vectors $\mu_1, \ldots, \mu_g$ form a $\CC$-basis for $V$.
The period matrix is given by $\Pi=(Z, I_g)$ for some $g \times g$ matrix $Z$.

\begin{proposition} \label{Pppavsymp}
\cite[Proposition~8.1.1]{birkenhakelange}
(a) ${}^TZ =Z$ and $\mathrm{Im}(Z) >0$; and
(b) $(\mathrm{Im}(Z))^{-1}$ is the matrix of $H$ with respect to the basis $\mu_1, \ldots, \mu_g$.
\end{proposition}

Note that the matrix $Z$ defining $X$ is determined by $g(g+1)/2$ of its entries.

\begin{definition} \label{Dsiegel}
The \emph{Siegel upper half-space} $\mathfrak{h}_g$
is the set of $g \times g$ complex-valued matrices satisfying 
${}^TZ =Z$ and $\mathrm{Im}(Z) >0$.
\end{definition}

We see that $\mathfrak{h}_g$ has dimension $g(g+1)/2$ because it is an open submanifold of 
the vector space of symmetric $g \times g$ matrices.
By \cite[Proposition~8.1.2]{birkenhakelange}, $\mathfrak{h}_g$ is the moduli 
space for principally polarized complex abelian varieties with symplectic basis.
This shows the following.

\begin{theorem} \label{TdimensionAg} \cite[Theorem~8.2.6]{birkenhakelange}
The moduli space $\cA_{g,\CC}$ of principally polarized complex abelian varieties of dimension $g$ 
is the quotient of $\mathfrak{h}_g$ by the sympletic group $\mathrm{Sp}_{2g}(\ZZ)$.
Thus it is irreducible and has dimension $g(g+1)/2$. 
\end{theorem}

\subsection{Jacobian of a complex curve} \label{SccJac}

Suppose $C$ is a (smooth, projective, irreducible) curve over $\CC$. 
We define the Jacobian of $C$, following \cite[Section~1.3]{ACGH1} and \cite[Chapter VIII]{miranda}.  

Let $H^0(C, \Omega^1)$ denote the vector space of holomorphic $1$-forms on $C$.
A \emph{linear functional} on $C$ is an element of the dual space $H^0(C, \Omega^1)^*$,
namely a linear transformation $H^0(C, \Omega^1) \to \CC$.
Loops $c$ in $C$ can be represented by homology classes.
The homology group $H_1(C, \ZZ)$ is a free abelian group of rank $2g$. 
Every homology class $[c]$ defines a linear functional
$\int_{[c]}: H^0(C, \Omega^1) \to \CC$, which takes a regular $1$-form $\omega$ to its integral over $c$.
The linear functionals that occur in this way are called \emph{periods}.
The set $\Lambda$ of periods is a subgroup of $H^0(C, \Omega^1)^*$.

\begin{definition}
The \emph{analytic Jacobian} of $C$ is $\mathrm{Jac}(C) = H^0(C, \Omega^1)^*/\Lambda$.
\end{definition}

After choosing a basis for $H^0(C, \Omega^1)^*$, there is an isomorphism $\mathrm{Jac}(C) \cong \CC^g/\Lambda$, 
which is a complex torus of dimension $g$.
One can show that the periods satisfy the Riemann relations.
The curve $C$ determines a principal polarization on $\mathrm{Jac}(C)$.
Thus $\mathrm{Jac}(C)$ is a principally polarized abelian variety.

The Abel--Jacobi map:
Choose a base point $p_\circ$ on $C$.  For each point $x \in C$, 
choose a path $\gamma_x$ from $p_\circ$ to $x$.
This is possible because $C$ is connected.
There is a map $C \to H^0(C, \Omega^1)^*$, sending $x$ to the linear functional $\int_{\gamma_x}$
of integration along $\gamma_x$.
This map is not well-defined because different paths from $p_\circ$ to $x$ may not be homotopic.  
Taking the quotient by the periods $\Lambda$ yields a well-defined map, 
still depending on the base point $p_\circ$, called the Abel--Jacobi map: $j: C \to \mathrm{Jac}(C)$.

\section{Curves and covers of curves} \label{Sbackgroundcurves}

More information about (moduli spaces of) curves can be found in these textbooks:
Mumford \cite{mumfordbook}; 
Arbarello, Cornalba, Griffiths, \& Harris \cite{ACGH1};
Bosch, L{\"u}tkebohmert, \& Raynaud \cite{BLR};
Miranda \cite{miranda}; 
Harris \& Morrison \cite{harrismorrison}; 
and Arbarello, Cornalba, \& Griffiths \cite{ACGH2}.

Recall that $k$ is an algebraically closed field.

\subsection{Curves and the genus of curves} 

\begin{definition}
A \emph{curve} $C$ over $k$ is a smooth, projective, irreducible variety of dimension $1$ over $k$.
\end{definition}

A $1$-form $\omega$ on $C$ is a section of the cotangent bundle.
The $1$-form is \emph{regular} if it has no poles.
Let $\Omega^1$ denote the sheaf of $1$-forms on $C$.

\begin{definition} \label{Dgenus}
Let $H^0(C, \Omega^1)$ denote the vector space of regular $1$-forms.
The \emph{genus} $g$ of $C$ is the dimension of $H^0(C, \Omega^1)$.
\end{definition}

The projective line $\mathbb{P}^1$ is the unique curve of genus $0$ over $k$.

For a local description of a $1$-form $\omega$ near a point $P$, we consider a uniformizing parameter $t$ at $P$, namely 
a function on an affine subset $U$ of $C$ 
containing $P$ such that $t$ vanishes with order $1$ at $P$.  Then $\omega$ has an expression of the 
form $f(t) dt$ where $f(t)$ is a rational function on $U$. 
The order of the zero (or pole) of $\omega$ at $P$ is the order of the zero (or pole) of $f(t)$ at $P$.

\begin{example}
Consider the $1$-form $\omega = dx$ on $\mathbb{P}^1$.  This is regular away from $\infty$, but has a pole of order $2$ at $\infty$.
To see this, write $\overline{x} = 1/x$ and compute $\omega = (-1/\overline{x}^2) d\overline{x}$.
There are no non-zero regular $1$-forms on $\mathbb{P}^1$.  
\end{example}

\begin{example}
For the elliptic curve $y^2 = x^3 + ax +b$ from Example~\ref{Eelliptic}, 
the $1$-form $dx/y$ is regular.
\end{example}

\subsection{Equations for curves}

The easiest way to describe a curve of positive genus is with an affine equation. 
Frequently, we consider an affine curve $C' \subset {\mathbb A}^2$
given by the vanishing of a polynomial equation $h(x,y) =0$.
It is no loss of generality to work with affine curves because of this fact:

\begin{fact} \label{Fprojective}
The smooth completion of a smooth affine algebraic curve $C'$ is a complete 
smooth algebraic curve $C$ which contains $C'$ as an open subset.
Smooth completions exist and are unique over $k$.
\end{fact}

Every (smooth, projective, irreducible) curve $C$ can be embedded in $\mathbb{P}^3$.
Sometimes $C$ can be embedded in $\PP^2$ and sometimes not. 
If not, it is often a hassle to find the blow-ups that resolve the singularities of a planar model of $C$.  
In light of Fact~\ref{Fprojective}, we usually work with one affine equation for a curve.

\begin{example} 
Let $C'$ be the curve with affine equation $y^2 = x^5 - 2x$ (here $p \not = 2,5$). 
The homogenization $y^2z^3 = x^5 - 2xz^4$ yields a curve in $\mathbb{P}^2$ 
that has a singularity when $[x:y:z]=[0:1:0]$.
The smooth completion of $C'$ has another affine patch $C''$ where this singularity is resolved.
To find it, we define $\bar{x} = 1/x$ and $\bar{y} = y \bar{x}^3$.
Then $C''$ is given by the affine equation $\bar{y}^2 = \bar{x} - 2 \bar{x}^5$.
The point $(\bar{x}, \bar{y}) = (0, 0)$ is colloquially called the point at infinity on $C'$.
\end{example}

\subsection{Covers of curves}

If $C, D$ are curves, a \emph{branched cover} $\pi:C \to D$ is a finite surjective morphism,
which is generically unramified.  The branch locus is the finite set of points $B \subset D$ over which $\pi$ is ramified.
We frequently drop the word branched from the terminology.

The morphism $\pi$ corresponds to an inclusion of the function field $k(D)$ in the function field $k(C)$.
The \emph{degree} of $\pi$ is the dimension of $k(C)$ as a vector space over $k(D)$.
The cover $\pi$ is Galois if the finite extension $k(C)/k(D)$ is Galois and is \emph{cyclic} if it is Galois with cyclic Galois group.

For $\eta \in C$, let $R_{C, \eta}$ be the complete local ring of $C$ at $\eta$, 
and let $S_{D, \pi(\eta)}$ be the complete local ring of $D$ at $\pi(\eta)$.
Let $t_1$ (resp.\ $t_2$) be a uniformizing parameter of $R_{C, \eta}$ at $\eta$ (resp.\ $S_{D, \pi(\eta)}$ at $\pi(\eta)$).
Let $e_\eta$ denote the ramification index of $\pi$ at $\eta$.
This is the positive integer such that $t_2 = u t_1^{e_\eta}$ for some $u \in R_{C, \eta}$.
If $\mathrm{char}(k)=p >0$, the cover is \emph{tamely ramified} if $p \nmid e_\eta$ for all $\eta \in C$, 
and is \emph{wildly ramified} otherwise.

The Riemann--Hurwitz formula provides a good way to compute the genus.

\begin{theorem} \label{Triemannhurwitz} (Riemann--Hurwitz formula)
Suppose $\pi: C \to D$ is a tamely ramified cover of curves of degree $d$.  
Then the genus $g_C$ of $C$ and the genus $g_D$ of $D$ are related by the formula:
\[2g_C -2 = d(2g_D-2) + \sum_{\eta \in C} (e_\eta -1).\]
\end{theorem}

Finding the orders of poles of a $1$-form is a delicate process.
For a cover of curves, the following lemma is useful.

\begin{lemma} \label{Lpullback1form}
Suppose $\pi: C \to D$ is a tamely ramified cover of curves.
If $\omega$ is a regular $1$-form on $D$, then the pullback $\pi^*\omega$ is a regular $1$-form on $C$.
More specifically, if $\eta \in C$, then
$\mathrm{ord}_\eta(\pi^* \omega) = (1 + \mathrm{ord}_{\pi(\eta)}(\omega)) e_\eta -1$.
\end{lemma}

\begin{proof}
By definition of the ramification index, $t_2 = u t_1^{e_\eta}$, 
where $t_1$ (resp.\ $t_2$) is a uniformizing parameter of $C$ at $\eta$ (resp.\ $D$ at $\pi(\eta)$) as defined above.
Write $n = \mathrm{ord}_{\pi(\eta)}(\omega)$.  
So $\omega = u' t_2^n dt_2$, where $u'$ is a unit in the complete local ring $S_{D, \pi(\eta)}$.
So $\pi^* \omega = u'  (ut_1^{e_\eta})^n e_\eta u t_1^{e_\eta -1}$, whose order of vanishing at $\eta$ is $ne_\eta + e_\eta -1$.
\end{proof}  

\subsection{Examples of branched covers of curves} \label{Sexamplecurve2}

\begin{definition}
A \emph{hyperelliptic curve} is a curve $C$ that admits a degree two branched cover $\pi: C \to \mathbb{P}^1$.
\end{definition}

\begin{fact}
If $\mathrm{char}(k) \not = 2$, a hyperelliptic curve has an affine equation $y^2 = f(x)$ for some separable polynomial $f(x) \in k[x]$.
The hyperelliptic involution $\iota$ acts by $\iota((x,y)) = (x, -y)$.
There is a unique hyperelliptic involution on a hyperelliptic curve $C$ and it is contained in the center of the 
automorphism group of $C$.
\end{fact}

The next example can be checked using Theorem~\ref{Triemannhurwitz} and Lemma~\ref{Lpullback1form}.

\begin{lemma}
Suppose $\mathrm{char}(k)\not =2$.
Suppose $f(x) \in k[x]$ is a separable polynomial of degree $2g+1$ or $2g+2$.  
The hyperelliptic curve $C$ with affine equation $y^2=f(x)$ has genus $g$. 
A basis for $H^0(C, \Omega^1)$ is given by $\{dx/y, xdx/y, \ldots, x^{g-1}dx/y\}$.
\end{lemma}

The next example arises from Kummer theory.

\begin{definition}
A \emph{superelliptic curve} is a curve $C$ that admits a cyclic cover $\pi: C \to \mathbb{P}^1$ (with degree $m \geq 2$).
\end{definition}

\begin{fact} \label{Factsup} If $\mathrm{char}(k)$ does not divide the degree $m$ of $\pi$, then
the superelliptic curve $C$ has
an affine equation $y^m = \prod_{i = 1}^N(x-b_i)^{a_i}$, for some integer $N \geq 2$, 
some set of distinct values $\{b_1, \ldots, b_N\}$ in $k$,
and some \emph{inertia type} which is a tuple $\vec{a}$ of integers with 
\begin{equation} \label{Einertiatype}
\vec{a}=(a_1, \ldots, a_N), \ 1 \leq a_i \leq m-1, \text{ and } \sum_{i=1}^N a_i \equiv 0 \bmod m.
\end{equation}

In this context, $\{b_1, \ldots, b_N\}$ is the branch locus of $\pi : C \to \mathbb{P}^1$.
If we prefer $\infty$ to be one of the branch points (say the last one), we remove the last term $(x-b_N)^{a_N}$
from the equation.

The $\mu_m$-action on $C$ is given by $\phi((x,y))=(x, \zeta y)$ for $\zeta \in \mu_m$.
\end{fact}

\begin{notation} \label{Nsuperelliptic}
To classify superelliptic curves $\pi: C \to \mathbb{P}^1$, 
we record the discrete data $\gamma = (m, N, \vec{a})$, where
$\pi$ has degree $m$ and $N$ branch points, and 
the inertia type $\vec{a}$ is the tuple satisfying \eqref{Einertiatype}.
\end{notation}

\begin{lemma}
With notation as in Fact~\ref{Factsup}:
Above the point $x=b_i$, the curve $C$ has $g_i=\mathrm{gcd}(m, a_i)$ points, each with 
inertia group of order $m/g_i$.
By the Riemann--Hurwitz formula, the genus of $C$ satisfies:
\[2g_C-2 = m(-2) + \sum_{i=1}^N g_i (\frac{m}{g_i} -1).\]
If $g_i=1$ for $1 \leq i \leq N$ (e.g., if $m$ is prime), then
$g_C = (N-2)(m-1)/2$.
\end{lemma}

The next example arises from Artin--Schreier theory.

\begin{definition} \label{Dartinschreier} Suppose $\mathrm{char}(k) =p>0$.
An \emph{Artin--Schreier curve} is a curve $C$ that admits a degree $p$ cyclic cover $\pi: C \to \mathbb{P}^1$. 
\end{definition}

\begin{fact} \label{FactAS}
An Artin--Schreier curve has an affine equation $y^p-y = h$ for some rational function $h \in k(x)$; 
the curve is connected if and only if $h \not = z^p-z$ for any rational function $z \in k(x)$.

Without loss of generality, we can suppose that the order of each pole of $h$ is relatively prime to $p$.
Then the branch locus of $\pi$ is the set of poles of $h$ and $\pi$ is wildly ramified above each of these points.

The $\ZZ/p\ZZ$-action on $C$ is generated by the automorphism $\phi((x,y)) = (x, y+1)$.  
\end{fact}

\begin{example}
With notation as in Fact~\ref{FactAS}:
suppose $h \in k[x]$ is a polynomial of degree $j$ with $p \nmid j$.
Then the genus of $C$ is $g=(p-1)(j-1)/2$. 
This can be proven with a wild generalization of the Riemann--Hurwitz formula \cite{Se:lf}.
A basis for $H^0(C, \Omega^1)$ is given by 
\[\{y^r x^b dx \mid 0 \leq r \leq p-2, \ 0 \leq b \leq j-2, \ rj+bp \leq pj -j -p -1\}.\]
\end{example}

\section{The Jacobian of an algebraic curve} \label{SJacobianabelvar}

References for this topic are \cite[Chapter~VII]{milneJacobian} and \cite[Chapter~III]{milneAV}.

Suppose $C$ is a (smooth, projective, irreducible) curve of genus $g \geq 1$ over $k$.
The Jacobian of $C$ is $\mathrm{Pic}^0(C)$, whose points are
the group of isomorphism classes of line bundles (invertible sheaves) on $C$ of degree $0$.

In greater detail, for an irreducible scheme $T$ over $k$, consider 
$\mathrm{Pic}_C^0(T)$ as the group of families of invertible sheaves on
$C$ of degree $0$ parametrized by $T$, modulo trivial families.
Then $\mathrm{Pic}_C^0$ is a functor from schemes over $k$ to abelian groups.
By \cite[Theorem~1.1]{milneJacobian}, there is an abelian variety $J$ representing $\mathrm{Pic}_C^0$, with
$\mathrm{Pic}_C^0(T) \cong J(T)$ whenever $C(T)$ is non-empty.
We call $J$ the Jacobian $\mathrm{Jac}(C)$ of $C$. 
In light of this, and also \cite[Section~III.2]{milneAV}, we identify $\mathrm{Pic}^0(C)$ and $\mathrm{Jac}(C)$
without comment throughout this text.

\begin{theorem} \label{Ttorelli0} \cite[Proposition~2.1]{milneJacobian}, 
The Jacobian $\mathrm{Jac}(C)$ of $C$ is an abelian variety of dimension $g$,
whose tangent space at the origin is isomorphic to $H^1(C, \mathcal{O}_C)$.
\end{theorem}

Choose a $k$-point $P_0$ of $C$.  Let $\mathcal{L}^{P_0}$ be the invertible sheaf on $C\times C$ given by
$\mathcal{L}(\Delta - (C \times \{P_0\}) - (\{P_0\} \times C))$, where
$\Delta \subset C \times C$ is the diagonal. 
For a point $Q \in C$, the restriction of $\mathcal{L}^{P_0}$ to $C \times \{Q\}$ is isomorphic to 
$\mathcal{L}(Q-P_0)$.

By \cite[page~172]{milneJacobian}, the sheaf $\mathcal{L}^{P_0}$ induces a correspondence between 
$(C,P_0)$ and itself.
This yields a map $j: C \to \mathrm{Jac}(C)$, such that a $k$-point $Q \in C$
maps to the element of $\mathrm{Jac}(C)(k)$ identified with $\mathcal{L}(Q) \otimes \mathcal{L}(P_0^{-1})$ in $\mathrm{Pic}^0(C)(k)$.
By \cite[Proposition~2.3]{milneJacobian}, the map $j$ is a closed immersion.

\begin{remark} \label{Rotherview}
Here is another perspective on the Jacobian.
Let $\mathrm{Div}(C)$ denote the group of divisors on $C$, 
namely finite sums of the form $D=\sum_{P \in C} n_P [P]$, where $n_P$ is an integer for each point $P \in C$.
The degree of $D$ is $\sum_{P \in C} n_P$.
Let $\mathrm{Div}^0(C) \subset \mathrm{Div}(C)$ denote the subgroup of divisors of degree $0$.

A divisor $D$ is \emph{principal} if it is the divisor of a rational function $f$ on $C$.
This means that $n_P$ is the order of vanishing of $f$ at $P$ for every point $P \in C$.
The degree of a principal divisor is $0$. 
Let $\mathrm{PDiv}(C) \subset \mathrm{Div}^0(C)$ denote the subgroup of principal divisors.
Then $\mathrm{Jac}(C)(k)$ is isomorphic to $\mathrm{Div}^0(C)/\mathrm{PDiv}(C)$.
The map $j: C \to \mathrm{Jac}(C)$ can
be identified with the map $j: C \to \mathrm{Div}^0(C)/\mathrm{PDiv}(C)$ which takes $Q$ to $[Q-P_0]$, where $[D]$ 
denotes the equivalence class of the divisor $D$.
\end{remark}

For $n \geq 1$, let $C^{(n)}=C^n/S_n$ where $S_n$ is the symmetric group on $n$ letters.
This is a smooth variety over $k$ \cite[Proposition~3.2]{milneJacobian}.  
The points in $C^{(n)}$ represent unordered multisets $\{x_1, \ldots, x_n\}$ of $n$ points of $C$.
By \cite[Theorem~3.13]{milneJacobian}, $C^{(n)}$ represents $\mathrm{Div}^n_C$, 
the functor of effective divisors on $C$ of degree $n$.

The map $j^n: C^n \to \mathrm{Jac}(C)$ descends to $j^{(n)}: C^{(n)} \to \mathrm{Jac}(C)$. 
In terms of the description in Remark~\ref{Rotherview}, $j^{(n)}(\{Q_1, \ldots, Q_n\}) = [Q_1 + \cdots Q_n -nP_0]$.
By \cite[Theorem~5.1]{milneJacobian}, for $n \leq g$, the map $j^{(n)}$ is birational onto its image and $j^{(g)}$ is also surjective.
The map $j^{(g-1)}$ defines a theta divisor of effective divisor classes of degree $g-1$, 
which determines a canonical principal polarization on $\mathrm{Jac}(C)$.

Torelli's Theorem states that 
every smooth curve $X$ over $k$ is uniquely determined by its Jacobian.

\begin{theorem} \label{Ttorelli1} (Torelli's Theorem) \cite[Theorem~12.1]{milneJacobian}
Suppose $C$ and $C'$ are two (smooth, projective, irreducible) curves over $k$ of genus $g \geq 2$.
If $\mathrm{Jac}(C)$ and $\mathrm{Jac}(C')$ are isomorphic as principally polarized abelian 
varieties over $k$, then $C$ and $C'$ are isomorphic as curves over $k$.
\end{theorem}

\section{Related topics: semi-abelian varieties and stable curves}

For completeness, we include a definition of semi-abelian varieties, although they are not needed in this document.
An (algebraic) torus is a commutative group scheme which is isomorphic to finitely many copies of the multiplicative group scheme 
$\mathbb{G}_m$.
A \emph{semi-abelian variety} is a commutative group variety which is an extension of an abelian variety by an algebraic torus \cite[Section I.2]{faltingschaidegeneration}.

In later sections, we consider curves that are not smooth, but whose singularities are manageable.
A \emph{stable curve} is a projective irreducible curve whose only singularities are ordinary double points,
and whose automorphism group is finite. 

\begin{example}
The affine equation $y^2=x^3+x^2$ defines a stable curve with an ordinary double point at $(x,y)=(0,0)$.
\end{example}

A stable curve $C$ has \emph{compact type} if each irreducible component of $C$ is smooth and
the dual graph of $C$ is a tree.

In Section~\ref{Sboundary}, we consider the Picard variety (or Jacobian) of a singular stable curve.
The Picard variety $\mathrm{Pic}^0(C)$ is an abelian variety if and only if $C$ has compact type.
If not, then $\mathrm{Pic}^0(C)$ is a semi-abelian variety.

\begin{fact} \label{Ftorellifalse}
Torelli's Theorem~\ref{Ttorelli1} is false for singular stable curves.
\end{fact}

\begin{example}
Let $C$ be a curve of genus $3$ that has two components, an elliptic curve $C_1$ and a genus $2$ curve $C_2$, 
which are identified (clutched together, see Section~\ref{Sboundary}) at the identity on $C_1$ and a point $P \in C_2$.
There is a one-parameter family of such curves, as the point $P \in C_2$ varies.
However, $\mathrm{Jac}(C) \cong \mathrm{Jac}(C_1) \times \mathrm{Jac}(C_2)$, 
and this p.p.\ abelian variety does not depend on the choice of $P$.
\end{example}

\section{Open questions on variations of the Schottky problem}

The questions in this section aim to clarify the relationships between abelian varieties and Jacobians.

Every abelian variety over an algebraically closed field $k$ is a quotient of a Jacobian \cite{matsusakagenerating};
see \cite[Theorem~10.1]{milneJacobian} for this statement over any infinite field.
Every abelian variety over a field can be embedded in a Jacobian by \cite[Corollary~2.5]{gabber}.

\begin{example}
There are examples, usually attributed to Mumford, of p.p.\ abelian varieties that are not Jacobians.
For example, the intermediate Jacobian of a non-singular cubic threefold is not a Jacobian \cite[Appendix C]{clemensgriffiths}.
There are examples of Prym varieties of unramified double covers of curves which are not Jacobians \cite{beauville}.
See also \cite{Farb}.
\end{example}

\subsection{The Schottky problem}

The Schottky problem
asks for a characterization of the p.p.\ abelian varieties that are Jacobians of curves.
Historically, this topic was introduced by Schottky and Jung. 
The survey by Grushevsky \cite{grushevskysurvey} gives a good introduction, together with many references.

The Schottky problem has been studied with many approaches including theta constants, 
modular forms, singularities of the theta divisor, minimal cohomology classes, and Kummer varieties.
Theoretically, there are several methods to determine which p.p.\ abelian varieties are Jacobians.
There is the criterion of Matsusaka-Ran \cite{matsusaka}, \cite{ran}, \cite{collino}.
A small selection of the papers on this topic includes Andreotti--Mayer \cite{andreottimayer}, 
Gunning \cite{gunning}, Welters \cite{welters84}, 
Shiota \cite{shiota}, and Krichever \cite{krichever06, krichever10}. 

From certain perspectives, the Schottky problem has been completely solved.
Yet these methods are fairly difficult to apply in practice for a given abelian variety.
There are still open problems on the characterization of the Torelli locus.
For example, for $g=4$, see two recent descriptions in \cite{kramerweissauer} 
and \cite{hanselmanpiepersam}.

\subsection{Complete decomposability}

For every integer $g \geq 2$, there exists a p.p.\ abelian variety of dimension $g$ 
which decomposes completely as a product of $g$ elliptic curves.
Ekedahl and Serre asked the following question.  They provided examples for numerous values of $g$ up to $1297$.

\begin{question} \label{Qchapter2} \cite{ekedahlserre93}
Given $g \geq 2$, does there exist a smooth curve $C$ of genus $g$ such that the Jacobian $J_C$
is isogenous to a product of $g$ elliptic curves?
\end{question}

A paper by Paulhus and Rojas \cite{paulhusrojas17} shows that the question has an affirmative answer for many new 
values of $g$.  It also includes references to other papers on this topic.
Recently, thirteen new cases (including $g=38$) were resolved by Paulhus and Sutherland using modular curves \cite{paulhussutherland}.
Currently, the smallest genus for which the answer is not known is $g=56$.

\subsection{Abelian varieties not isogenous to a Jacobian}

For any $g \geq 4$, Chai and Oort \cite{chaioortnotisog} proved the existence of an abelian variety of dimension $g$
defined over $\overline{\QQ}$ which not isogenous to any Jacobian conditional on the Andr\'e--Oort conjecture; 
Tsimerman \cite{tsimermannotisog} proved this result unconditionally.
Recent work on this topic includes \cite{masserzannier} and \cite{notisogFortmanSchreider}.

For abelian varieties of dimension $g \geq 4$ defined over number fields, 
we expect there are still many results about isogenies with Jacobians waiting to be discovered.

\chapter{The $p$-torsion invariants} \label{Chapter3}

\section{Overview: generalizing the ordinary/supersingular distinction}

Let $k$ be an algebraically closed field of positive characteristic $p$.
An elliptic curve over $k$ can be ordinary or supersingular, 
depending on how many $p$-torsion points it has, see Sections~\ref{Scollapse} and \ref{SssE}.
This section describes several ways to generalize the distinction between ordinary and supersingular 
elliptic curves
for abelian varieties of dimension greater than $1$.

Suppose $X$ is a p.p.\ abelian variety of dimension $g$ defined over $k$.  
This section contains the definition of these invariants: 
the {\bf $p$-rank}, the {\bf $a$-number}, the {\bf Newton polygon}, and the {\bf Ekedahl--Oort type}.
We call these the \emph{$p$-torsion invariants}, although
the Newton polygon is technically an invariant of the $p$-divisible group of $X$.
Good references for this section are \cite{LO}, \cite{O:strat}, and \cite{EVdG}.

\subsection{Collapsing of $p$-torsion points modulo $p$} \label{Scollapse}

Suppose $E$ is an elliptic curve over $k$, with $\mathrm{char}(k)=p$.
In this expository section, we show through some examples that the number of $p$-torsion points on $E$ is either $p$ or $1$. 
 
If $\ell \not = p$ is prime, then there are $\ell^2$ points of order dividing $\ell$ on $E$.
One of these is the point at infinity $\mathcal{O}_E$.  
The $x$-coordinates of the other points are the roots of the $\ell$-division polynomial of $x$.

\begin{example} If $p \not = 2$, 
write $E: y^2=x^3+ax^2+bx+c$.
Let $\ell = 3$.
A point $Q$ has order 3 if and only if $3Q = \mathcal{O}_E$, equivalently $2Q=-Q$, so $x(2Q)=x(Q)$. 
Using this, we can show that $Q$ has order $3$ if and only if $x(Q)$ is a root of the $3$-division polynomial:
\[d_3(x) =3x^4 + 4ax^3 + 6bx^2 + 12cx - b^2 + 4ac \in k[x].\]
If $p \nmid 6$, then $d_3(x)$ has 4 distinct roots in $k$ and these are the $x$-coordinates of points of order $3$ on $E$. 
For each $x$-coordinate, there are two choices for $y$, so $E$ has 8 points of order 3.
The $3$-torsion $E[3](k)$ consists of these 8 points and $\mathcal{O}_E$.

Now suppose that $\mathrm{char}(k)=p=3$.
Then $d_3(x) = ax^3 - b^2 + ac$, and 
the polynomial $d_3(x) \in k[x]$ has one (triple) root if $a \not \equiv 0 \bmod 3$
and has no roots if $a \equiv 0 \bmod 3$.
So the number of $3$-torsion points is either $3$ or $1$, not $9$.
\end{example}

\begin{example}
Write $E: y^2 = x^3 + bx + c$.
When $\mathrm{char}(k)=5$, the $5$-division polynomial is $2bx^{10} - b^2cx^5 + b^6 - 2b^3c^2 - c^4$.
This has either $2$ or zero roots, so the number of $5$-torsion points is either $5$ or $1$.

When $\mathrm{char}(k)=7$, the $7$-division polynomial is
\[3cx^{21} + 3b^2c^2x^{14} + (-b^7c - 2b^4c^3 + 3bc^5)x^7
- b^{12} - b^9c^2 + 3b^6c^4 - b^3c^6 + 2c^8.\]
This has either $3$ or zero roots, so the number of $7$-torsion points is either $7$ or $1$.
\end{example}

More generally, in characteristic $p$, the $p$-division polynomial of $E$ has either $(p-1)/2$ or zero roots.  
It is not easy to show this explicitly for larger $p$ because the $p$-division polynomials have higher degree.
For this reason, to show that the $p$-torsion points on $E$ collapse to either $p$ points or $1$ point in characteristic $p$, 
we switch to looking at the kernel $E[p]$ of the multiplication by $p$ map $[p]$ on $E$,   
and use the fact that $[p]$ factors as $[p] = \mathrm{Ver} \circ \mathrm{Fr}$,
where the relative Frobenius morphism $\mathrm{Fr}:E \to E^{(p)}$ is purely inseparable.

\subsection{Supersingular elliptic curves} \label{SssE}

Suppose that $E$ is an elliptic curve defined over a finite field $\FF_q$ where $q=p^r$.
Let $a \in \ZZ$ be such that $\#E(\FF_q)=q+1-a$.
The zeta function of $E/\FF_q$ is 
\[Z(E/\FF_q, T)= \frac{1-aT+qT^2}{(1-T)(1-qT)}.\]

The supersingular condition was studied by Hasse \cite{Hasse35}, Deuring \cite{deuring}, and Igusa \cite{igusa}.
As seen in \cite[Theorem V.3.1]{aec}, there are many equivalent ways to define what it means for $E$ to be supersingular.
In this section, we say $E/\FF_q$ is supersingular when $p \mid a$, see \cite[page 142]{aec}; 
otherwise $E$ is ordinary.

If $p=2$, then $E:y^2+y=x^3$ is supersingular; this is a special case of Lemma~\ref{Lhermitian}. 
In fact, this is an equation for the unique supersingular elliptic curve over $\overline{\FF}_2$ up to isomorphism.

By \cite[Example V.4.4]{aec}, the elliptic curve $E:y^2=x^3+1$ ($j$-invariant $0$) is supersingular if and only if 
$p \equiv 2 \bmod 3$ and $p$ is odd.
By \cite[Example V.4.5]{aec}, the elliptic curve $E:y^2=x^3+x$ ($j$-invariant $1728$)
is supersingular if and only if $p \equiv 3 \bmod 4$. 
When $p=3$, this is an equation for the unique 
supersingular elliptic curve over $\overline{\FF}_3$ up to isomorphism.

Suppose $p$ is odd and $E: y^2=h(x)$, where $h(x) \in k[x]$ is a cubic with distinct roots.
Then $E$ is supersingular 
if and only if the coefficient $c_{p-1}$ of $x^{p-1}$ in $h(x)^{(p-1)/2}$ is zero.
This is a special case of Example~\ref{EhyperCartier}.

As seen in \cite[Theorem V.4.1]{aec}, for $p$ odd, Deuring \cite{deuring} proved that
\[E_\lambda: y^2=x(x-1)(x-\lambda)\] is supersingular for exactly $(p-1)/2$ choices of 
$\lambda \in \overline{\FF}_p$.  
This shows that the number of isomorphism classes of supersingular elliptic curves
is $\lfloor \frac{p}{12} \rfloor + \epsilon$ with $\epsilon = 0,1,1,2$ when $p \equiv 1,5,7,11 \bmod 12$ respectively.

Every supersingular elliptic curve defined over a finite field of characteristic $p$
can be defined over $\FF_{p^2}$.

\subsection{Ordinary and supersingular elliptic curves} \label{Srevisit1} 

We revisit the case of elliptic curves and describe the 
distinction between ordinary and supersingular elliptic curves from several points of view. 

Let $E/k$ be an elliptic curve and let $\ell$ be prime.
The $\ell$-torsion group scheme $E[\ell]$ of $E$ is the kernel of the multiplication-by-$\ell$ morphism
$[\ell]:E \to E$.
Then 
\[
\#E[\ell](k) =
\begin{cases}
\ell^2 & \text{if } \ell \not = p,\\
\ell & \text{if } \ell = p, \text{ and } E \text{ is ordinary}, \\
1 & \text{if } \ell = p, \text{ and } E \text{ is supersingular}.
\end{cases}\]

We refer to Section~\ref{Sbackptorsion} for definitions.
Each of the following conditions is equivalent to $E$ being ordinary: $E$ has $p$ points of order dividing $p$;
the Newton polygon of $E$ has slopes $0$ and $1$; or $E[p]$ is isomorphic to $\ZZ/p\ZZ \oplus \mu_p$ 
(see Example~\ref{EdefineL}).

Each of the following conditions is equivalent to $E$ being supersingular:

\smallskip

{\bf (A)'} The only $p$-torsion point of $E$ is the identity: $E[p](k) = \{\mathrm{id}\}$.

\smallskip

{\bf (B)'} The Newton polygon of $E$ is a line segment of slope $1/2$.
\smallskip

{\bf (C)'} The group scheme $E[p]$ is isomorphic to $I_{1,1}$, the unique local-local symmetric $\mathrm{BT}_1$ group scheme of rank $p^2$, 
see Example~\ref{EdefineI11}.

\smallskip

Conditions (A)' and (B)' are equivalent \cite[Theorem V.3.1 \& page~142]{aec}.
More information about condition (C)' can be found in Example~\ref{EdefineI11}.

\section{The $p$-torsion invariants of abelian varieties} \label{Sbackptorsion}

Let $k$ be an algebraically closed field of characteristic $p>0$.
Let $X$ be a principally polarized abelian variety of dimension $g$ defined over $k$.
Let $X[p]$ denote the $p$-torsion group scheme of $X$ and let $X[p^\infty]$ denote the $p$-divisible group of $X$.
In this section, we define the following $p$-torsion invariants of $X$:

\smallskip

{\bf A. $p$-rank} - the integer $f$, with $0 \leq f \leq g$, such that $\#X[p](k)=p^{f}$.

\smallskip

Also the {\bf $a$-number} - the number of copies of $\alpha_p$ in $X[p]$.

\smallskip

{\bf B. Newton polygon} - the data of slopes for the $p$-divisible group $X[p^\infty]$.

\smallskip

{\bf C. Ekedahl-Oort type} - the data of the symmetric ${\mathrm BT}_1$ group scheme $X[p]$.

\subsection{The $p$-torsion group scheme} \label{Sgroupscheme}

The multiplication-by-$p$ morphism $[p]:X \to X$ is a finite flat morphism of degree $p^{2g}$.
There is a canonical factorization $[p] = \mathrm{Ver} \circ \mathrm{Fr}$, 
where $\mathrm{Fr}:X \to X^{(p)}$ is the relative Frobenius morphism
and $\mathrm{Ver}: X^{(p)} \to X$ is the Verschiebung morphism.
The morphism $\mathrm{Fr}$ comes from the $p$-power map on the structure sheaf; it is purely inseparable of degree $p^g$.  
Also $\mathrm{Ver}$ is the dual of $\mathrm{Fr}_{X^*}$, where $X^*$ is the dual of $X$.  

The \emph{$p$-torsion group scheme} of $X$ is 
\[X[p]= \mathrm{Ker}[p].
\]
In fact, $X[p]$ is
a symmetric $\mathrm{BT}_1$ group scheme \cite[2.1, Definition~9.2]{O:strat}.  
It has rank $p^{2g}$.  It is killed by $[p]$, with
$\mathrm{Ker}(\mathrm{Fr}) = \mathrm{Im}(\mathrm{Ver})$ and $\mathrm{Ker}(\mathrm{Ver}) = \mathrm{Im}(\mathrm{Fr})$.

The principal polarization on $X$ induces a principal quasipolarization (pqp) on $X[p]$,
i.e., an anti-symmetric isomorphism $\psi:X[p] \to X[p]^D$, where $D$ denotes the Cartier dual.  
(This definition needs to be modified slightly if $p=2$.)
Thus, $X[p]$ is a symmetric $\mathrm{BT}_1$ group scheme together with a principal quasipolarization.

We return to this topic in Section~\ref{Seotype} when defining the Ekedahl--Oort type.

More information about group schemes can be found in \cite[Appendix A]{G:book}.

\begin{example} \label{EdefineL}
Let $\ZZ/p\ZZ$ be the constant group scheme representing the cyclic group of order $p$.
Let $\mu_p$ be the kernel of Frobenius on $\mathbb{G}_m$.
As $k$-schemes, $\ZZ/p\ZZ = \mathrm{Spec}(k[x]/(x^p-x))$ and $\mu_p = \mathrm{Spec}(k[x]/(x^p-1))$.
Let $L:=\ZZ/p\ZZ \oplus \mu_p$.
If $E$ is an ordinary elliptic curve, then $E[p] \cong L$.
\end{example}

\begin{example} \label{EdefineI11}
See \cite[Appendix A, Example 3.14]{G:book}.
Let $\alpha_p$ be the group scheme which is the kernel of Frobenius on $\mathbb{G}_a$.
As a $k$-scheme, 
$\alpha_p \cong \mathrm{Spec}(k[x]/x^p)$ with co-multiplication $m^*(x)=x \otimes 1 + 1 \otimes x$
and co-inverse $\mathrm{inv}^*(x)=-x$.

Define $I_{1,1}$ to be the local-local symmetric $\mathrm{BT}_1$ group scheme of rank $p^2$; it is unique up to isomorphism.
It fits in a non-split exact sequence
\begin{equation} \label{EI11}
0 \to \alpha_p \to I_{1,1} \to \alpha_p \to 0.
\end{equation}
The image of $\alpha_p$ in \eqref{EI11} is the kernel of $\mathrm{Fr}$ (Frobenius).
If $E$ is a supersingular elliptic curve, then $E[p] \cong I_{1,1}$.
\end{example}

\subsection{The $p$-rank and $a$-number} \label{Sprankanumber} 

The \emph{$p$-rank} of $X$ is 
\begin{equation} \label{Edefprank}
f=\mathrm{dim}_{\FF_p} \mathrm{Hom}(\mu_p, X),
\end{equation}
where $\mu_p$ is the kernel of Frobenius on $\mathbb{G}_m$.

The $p$-rank determines the number of $p$-torsion points on $X$, namely $p^f =\#X[p](k)$. 
The reason is that the multiplicities of the group schemes $\ZZ/p\ZZ$ and $\mu_p$ in $X[p]$ are the same 
because of the symmetry induced by the polarization. 

The \emph{$a$-number} of $X$ is 
\begin{equation} \label{Edefanumber}
a=\mathrm{dim}_k \mathrm{Hom}(\alpha_p, X),
\end{equation}
where $\alpha_p$ is the kernel of Frobenius on $\mathrm{G}_a$.
Then $0 \leq f \leq g$ and $1 \leq a +f \leq g$.

\begin{definition} \label{Dordss}
An abelian variety $X$ of dimension $g$ is \emph{ordinary} if 
$f=g$; equivalently, $X$ is ordinary if $a = 0$.
We say $X$ is \emph{superspecial} if $a=g$.
\end{definition}

Since $\mu_p$ (resp.\ $\alpha_p$) is simple,
the $p$-rank (resp.\ $a$-number) is additive;
\begin{equation}
\label{eqfadditive}
f(X_1 \times X_2) = f(X_1)+f(X_2)\text{ and }a(X_1\times X_2) = a(X_1)+a(X_2).
\end{equation}

The $p$-rank and $a$-number can also be defined for a $p$-torsion group scheme.

\subsection{The $p$-divisible group} \label{SNPpdivisible}

For each $n \in \mathrm{N}$, consider the multiplication-by-$p^n$ morphism $[p^n]:X \to X$ and its kernel $X[p^n]$.
The $p$-divisible group of $X$ is $X[p^\infty] = \varinjlim X[p^n]$.

For each pair $(c,d)$ of non-negative relatively prime integers (not both zero), fix
an isoclinic $p$-divisible group $G_{c,d}$ of codimension $c$, dimension $d$, and thus height $c+d$.
Let $\lambda = d/(c+d)$.
By the Dieudonn\'e-Manin classification \cite{maninthesis}, 
for $\lambda \in \QQ \cap [0,1]$,
there exists $m_\lambda \in \ZZ^{\geq 0}$ (all but finitely many of which are zero) 
such that there is an isogeny of $p$-divisible groups 
\begin{equation} \label{Enpgcd}
X[p^\infty] \sim \oplus_{\lambda=\frac{d}{c+d}} G_{c,d}^{m_\lambda}.
\end{equation} 

We return to \eqref{Enpgcd} when discussing Dieudonn\'e modules in Section~\ref{subsecdefcartier}.

\subsection{Supersingular abelian varieties}

\begin{definition}
A principally polarized abelian variety $X$ is \emph{supersingular} if 
$\lambda=1/2$ is the only slope of 
its $p$-divisible group $X[p^\infty]$.
\end{definition}

Letting $G_{1,1}$ denote the $p$-divisible group of dimension $1$ and height $2$,
then $X$ is supersingular if and only $X[p^\infty] \sim G_{1,1}^g$ \cite[Section 1.4]{LO}.

When $X$ is defined over a finite field $\FF_q$ (with $q=p^r$),  
there are some other ways to characterize the supersingular property.
Consider the characteristic polynomial $P(X/\FF_q, T)$ of Frobenius on $X$ (or its $\ell$-adic Tate module, for $\ell \not = p$).

By the Weil conjectures, $P(X/\FF_q, T)$ is a monic polynomial of degree $2g$ with integer coefficients.
Furthermore, $P(X/\FF_q, T) = \prod_{i=1}^{2g} (T-\alpha_i)$ where $|\alpha_i|=\sqrt{q}$.
These facts imply that there are integers $a_1, \ldots, a_g$ such that 
\begin{equation} \label{EcharpolyFrob}
P(X/\FF_q, T) = T^{2g} + a_1 T^{2g-1} + \cdots + a_g T^g + q a_{g-1}T^{g-1} + \cdots + q^g.
\end{equation}

\begin{theorem} \label{Lpropertysupersingular}
A principally polarized abelian variety $X/\FF_q$ of dimension $g$ is supersingular if and only if any of the following equivalent conditions hold:
\begin{enumerate}
\item the integer $a_j$ is divisible by $p^{\lceil jr /2 \rceil}$ for $1 \leq j \leq g$ \cite{manin61} or \cite[page~116]{O:sub};
\item $\mathrm{End}_{\overline{\FF}_q}(X) \otimes \QQ \cong \mathrm{Mat}_g(D_p)$, 
where $D_p$ is the quaternion algebra ramified only over $p$ and $\infty$ \cite[Theorem 2d]{tate:endo};
\item $X$ is geometrically isogenous to $E^g$ for some 
supersingular elliptic curve $E/\overline{\FF}_p$  \cite[Theorem~4.2]{O:sub}, which relies on \cite[Theorem 2d]{tate:endo}.
\end{enumerate}
\end{theorem}

\subsection{The Newton polygon} \label{SNPabelianvariety}
The Newton polygon is an isogeny invariant of the $p$-divisible group of $X$ (and the finest isogeny invariant).
If $X_1$ and $X_2$ are isogenous, then they have the same Newton polygon.

\begin{definition}
Recall from \eqref{Enpgcd} the decomposition of the $p$-divisible group 
$X[p^\infty] \sim \oplus_{\lambda=\frac{d}{c+d}} G_{c,d}^{m_\lambda}$.
The \emph{Newton polygon} $\xi(X)$ is the multi-set of values of $\lambda\in \QQ \cap [0,1]$ in \eqref{Enpgcd}, which are called the \emph{slopes}. 
\end{definition}

The Newton polygon is determined by the multiplicities of its slopes.
For $\lambda \in \QQ \cap [0,1]$, let $n_\lambda$ denote the multiplicity of $\lambda$ in the multiset $\xi(X)$.
If $c,d\in\NN$ are relatively prime integers such that $\lambda=d/(c+d)$, then $n_\lambda = (c+d) m_\lambda$. 
The Newton polygon is \emph{symmetric} if $n_\lambda=n_{1-\lambda}$ (or $m_\lambda = m_{1-\lambda}$) 
for every $\lambda\in \QQ\cap [0,1]$. 

\begin{lemma}
The $p$-rank of $X$ is the multiplicity $n_1$ of the slope $1$ in $\xi(X)$.
\end{lemma}

The Newton polygon is typically drawn as a lower convex polygon, 
with slopes equal to the values of $\lambda$ occurring with multiplicity $n_\lambda$.
The Newton polygon of a $g$-dimensional abelian variety $X$ is symmetric and, when drawn as a polygon,  
it has endpoints $(0,0)$ and $(2g,g)$ and integral break points. 

There is a partial ordering on Newton polygons of the same height $2g$: 
one Newton polygon is smaller (meaning less ordinary) than a second if the lower convex hull of the first is never below the second.
We write $\xi_1\leq \xi_2$ if $\xi_1,\xi_2$ share the same endpoints and $\xi_1$ 
lies on or above $\xi_2$.  This defines a partial ordering on Newton polygons for abelian varieties of dimension $g$.
In this partial ordering, the ordinary Newton polygon is maximal and the supersingular Newton polygon is minimal. 

\begin{notation}
We use $\oplus$ to denote the union of multi-sets.  For any multi-set $\xi$, and $n\in\NN$, 
we write $\xi^n$ for the union of $n$ copies of $\xi$. 

Let $\mathrm{ord}$ denote the Newton polygon $\{0,1\}$ and $\mathrm{ss}$ denote the Newton polygon $\{1/2,1/2\}$.
Let $\sigma_g:=\mathrm{ss}^g$ denote the supersingular Newton polygon of height $2g$.

For $s,t\in \NN$, 
with $s \leq t/2$ and $\mathrm{gcd}(s,t)=1$, let $(s/t, (t-s)/t)$ denote the Newton polygon with slopes $s/t$ and  $(t-s)/t$, each with multiplicity $t$. 
\end{notation}

Thus the Newton polygon of an ordinary (resp.\ supersingular) abelian variety of dimension $g$
is $\mathrm{ord}^g$ (resp.\ $\sigma_g=\mathrm{ss}^g$).

\begin{remark} \label{SNPversion2}
Suppose $X$ is defined over $\overline{\FF}_p$.
Then there exists a finite subfield $\FF_0$ of characteristic $p$ 
such that $X$ is isomorphic to the base change to $\overline{\FF}_p$ of an abelian variety $X_0$ over $\FF_0$  
and Frobenius is an $\FF_0$-endomorphism on $X_0$.

Let $W(\FF_0)$ denote the Witt vector ring of $\FF_0$. 
Consider the action of Frobenius $\mathrm{Fr}$ on the crystalline cohomology group $H^1_{\mathrm{cris}}(X_0/W(\FF_0))$. 
By Weil's Theorem, there exists an integer $n$, e.g., $n=[\FF_0 : \FF_p]$, such that
the composition of $n$ Frobenius actions $\mathrm{Fr}^n$ is a linear map on $H^1_{\mathrm{cris}}(X_0/W(\FF_0))$.

In this situation, the \emph{Newton polygon} $\xi(X)$ of $X$ is the multi-set of rational numbers $\lambda$
such that $n\lambda$ are the valuations at $p$ of the eigenvalues of $\mathrm{Fr}^n$. 
Note that $\xi(X)$ is independent of the choice of $X_0$, $\FF_0$, and $n$. 
\end{remark}

\subsection{Dieudonn\'e modules} \label{subsecdefcartier} 

The $p$-divisible group $X[p^\infty]$ and the $p$-torsion group scheme $X[p]$ can be described using covariant Dieudonn\'e theory.\footnote{Differences between the covariant and contravariant theory 
do not cause a problem in this manuscript since we are working over a perfect field 
(and also because all objects we consider are 
principally quasipolarized and thus symmetric).}
References for this material are \cite[Section I.1.4, Appendix B.3]{chaiconradoort} and \cite[15.3]{O:strat}.  

Briefly, let $\sigma$ denote the Frobenius automorphism of $k$ and its lift to the Witt vectors $W(k)$.
Let $\tilde \EE = \tilde\EE(k) = W(k)[F,V]$ denote the non-commutative ring generated by $F$ and $V$ with relations
\begin{equation} \label{Efvchar0}
FV=VF=p, \ F a = a^\sigma F, \ a V=V a^\sigma,
\end{equation} for all $a \in W(k)$.  

There is an equivalence of categories $\dieu_*$
between $p$-divisible groups over $k$ and $\tilde\EE$-modules which are free of finite rank over $W(k)$.
The Dieudonn\'e module $D_\lambda :=  \dieu_*(G_{c,d})$ is a free $W(k)$-module of rank $c+d$.
Over $\operatorname{Frac}W(k)$, there is a basis $x_1, \ldots, x_{c+d}$ for $D_\lambda$ such that $V^{c}x_i=F^d x_i$.
Roughly speaking, $V^{c+d}$ acts like $p^d$, indicating that the slope is $d/(c+d)$ in the covariant theory.

We now consider Dieudonn\'e modules modulo $p$.
Let $\EE = \tilde \EE \otimes_{W(k)} k$ be the reduction of $\tilde{\EE}$ modulo $p$; it is a non-commutative ring $k[F,V]$ subject to the same constraints as \eqref{Efvchar0}, except that $FV = VF = 0$ in $\EE$.  
Again, there is an equivalence of categories $\dieu_*$ between finite commutative group schemes (of rank $2g$) annihilated by $p$ and $\EE$-modules of finite dimension ($2g$) over $k$.

For elements $w_1, \ldots, w_r \in \EE$, 
let $\EE(w_1, \ldots, w_r)$ denote the left ideal $\sum_{i=1}^r \EE w_i$ of $\EE$ generated by $\{w_i \mid 1 \leq i \leq r\}$.

For a p.p.\ abelian variety $X$ of dimension $2g$, 
the mod $p$ Dieudonn\'e module of $X[p]$ is a vector space of dimension $2g$ and an $\EE$-module.\footnote{
If $M = \dieu_*(I)$ is the mod $p$ Dieudonn\'e module of a symmetric $\mathrm{BT}_1$ group scheme $I$, then  
a principal quasipolarization  $\psi:I \to I^D$ induces a 
nondegenerate symplectic form $\ang{\cdot,\cdot}:M \times M \to k$
on the underlying $k$-vector space of $M$, with the additional constraint that $\ang{Fx,y} = \ang{x,Vy}^\sigma$, for all $x, y \in M$.}

\begin{example} Recall the group scheme $L=\ZZ/p\ZZ \oplus \mu_p$ from Example~\ref{EdefineL}.  
It has mod $p$ covariant Dieudonn\'e module $\dieu_*(L) = \EE/\EE(V, F-1)
\oplus \EE/\EE(F, V-1)$.
\end{example}

\begin{example}\label{exi11} Recall the group scheme $I_{1,1}$ from Example~\ref{EdefineI11}.
It has mod $p$ Dieudonn\'e module $D_{1,1} :=\dieu_*(I_{1,1}) = \EE/\EE(F+V)$.
\end{example}

\subsection{The Ekedahl-Oort type} \label{Seotype} 

The $p$-torsion $X[p]$ of $X$ is a symmetric $\mathrm{BT}_1$-group scheme (of rank $2g$) annihilated by $p$.
We describe the isomorphism class of $X[p]$ using a combinatorial object called its Ekedahl-Oort type (E--O type).

Isomorphism classes of pqp $\mathrm{BT}_1$ group schemes over $k$
have been completely classified using these Ekedahl-Oort types \cite[Theorem 9.4
\& 12.3]{O:strat}.  
This builds on work of Kraft \cite{Kraft} (unpublished, which did not include polarizations) and of Moonen
\cite{M:group} (for $p \geq 3$). 
(When $p=2$, there are complications with the polarization which are resolved in \cite[9.2, 9.5, 12.2]{O:strat}.)
References on this material are \cite{O:strat}, \cite{EVdG}, and \cite{moonenwedhorn}.

The isomorphism type of a symmetric $\mathrm{BT}_1$ group scheme 
$I$ over $k$ can be encapsulated into combinatorial data in numerous ways, including
the Kraft word, the Ekedahl--Oort type, and the Young type.  
See \cite{Pr:sg} and \cite{PriesUlmerBT1} for longer descriptions of these topics.

To define the Ekedahl--Oort type, as in \cite[Sections 5 \& 9]{O:strat},
suppose $I$ is a symmetric $\mathrm{BT}_1$ group scheme over $k$ with rank $p^{2g}$. 
Then there is a \emph{final filtration} $N_1 \subset N_2 \subset \cdots \subset N_{2g}$ 
of ${\mathbb D}_*(I)$ as a $k$-vector space which is stable under the action of $V$ and $F^{-1}$ such that $i=\mathrm{dim}(N_i)$ \cite[5.4]{O:strat}.

The \emph{Ekedahl-Oort type} of $I$ is 
\[\nu:=[\nu_1, \ldots, \nu_g], \text{ where } {\nu_i}=\mathrm{dim}(V(N_i)).\]

\begin{lemma}
The $p$-rank is $\mathrm{max}\{i \mid \nu_i=i\}$ and the $a$-number equals $g-\nu_g$.
\end{lemma}

There is a restriction $\nu_i \leq \nu_{i+1} \leq \nu_i +1$ on the Ekedahl-Oort type.
There are $2^g$ Ekedahl-Oort types of length $g$ since all sequences satisfying this restriction occur.   
By \cite[9.4, 12.3]{O:strat}, there are bijections between (i) Ekedahl-Oort types of length $g$; (ii) 
pqp $\mathrm{BT}_1$ group schemes over $k$ of rank $p^{2g}$;
and (iii) pqp Dieudonn\'e modules of dimension $2g$ over $k$.

By \cite{EVdG}, the Ekedahl-Oort type can also be described by its Young type $\mu$.  
Given $\nu$, for $1 \leq  j \leq g$, consider the strictly decreasing sequence
\[\mu_j = \#\{i \mid 1 \leq i \leq  g, \  i - \nu_i \geq j\}.\] 
There is a Young diagram with $\mu_j$ squares in the $j$th row. 
(Unlike in combinatorics, we draw the Young diagrams to look like a staircase, ascending to the right.)
The \emph{Young type}
is $\mu = \{\mu_1, \mu_2, . . .\}$, where one eliminates all $\mu_j$ which are $0$. 

\begin{lemma}
The $p$-rank is $g - \mu_1$ and the $a$-number is $a = \mathrm{max}\{j \mid \mu_j \not = 0\}$.
\end{lemma}

\begin{example} \label{Eir1}
Let $r \in \mathrm{N}$.  
There is a unique symmetric $\mathrm{BT}_1$ group scheme of rank $p^{2r}$
with $p$-rank $0$ and $a$-number $1$, which we denote $I_{r,1}$. 
The mod $p$ Dieudonn\'e module of $I_{r,1}$ is $\dieu_*(I_{r,1}) \cong \EE/\EE(F^r + V^r)$.
For $I_{r,1}$, the Ekedahl-Oort type is $[0,1,2,\ldots, r-1]$ and the Young type is $\{r\}$.
\end{example}

\section{Superspecial, supersingular, and $p$-rank $0$ abelian varieties}

\subsection{The difference between $p$-rank $0$ and supersingular} \label{Sqinv} 

Let $X$ be a principally polarized abelian variety of dimension $g$ over $k$.
Let $X[p]$ be the kernel of the multiplication-by-$p$ morphism of $X$.
The following conditions are all different for $g \geq 3$.

\smallskip

{\bf (A) $p$-rank $0$} -  The only $p$-torsion point of $X$ is the identity: $\#X[p](k) =1$.

\smallskip

{\bf (B) supersingular} - The Newton polygon of $X$ is a line segment of slope $1/2$.

\smallskip

{\bf (C) superspecial} -  The group scheme $X[p]$ is isomorphic to $(I_{1,1})^g$.

\smallskip

\begin{proposition} \label{Pprankssss}
For conditions $(A),(B),(C)$ as defined above, there are implications:
\[(C) \Rightarrow (B) \Rightarrow (A), 
\text{ but }   (A) \stackrel{g \geq 3}{\not \Rightarrow} (B) \stackrel{g \geq 2}{\not \Rightarrow} (C).\] 
\end{proposition}

\begin{proof}(Sketch)
\begin{enumerate}
\item{For the implication $(C) \Rightarrow (B)$:}
if the $p$-torsion of a $p$-divisible group $G$ satisfies (C), then 
$F^2 G \subset [p] G$.  By the basic slope estimate in \cite[1.4.3]{katzslope}, the slopes of the Newton polygon are all 
at least $1/2$; so the slopes all equal $1/2$, because the polarization forces the Newton polygon to be symmetric.
Thus $X$ is supersingular.
Alternatively, the implication $(C) \Rightarrow (B)$ follows from \cite[Theorem~4.2]{O:sub}. 
\item{For the non-implication $(B) \not \Rightarrow (C)$ when $g \geq 2$:}
an abelian variety can be isogenous 
but not isomorphic to a product of supersingular elliptic curves;
for example, quotients of a superspecial abelian variety by a general $\alpha_p$-subgroup scheme have this property
when $g \geq 2$.
\item{For the implication $(B) \Rightarrow (A)$:} more generally,  
the $p$-rank of a $p$-divisible group 
is the multiplicity of the slope $1$ in the Newton polygon, so if all the slopes equal $1/2$, then the $p$-rank is $0$.
\item{For the non-implication $(A) \not \Rightarrow (B)$ when $g \geq 3$:}
there exists a principally polarized abelian variety over $k$ whose $p$-divisible group is $G_{g-1,1} \oplus G_{1,g-1}$, 
so its Newton polygon has slopes $1/g$ and $(g-1)/g$.
This has $p$-rank $0$, but is not supersingular when $g \geq 3$. \qedhere
\end{enumerate}
\end{proof}

\section{Related topics: small dimension; semi-abelian varieties} \label{SStablesNPEO}

\subsection{The $p$-torsion invariants in small dimension}

In this section, we include data for dimensions $g=2,3$.
These tables previously appeared in \cite{Pr:sg}, along with a table for $g=4$.
In these tables, the symmetric $\mathrm{BT}_1$ group schemes are first listed by name; 
recall that $L = \ZZ/p \oplus \mu_p$ and see Example \ref{Eir1} for the definition of $I_{r,1}$.
Then, for each group scheme, the tables include the
codimension of its stratum in $\cA_g$, and its $p$-torsion invariants
($p$-rank $f$, $a$-number $a$, Ekedahl-Oort type $\nu$, 
Young type $\mu$, and mod $p$ Dieudonn\'e module).  We also include some information about the slopes of the Newton polygon.

\smallskip

{\bf The case $g=2$} 

The following table shows the 4 symmetric $\mathrm{BT}_1$ group schemes that occur
for principally polarized abelian surfaces.  

\renewcommand{\arraystretch}{1.065}
\[\begin{array}{|l|c|c|c|c|l|r|r|}
\hline
\mbox{Name} & \mathrm{cod} & f & a &  \nu & \mu & \mbox{Dieudonn\'e module} & \mbox{Newton polygon} \\
\hline
\hline
L^2 & 0 & 2 & 0 &  [1,2] & \emptyset & D(L)^2 & 0,0,1,1 \\
\hline
L \oplus I_{1,1} & 1 & 1 & 1 &  [1,1] & \{1\} & D(L) \oplus D_{1,1} & 0,\frac{1}{2},\frac{1}{2},1 \\
\hline
I_{2,1} & 2 & 0 & 1 &  [0,1] & \{2\} &{\mathbb E}/\EE(F^2+V^2) & \frac{1}{2},\frac{1}{2},\frac{1}{2},\frac{1}{2}\\
\hline
(I_{1,1})^2 & 3 & 0 & 2 &  [0,0] & \{2,1\} & (D_{1,1})^2 & \frac{1}{2},\frac{1}{2},\frac{1}{2},\frac{1}{2} \\
\hline
\end{array}
\]

The last two rows contain all the supersingular objects.

\smallskip

{\bf The case $g=3$} 

The following table shows the 8 symmetric $\mathrm{BT}_1$ group schemes that occur
for principally polarized abelian threefolds.  

\renewcommand{\arraystretch}{1.065}
\[\begin{array}{|l|c|c|c|c|l|r|}
\hline
\mbox{Name} & \mathrm{cod} & f & a & \nu & \mu &  \mbox{Dieudonn\'e module}  \\
\hline
\hline
L^3 & 0 & 3 & 0 &          [1,2,3] & \emptyset & D(L)^3 \\
\hline
L^2 \oplus I_{1,1} & 1 & 2 & 1 & [1,2,2]  &  \{1\} & D(L)^2 \oplus D_{1,1} \\
\hline
L \oplus I_{2,1} & 2 & 1 & 1&    [1,1,2] &  \{2\} & D(L) \oplus {\mathbb E}/\EE(F^2+V^2) \\
\hline
L \oplus (I_{1,1})^2 & 3 & 1 & 2 & [1,1,1] &  \{2,1\} & D(L) \oplus (D_{1,1})^2 \\
\hline
I_{3,1} & 3 & 0 & 1 &  [0,1,2] & \{3\} & {\mathbb E}/\EE(F^3+V^3)  \\
\hline
I_{3,2} & 4 & 0 & 2 &          [0,1,1]  &  \{3,1\} &  \EE/\EE(V^{2}+F) \oplus \EE/\EE(F^{2}+V) \\
\hline
I_{1,1} \oplus I_{2,1} & 5 & 0 & 2 &   [0,0,1] &  \{3,2\} & D_{1,1} \oplus {\mathbb E}/\EE(F^2+V^2) \\
\hline
(I_{1,1})^3 & 6 & 0 & 3 &      [0,0,0] &  \{3,2,1\}   & (D_{1,1})^3 \\
\hline
\end{array}
\]

The objects in the last two rows are always supersingular.
If the E--O type is $I_{3,2}$, then the Newton polygon cannot be supersingular.
The situation for $I_{3,1}$ is more subtle.
By \cite[Theorem 5.12]{O:hypsup},
if $X[p] \cong I_{3,1}$, then the $p$-divisible group is usually isogenous to $G_{2,1} \oplus G_{1,2}$ 
(slopes $1/3, 2/3$) 
but it can also be isogenous to $G_{1,1}^3$ (supersingular).
This shows that the Ekedahl-Oort stratification does not refine the Newton polygon stratification
for $g \geq 3$.

\subsection{The $p$-torsion invariants for semi-abelian varieties} \label{Sptorsemi}

The material here is not needed in this manuscript, but we mention it for completeness.

Suppose $\tilde{X}$ is a semi-abelian variety which is an extension of an abelian variety $X$ by a torus of rank $e$. 
Let $f_X$, $a_X$, $\xi_X$, $X[p]$, $\nu_X$ denote respectively the 
$p$-rank, $a$-number, Newton polygon, $p$-torsion group scheme, and Ekedahl--Oort type of $X$.

The definition of the $p$-rank is also valid for semi-abelian varieties, namely
$f_{\tilde{X}}=\mathrm{dim}_{\FF_p} \mathrm{Hom}(\mu_p, \tilde{X})$.
Note that $f_{\tilde{X}} = f_X + e$.

The Newton polygon of $\tilde{X}$ is defined to be $\xi\oplus \mathrm{ord}^{e}$, which is obtained 
by adjoining $e$ slopes of $0$ and $e$ slopes of $1$ to $\xi$.

In \cite[Section~5]{EVdG}, the canonical filtration for $\tilde{X}$ is studied.  
As a consequence, the Ekedahl--Oort type of $\tilde{X}$ turns out
to be the E--O type for $X[p] \oplus L^e$ where $L= \mu_p \oplus \ZZ/p\ZZ$.
In other words, if $\nu_X = [\nu_1, \ldots, \nu_{g_X}]$, then the E--O type of $\tilde{X}$ is
$\nu_{\tilde{X}} = [1, \ldots, e, e + \nu_1, \ldots, e+\nu_{g_X}]$.
Thus the $a$-number of $\tilde{X}$ is $a_{\tilde{X}}=a_X$.

\section{Open questions on Ekedahl--Oort types}

For a p.p.\ abelian variety $X$, 
its Ekedahl-Oort type places restrictions on its Newton polygon and vice-versa, see \cite{harashita07, harashita10}.
However, there are situations where neither of these $p$-torsion invariants determines the other.
There are several questions in Karemaker's manuscript \cite[Section~3.5.2]{karemakerAWS} in this volume about the interaction between 
the Newton polygon and Ekedahl--Oort type.

We later include a question about the geometry of the Ekedahl--Oort stratification, see
Question~\ref{QpurityEO}.

\chapter{Existence of curves with given $p$-torsion invariants} \label{Chapter4}

\section{Overview: motivating questions about curves}

Let $k$ be an algebraically closed field of characteristic $p >0$.
Suppose $C$ is a (smooth, projective, irreducible) curve of genus $g$ defined over $k$.
The $p$-torsion invariants of $C$ are defined to be those of its Jacobian $\mathrm{Jac}(C)$.

This chapter contains some existence results for smooth curves 
with certain Newton polygons or Ekedahl--Oort types.
More general results about the $p$-rank are contained in Section~\ref{SFVdG}.

Here are the motivating questions.

\begin{question} \label{Qmot2}
If $p$ is prime and $g \geq 2$: does there exist a smooth curve $C/\overline{\FF}_p$ of genus $g$
whose Jacobian (A) has $p$-rank $0$; (B) is supersingular; or (C) is superspecial?  
\end{question}

For Question~\ref{Qmot2}, the answer to part (A) is yes for all $g$ and $p$, see Theorem~\ref{TFVdG};
part (B) is sometimes yes, as seen in this section, but is not known in most cases;  
the answer to part (C) is not known in most cases, but is sometimes no when $p$ is small relative to $g$, 
see Theorem~\ref{Tekedahl}.  More generally, one can ask:

\begin{question} \label{Qmot2general}
If $p$ is prime and $g \geq 2$: which $p$-ranks, $a$-numbers, Newton polygons, and Ekedahl-Oort types
occur for the Jacobians of smooth curves $C/\overline{\FF}_p$ of genus $g$? 
\end{question}

We survey some of the results on this topic.
In this chapter, we focus on results proven using cohomological calculations or decomposition of the Jacobian.

\section{Newton polygons of curves}

The Newton polygon of a curve $C$ is the Newton polygon of its Jacobian, as defined in Section~\ref{SNPabelianvariety}.
If $C$ has genus $g$, then its Newton polygon $\xi(C)$ consists of finitely many line segments, which break at points with integer coefficients, 
starting at $(0,0)$ and ending at $(2g,g)$.
It is symmetric; the multiplicities $n_\lambda$ and $n_{1-\lambda}$ of the slopes $\lambda$ and $1-\lambda$ in $\xi(C)$ are equal.
As before, the $p$-rank equals the multiplicity $n_1$ of the slope $1$ in $\xi(C)$.
As before, $C$ is supersingular if and only if the only slope of $\xi(C)$ is $\lambda = 1/2$.

\subsection{The Newton polygon of a curve over a finite field} \label{SNPcurve}

In Remark~\ref{SNPversion2}, we defined the Newton polygon of an abelian variety over a finite field.
When $C$ is defined over a finite field, there is an alternative definition of the Newton polygon. 

Let $C/\FF_q$ be a (smooth, projective, geometrically irreducible) curve of genus $g$ defined over a finite field $\FF_q$ of characteristic $p$.

\begin{definition}
For an integer $s \geq 1$, let $N_s = \#C(\FF_{q^s})$ be the number of points of $C$ defined over $\FF_{q^s}$.
The \emph{zeta function} of $C/\FF_q$ is 
\[Z(C/\FF_q, T)=\mathrm{exp}(\sum_{s=1}^\infty \frac{N_s T^s}{s}).\]

\end{definition}

Here is the famous theorem of Weil.

\begin{theorem} \label{Tweilconj} (Weil conjectures for curves \cite[\S IV, page 22]{weil}, \cite[\S IX, page 69]{weil2})
There is a polynomial $L(C/\FF_q, T) \in \ZZ[T]$ of degree $2g$ such that 
\[Z(C/\FF_q, T)=\frac{L(C/\FF_q, T)}{(1-T)(1-qT)}.\]
Furthermore, \[L(C/\FF_q, T)= \prod_{i=1}^{2g} (1- \alpha_i T),\] where 
the reciprocal roots $\alpha_i$ of $L(C/\FF_q, T)$ have the property that $|\alpha_i| = \sqrt{q}$.
\end{theorem}

The polynomial $L(C/\FF_q, T)$ is the \emph{$L$-polynomial} of $C/\FF_q$.
Its roots all have archimedean absolute value $1/\sqrt{q}$ in $\CC$.

\begin{lemma} \label{LcomparecharFrob}
The characteristic polynomial of the Frobenius endomorphism of $\mathrm{Jac}(C)$ is 
$P(\mathrm{Jac}(C)/\FF_{q},T) = T^{2g}L(C/\FF_{q},T^{-1})$.
\end{lemma}

The Newton polygon keeps track of the $p$-adic valuations of the roots or, equivalently, of the coefficients
of $L(C/\FF_q, T)$.  

\begin{notation}
Let $v_i$ be the $p$-adic valuation of the coefficient of $T^i$ in the $L$-polynomial $L(C/\FF_q, T)$.
Let $v_i/r$ be its normalization for the extension $\FF_q/\FF_p$, where $q=p^r$.
The Newton polygon is the lower convex hull of the points $(i, v_i/r)$ for $0 \leq i \leq 2g$.
\end{notation}

One can prove that the Newton polygons of $C/\FF_q$ and $\mathrm{Jac}(C)/\FF_q$ are the same.
So the curve $C/\FF_q$ is supersingular when the Newton polygon of $L(C/\FF_q, T)$ is a line segment of slope $1/2$.  
The $p$-rank of $C/\FF_q$ equals the number of roots of $L(C/\FF_q, T)$ that are $p$-adic units.

There are several ways to characterize the supersingular property for curves, in addition to those already described in Lemma~\ref{Lpropertysupersingular}.

\begin{lemma} \label{Lpropertysupersingular2}
Consider a curve $C/\FF_q$ of genus $g$.  
The following properties are equivalent:
\begin{enumerate}
\item the curve $C$ is supersingular;
\item the normalized Weil numbers $\alpha_i/\sqrt{q}$ are all roots of unity \cite[Theorem~4.1]{maninthesis}; and
\item the curve $C$ is minimal (meaning that it satisfies the lower bound in the Hasse-Weil bound for the number of points) over $\FF_{q^s}$ for some $s \geq 1$.
\end{enumerate}
\end{lemma}

\begin{remark}
Examples of supersingular curves were found by Weil \cite{weil}, Honda \cite{honda}, Yui \cite{yuiFermat}, Gross--Rohrlich \cite{GR}, and others.  
Their technique was to calculate the normalized Weil numbers of Fermat curves using Jacobi sums.  
\end{remark}

\begin{remark}
Many people worked on finding fast algorithms to compute the zeta function of an arbitrary 
curve over a finite field. 
Schoof developed a deterministic polynomial time algorithm for counting points on elliptic curves \cite{Schoof}.
This was generalized for hyperelliptic curves by Kedlaya \cite{kedlaya01, kedlaya04}, Harvey \cite{harvey07},
and Harvey--Sutherland \cite{harveysutherland1, harveysutherland2}.
For more information, we refer the reader to the (planned) lecture notes for the Arizona Winter School 2025.
\end{remark}

\section{The Hasse--Witt and Cartier--Manin matrices}

Consider the action of Vershiebung $\mathrm{Ver}$ on the vector space 
of regular differentials $H^0(C, \Omega^1)$. 
By \cite{Cartier}, \cite{maninthesis}, this action is given by the 
(so-called modified\footnote{The distinction between the modified and unmodified versions of the Cartier operator is a technical point that does not play a key role here.}) Cartier operator $\mathcal{C}$, which is the semi-linear map 
$\mathcal{C}:  H^0(C ,\Omega^1) \to H^0(C, \Omega^1)$
satisfying these rules, for $\omega_1, \omega_2 \in H^0(C ,\Omega^1)$ and $f \in  \mathcal{O}_C$:

(i) $\mathcal{C}(\omega_1 + \omega_2) = \mathcal{C}(\omega_1) + \mathcal{C}(\omega_2)$; 

(ii) $\mathcal{C}(f^p \omega_1) = f\mathcal{C}(\omega_1)$; and

(iii) $\mathcal{C}(f^{n-1}df) = 
\begin{cases} 
df & \mbox{ if } n = p,\\
0 & \mbox{ if } 1 \leq n < p.
\end{cases}$

\begin{lemma} \label{Lstablerank}
The $p$-rank of $C$ is the stable rank of the Cartier operator.
The $a$-number of $C$ is the corank of the Cartier operator.
\end{lemma}

Fix a basis $\beta = \{\omega_1, \ldots, \omega_g\}$ for $H^0(C, \Omega^1)$.  
From Serre duality, this fixes a basis $\beta^*$ for the dual space $H^1(C, \mathcal{O})$. 
The Hasse--Witt matrix is the matrix for the action of Frobenius $\mathrm{Fr}$ on $H^1(C, \mathcal{O})$ with respect to $\beta^*$.
The Cartier--Manin matrix is the matrix for the action of Vershiebung $\mathrm{Ver}$ 
on $H^0(C, \Omega^1)$ with respect to $\beta$.

For each $\omega_j$, let $m_{i,j} \in k$ be such that
$\mathcal{C}(\omega_j) = \sum_{i=1}^g m_{i,j} \omega_i$.
The $g \times g$-matrix $M = \left(m_{i,j}\right)$ is the (modified) Cartier--Manin matrix 
and it gives the action of the (modified) Cartier operator. 

The Cartier--Manin matrix is the matrix for the (unmodified) Cartier operator, and it is given by
$\tilde{M} := M^{(p)}$, where each entry is raised to the $p$th power.
The $a$-number is the corank of $\tilde{M}$ or, equivalently, of $M$. 
The $p$-rank is the rank of the product of $g$ twists of $\tilde{M}$ (or $M$) but this needs to be computed 
very carefully as described in Remark~\ref{Rwarning}.

\begin{example}
In \cite[Proposition~1]{miller}, 
Miller proved that there exists an ordinary curve of genus $g$ over $\bar{\FF}_p$ for all primes $p$ and $g \geq 2$.
Specifically: he proved that
$y^2=x^{2g+1} + t x^{g+1} + x$ is ordinary for a generic $t$ if $p \nmid g$; 
and $y^2=x^{2g+2} + t x^{g+1} + 1$ is ordinary for a generic $t$ if $p \mid g$.
The strategy is to find a basis for $H^0(C, \Omega^1)$ 
for which the Cartier--Manin matrix $\tilde{M}$ is a permutation matrix.  The result follows by showing that 
the determinant of $\tilde{M}$ is a non-zero polynomial in $t$.
\end{example}

\begin{example} \label{EhyperCartier}
Let $p$ be odd.
Let $C$ be a hyperelliptic curve with equation $y^2=h(x)$, where $h(x) \in k[x]$ is separable.
Consider the basis $\{dx/y, \ldots, x^{g-1}dx/y\}$ of $H^0(C, \Omega^1)$.
By \cite{yui78}, see also \cite[Section~3.1]{achterhowe},
with respect to this basis, the entry $m_{i,j}$ of $M$ is given by 
the coefficient of $x^{pi-j}$ in $h(x)^{(p-1)/2}$.
This is because 
\[\mathcal{C}(x^j\frac{dx}{y})=\mathcal{C}(x^j\frac{y^{p-1}dx}{y^p})=\frac{1}{y}\mathcal{C}(x^jh(x)^{(p-1)/2}dx)=
\sum_{i=1}^g (c_{ip-j})^{1/p} \frac{dx}{y}.\]
\end{example}

\begin{remark} \label{Rwarning}
{\bf Warning:} if $C$ is defined over a field field other than $\FF_p$, 
it is important to be extremely careful when using Lemma~\ref{Lstablerank}.
There are numerous mistakes in the literature about this, which were corrected in \cite{achterhowe}.
Because of the semi-linear property, when iterating $\tilde{M}$, the coefficients of the matrix need to be modified by 
$p$th powers. 
The $p$-rank is the rank of $\tilde{M}\tilde{M}^{(1/p)} \cdots \tilde{M}^{(1/p^{g-1})}$, which is the same as the rank 
of $\tilde{M}^{(p^{g-1})} \cdots \tilde{M}^{(p)} \tilde{M}$.
This may not be the same as the rank of 
$\tilde{M} \tilde{M}^{(p)} \cdots \tilde{M}^{(p^{g-1})}$.  
The ambiguity of acting on the left or the right also causes ambiguity.
We refer to \cite{achterhowe} for a careful analysis of this. 
\end{remark}

\begin{example}
In \cite{iko}, Ibukiyama, Katsura, and Oort count the number of superspecial curves of genus 
$2$ in terms of $p$, and the sizes of their automorphism groups.  
They use Igusa's description of curves of genus $2$ having extra automorphisms.
For example, for $t \in k-\{1\}$, the curve $y^2=(x^3-1)(x^3-t)$ has an action by the symmetric group $S_3$;
for $t \in k - \{0,\pm 1\}$, the curve
$y^2=x(x^2-1)(x^2-t)$ has an action by the dihedral group $D_4$ of order $8$.
For every smooth curve in these two families, the Cartier--Manin matrix is either invertible or is the zero matrix.
In the latter case, the curve is superspecial, and thus supersingular.
For each of these families, they count the 
number of values $t$ for which the curve is superspecial.
\end{example}

\begin{example}
A formula for the Cartier operator on plane curves is in \cite{stohrvoloch}.
\end{example}

\section{Ekedahl--Oort types and the de Rham cohomology} \label{SEOdRham}

The Ekedahl--Oort type of a curve over $k$ can be computed from its de Rham cohomology.
If $C$ is a curve of genus $g$ over $k$, then the de Rham cohomology group $H^1_{\mathrm{dR}}(C)$ 
is a vector space of dimension $2g$, with semi-linear operators $F$ and $V$.

Recall from Section~\ref{subsecdefcartier} that $\EE = \EE(k) = k[F,V]$ 
is the non-commutative ring generated by semilinear operators $F$ and $V$ with relations,  for all $a \in k$,
\begin{equation} \label{Efv}
FV=VF=0, \ F a = a^\sigma F, \ a V=V a^\sigma.
\end{equation}

Oda proved that there is an isomorphism of $\EE$-modules between the \emph{contravariant} 
Dieudonn\'e module over $k$ of $\mathrm{Jac}(C)[p]$ 
and $H^1_{\mathrm{dR}}(C)$ by \cite[Section~5]{Oda}.  
The canonical principal polarization on $\mathrm{Jac}(C)$ induces a canonical isomorphism 
$\dieu_*(\mathrm{Jac}(C)[p]) \cong H^1_{\mathrm{dR}}(C)$. 

Moonen developed another method to compute the Ekedahl--Oort type of curves using Hasse--Witt triples \cite{moonendiscrete}.
In \cite{dusan2}, Dragutinovi\'c extended the definition of Hasse--Witt triple for a stable curve (of non-compact type),
then showed that it is compatible with \cite[Section~5]{EVdG}, as discussed in Section~\ref{Sptorsemi}.

Not many Ekedahl--Oort types have been shown to occur for Jacobians.  The Ekedahl--Oort types of Hermitian curves
are computed in \cite{PW12}.  Every Ekedahl--Oort type occurs for a summand of the $p$-torsion of a Fermat curve 
\cite{priesulmerPAMS}.

\section{Existence results for curves of small genus} \label{Smallg}

When $g$ is small, there are more results about Question~\ref{Qmot2}.
There are both arithmetic and geometric approaches when $g \leq 4$.

Sometimes information about curves of low genus can be leveraged to yield information 
about curves of higher genus.  
Section~\ref{Sinductive} contains an inductive method to do this for information about $p$-torsion invariants of curves.

\subsection{The case $g=2$} 

The existence of supersingular curves of genus $2$ follows from
Serre's work on maximal curves in \cite[Th\'eor\`eme 3]{serre82}.

There exists a smooth curve of genus $2$ over $\overline{\FF}_p$ that is superspecial if and only if $p \geq 5$.
This is a special case of \cite[Proposition~3.1]{iko},
in which the authors determine the number of superspecial curves of genus $2$ for every $p$.

Every p.p.\ abelian surface is either the Jacobian of a curve of genus $2$, or the product of two elliptic curves, with the product polarization.
By avoiding the latter, one can check that all 3 Newton polygons and all 4 Ekedahl-Oort types occur for Jacobians of smooth curves
of genus $2$ over $\overline{\FF}_p$ for all $p$, except for the superspecial E--O type when $p=2,3$.

One parameter families of supersingular curves of genus $2$ were studied by Moret-Bailly and Pieper made this construction explicit 
for every odd $p$ in \cite{pieperg2}. 

\subsection{The case $g=3$} 

Every p.p.\ abelian threefold is either the Jacobian of a smooth curve of genus $3$, or the product of an 
elliptic curve and an abelian surface, with the product polarization.
It is now more difficult to avoid the latter; 
for example, there are $1$-dimensional families of p.p.\ abelian threefolds with $p$-rank $0$ that decompose as 
 the product of a supersingular elliptic curve and a supersingular abelian surface, with the product polarization.

However, it is still possible to confirm that all 5 Newton polygons and all 8 Ekedahl-Oort types occur for 
Jacobians of smooth curves over $\overline{\FF}_p$, except for the following cases:
when $p=2$, there does not exist a smooth curve of genus $3$ which is superspecial, or 
one with $p$-torsion group scheme $I_{1,1} \oplus I_{2,1}$ (E--O type $[0,0,1]$).

Here are some references for the $p$-rank $0$ cases, 
found in the four bottom rows of the $g=3$ table in Section~\ref{SStablesNPEO}.
There exists a smooth curve $C$ of genus $3$ over ${\overline \FF}_p$ such that $\mathrm{Jac}(C)$ has the given $p$-torsion group scheme:
\begin{enumerate}
\item $I_{3,1}$, for all $p$ by \cite[Theorem~5.12(2)]{O:hypsup};
\item $I_{3,2}$, \cite[Lemma 4.8]{Pr:large} for $p \geq 3$ and \cite[Example 5.7(3)]{EP13} for $p=2$;
\item $I_{1,1} \oplus I_{2,1}$, \cite[Lemma 4.8]{Pr:large} for $p \geq 3$ 
(using \cite[Proposition~7.3]{O:strat});\\ 
when $p =2$, this $2$-torsion group scheme does not occur for a hyperelliptic curve by \cite{EP13}
or for a smooth plane quartic by \cite{stohrvoloch}.
\item $(I_{1,1})^3$, if and only if $p \geq 3$ by \cite[Theorem~5.12(1)]{O:hypsup}.
\end{enumerate}

For a constructive approach to supersingular curves of genus $3$, see \cite{piepertheta}.

\subsection{The case $g=4$} 

The following result was proven by Harashita, Kudo, and Senda.
The construction of the proof uses curves that admit two commuting automorphisms of order $2$.
 
\begin{theorem} \cite[Corollary~1.2,1.3]{khs20} \label{Tkhs}
For every prime $p$, there exists a smooth curve of genus $4$ that is supersingular and has $a$-number at least $3$.
\end{theorem}

There is also a geometric approach to Question~\ref{Qmot2general} which is discussed in more detail in Section~\ref{S2BNPmain}.
Using the geometric approach, one can show for all $p$ that there exist smooth curves of genus $4$ over $\bar{\FF}_p$
with these Newton polygons:

$G_{3,1} \oplus G_{1,3}$ with slopes $1/4, 3/4$, by \cite[Corollary~5.6]{APgen}; and

$(G_{2,1} \oplus G_{1,2}) \oplus G_{1,1}$ with slopes $1/3, 1/2, 2/3$, by \cite[Corollary~4.1]{Pssg=4}; and

$(G_{1,1})^4$ (supersingular), by \cite[Corollary~1.2]{Pssg=4}, see Theorem~\ref{T4supersingular}.

\smallskip

For $g=4$, there are $16$ symmetric $\mathrm{BT}_1$ group schemes of rank $p^8$; 
see the table in \cite[Section~4.4]{Pr:sg}. 
For those with $p$-rank $0$ and $a$-number at least two, it is not known whether 
they occur as the $p$-torsion for a smooth curve of genus $4$ for most $p$.
Specifically, the open cases (expressed using the Young type $\mu$) are:
\begin{equation} \label{Eg4prank0}
\{4\}, \ \{4,1 \}, \ \{4,2 \}, \ \{4,3 \}, \{4,2,1 \}, \ \{4,3,1 \}, \{4,3,2 \}, \ \{4,3,2,1 \}.
\end{equation}

Here are some cases in which the answer is known:
\begin{enumerate}
\item \cite[Theorem~1.2]{zhou20} If $p$ is odd with $p \equiv \pm 2 \bmod 5$, Zhou proved the answer is yes for 
the Young types $\{4,2\}$ and $\{4,3\}$.
\item \cite[Theorem~1.2]{zhou20} If $p \equiv 4 \bmod 5$, there exists a superspecial curve of genus $4$
(Young type $\{4,3,2,1\}$).
\item \cite[Theorem~1.1]{kudoharashitahowe} if $7 < p < 20,000$ or $p \equiv 5 \bmod 6$, 
there exists a superspecial curve of genus $4$.
\item \cite[Corollary~6.6]{dusan1} If $p=2$, Dragutinovic proved that 
the answer is yes for $\{4\}$, $\{4,1 \}$, and $\{4,2 \}$ (and the strata for these curves have the right dimension);
and the answer is no for the other strata in \eqref{Eg4prank0}.
Similar results for $p=3$ are in \cite[Proposition~6.3]{dusan2}.
\end{enumerate}

\section{Related results, especially for Artin--Schreier curves}

The situation is quite different for curves that have a wildly ramified automorphism $\phi$.
For example, this includes the Artin--Schreier curves from Definition~\ref{Dartinschreier}, which are
curves that admit a degree $p$ cyclic cover of $\PP^1$.
We include a small selection of the many results about Newton polygons of Artin--Schreier curves.

More generally, suppose $\pi: C_1 \to C_2$ is a Galois cover of curves with Galois group $\ZZ/p\ZZ$ 
such that $p$ divides at least one of the ramification indices.
In this context, the wild Riemann--Hurwitz formula \cite[IV]{Se:lf} determines the genus of $C_1$ in terms of the genus of $C_2$ and 
the ramification jumps.
Also, the Deuring--Shafarevich formula \cite[Theorem~4.2]{subrao} determines the $p$-rank of $C_1$ 
in terms of the $p$-rank of $C_2$ and the ramification jumps.

\subsection{Hermitian curves are supersingular} \label{Shermitian}

The Hermitian curve $H_q$ is the curve in ${\mathbb P}^2$ defined by the affine equation $y^q+y=x^{q+1}$.
Because $H_q$ is a smooth plane curve of degree $q+1$, the genus of $H_q$ is $g=q(q-1)/2$.

\begin{proposition} \label{Lhermitian} \cite[VI 4.4]{sti09}, \cite[Proposition~3.3]{HansenDL}
The Hermitian curve $H_q$ is maximal over $\FF_{q^2}$
(meaning $\#H_q(\FF_{q^2})$ realizes the upper bound in the Hasse--Weil bound).  Also $L(H_q/\FF_q, T)=(1+qT^2)^g$ and $H_q$ is supersingular.
\end{proposition}

There are results similar to Proposition~\ref{Lhermitian} for Suzuki curves when $p=2$ and for Ree curves when $p=3$.

\subsection{Non-existence of superspecial curves}

Currently, there is one non-existence result known about Question~\ref{Qmot2}.
By Definition~\ref{Dordss}, a curve $C$ is superspecial when $a=g$
or, equivalently, when $\mathrm{Jac}(C)[p]$ is isomorphic to $(I_{1,1})^g$.

\begin{theorem} \label{Tekedahl} \cite{ekedahl87}, see also \cite{Baker}
If $C/\overline{\FF}_p$ is a superspecial curve of genus $g$, then $g \leq p(p-1)/2$.
\end{theorem}

By Theorem~\ref{Tekedahl}, if $g > p(p-1)/2$, then a smooth curve of genus $g$ defined over 
$\overline{\FF}_p$ cannot be superspecial.
The Hermitian curve $H_p$ is superspecial and its genus 
realizes the bound in Theorem~\ref{Tekedahl}.

The superspecial condition is equivalent to $a=g$ (or equivalently, $V=0$).
In \cite{Re}, Re generalized Theorem~\ref{Tekedahl}, giving a bound on the genus when 
the $a$-number is large relative to $g$ or when $V^r=0$ for some small $r$.

\subsection{Artin--Schreier curves} \label{Sartinschreier}

In characteristic $p=2$, there are supersingular curves of every genus.

\begin{theorem}\label{Tasss2} \cite[Theorem~2.1]{VdGVdV}
If $p=2$ and $g \geq 1$, then there exists a supersingular curve $C$ of genus $g$ defined over $\mathbb{F}_2$.
\end{theorem}

\begin{example}
It is possible that a Newton polygon may occur for a smooth curve in some characteristics but not in others.
When $p=2$, the Newton polygon of the curve $y^2+y = x^{23}+x^{21}+x^{17}+x^7+x^5$ 
has slopes $5/11, \ 6/11$.  
When $p=2$, the Newton polygon of the curve $y^2+y=x^{25}+x^9$ has slopes $5/12, \ 7/12$.
It is not known whether these Newton polygons occur for curves in any odd characteristic.
See \cite[Expectation~8.5.3]{oortpadova05}.
\end{example}

For every odd characteristic $p$, and any fixed natural number $N$, 
there are supersingular curves whose genus exceeds $N$.

\begin{theorem} \label{Tasss} \cite[Theorem~13.7]{VdGVdV92}, see also \cite[Corollary~3.7(ii)]{blache}, 
and \cite[Proposition~1.8.5]{bouwzeta}
If $\FF_q$ is a finite field of characteristic $p$ and
$R(x) \in \FF_q[x]$ is an additive polynomial of degree $p^h$, then $C :y^p-y=xR(x)$ is supersingular with genus $p^h(p-1)/2$.
\end{theorem}

\subsection{More genera for supersingular Artin--Schreier curves}

We take this opportunity to fix a mistake in \cite[Corollary~2.6]{priesCurrent}, where the genus was erroneously 
computed as $g=\delta p(p-1)^2/2$ because we did not take the projectivization of a vector space. 

\begin{corollary} \label{CKP} [Karemaker/Pries]
Let $p$ be prime.  
Let $\delta \geq 1$ be an integer such that $0$ and $1$ are the only coefficients in the base $p$ expansion of $\delta$.
If $g=\delta p(p-1)/2$, then there exists a supersingular curve of genus $g$ defined over a finite field of characteristic $p$.
\end{corollary}

When $p=2$, note that Corollary~\ref{CKP} yields every positive integer for the genus, like in Theorem~\ref{Tasss2},
because the condition on $\delta$ is vacuous and $g=\delta$.

\begin{proof} 
The condition on $\delta$ implies that, for some integer $t \geq 1$,  
\begin{eqnarray*} \label{Gsum}
\delta & &=  \sum_{i=1}^t p^{s_i} (1 + p + \cdots p^{r_i}), \\
& & \text{ for some } 
r_i, s_i \in \ZZ^{\geq 0} \text{ such that } s_i \geq s_{i-1} +r_{i-1}+2.
\end{eqnarray*}

Let $u_i =(s_i+1)-\sum_{j=1}^{i-1} (r_j+1)$ and note $u_{i+1} \geq u_i + 1$.

Choose an $\FF_p$-linear subspace $L_i$ of dimension $d_i:=r_i+1$ in the vector subspace of 
$\overline{\FF}_p[x]$ of additive polynomials of degree $p^{u_i}$, with the requirement that
$L_i \cap L_j = \{0\}$ if $i \not = j$.
Let ${\mathbb L}= \oplus_{i=1}^t L_i$.  

For $f \in {\mathbb L}-\{0\}$, let $C_f:y^p-y=xf$.  By definition, $C_f$ comes equipped with
a preferred map $C_f \to \PP^1$.
If $f \in {\mathbb L}-\{0\}$ is such that it has a non-zero component in $L_i$, 
but not in $L_j$ for $j >i$, then $g_{C_f} =p^{u_i}(p-1)/2$. 
By Theorem~\ref{Tasss}, $\mathrm{Jac}(C_f)$ is supersingular.

Let $\PP({\mathbb L})$ denote the projectivization of the $\FF_p$-vector space $\mathbb{L}$.
Specifically, there is a diagonal embedding of $\FF_p^*$ in $\mathbb{L}$.  If $f_1, f_2 \in \mathbb{L}-\{0\}$, and
if $f_1 = c f_2$ for some $c \in \FF_p^*$, then the curves $C_{f_1}$ and $C_{f_2}$ are isomorphic over $\FF_p$, 
and this isomorphism is compatible with the preferred maps to $\PP^1$.
With some abuse of notation, we write $f \in \PP(\mathbb{L})$ to denote an equivalence class of 
$f \in {\mathbb L} -\{0\}$ up to scaling by constants in $\FF_p^*$ and we write 
$C_f$ for $f \in \PP(\mathbb{L})$ to denote the curve $C_f$ for any representative of $f \in \mathbb{L}-\{0\}$ in this equivalence class.

Let $Y$ be the fiber product of $C_f \to {\mathbb P}^1$ for all $f \in \PP({\mathbb L})$.
By \cite[Theorem~B]{kanirosen}, $\mathrm{Jac}(Y)$ is isogenous to $\oplus_{f \in \PP({\mathbb L})} \mathrm{Jac}(C_f)$.
So $\mathrm{Jac}(Y)$ is supersingular.
The genus of $Y$ is $g_Y = \sum_{f \in \PP({\mathbb L})} g_{C_f}$.

There are $p^{d_i}-1$ non-zero polynomials in $L_i$.
The number of $f \in {\mathbb L}$ which have a non-zero contribution from $L_i$, but not from $L_j$ for $j >i$, is 
$(p^{d_i} -1) \prod_{j=1}^{i-1} p^{d_j}$.  The number of equivalence classes of these $f$ in $\PP(\mathbb{L})$ is the 
quotient of this number by $p-1$.  Thus we obtain:
\begin{eqnarray*} 
g_Y & = & \sum_{i=1}^t \frac{(p^{d_i} -1)}{p-1} (\prod_{j=1}^{i-1} p^{d_j}) p^{u_i} (p-1)/2\\
& = & \sum_{i=1}^t (p^{r_i} + \cdots + 1) p^{\sum_{j=1}^{i-1} (r_j+1)} p^{u_i-1} p(p-1)/2\\
& = & \sum_{i=1}^t (p^{r_i} + \cdots + 1) p^{s_i} p(p-1)/2 = \delta p(p-1)/2.
\end{eqnarray*}
\end{proof}

\subsection{Ekedahl--Oort types for Artin--Schreier curves}

First, consider the case $p=2$.
If $C$ is a hyperelliptic curve in characteristic $2$, then $C$ is an Artin--Schreier curve, with an affine equation of the 
form $y^2+y = f(x)$, for some $f(x) \in k[x]$.
The combination of $C$ being both Artin--Schreier and hyperelliptic places many constraints 
on the Frobenius action on its cohomology.
As a result, very few of the possible $a$-numbers and Ekedahl--Oort types occur for hyperelliptic curves 
in characteristic $2$ \cite{EP13}.  Similar restrictions occur for curves arising as ramified double covers 
of arbitrary curves in characteristic $2$ \cite{volochhyp2}.


Suppose $p$ is odd and $\pi: C_1 \to C_2$ is a Galois cover of curves with Galois group $\ZZ/p\ZZ$.
The relationship between the $a$-numbers (and the E--O types) of $C_1$ and $C_2$ is complicated, but there
are some constraints, e.g.\ \cite{boohercais} and \cite{caisulmer}.

\section{Open questions on supersingular curves and non-ordinary curves}

\subsection{Existence of supersingular curves} \label{Squesexistss}

For all primes $p$, for $g=1,2,3,4$, 
recall from Section~\ref{Smallg} that there exists a supersingular curve of genus $g$ over $\bar{\mathbb{F}}_p$.

\begin{question} \label{Qopenss}
For $p$ prime, and $g \geq 5$, does there exist 
a (smooth, projective, irreducible) curve $X$ of genus $g$ defined over $\bar{\mathbb{F}}_p$ that is supersingular?
\end{question}

Generalizing Question~\ref{Qopenss}, \cite[Conjecture~8.5.7]{oortpadova05} 
predicts that Newton polygons whose slopes have small denominators will
always occur for Jacobians of smooth curves.

When $p = 2$, the answer to Question~\ref{Qopenss} is yes for all $g$, see Theorem~\ref{Tasss2}.
For a fixed odd prime $p$, the answer is yes for infinitely many positive integers $g$, see Proposition~\ref{Lhermitian}, Theorem~\ref{Tasss},   
and Corollary~\ref{CKP}.

\begin{remark}
The first open situation for Question~\ref{Qopenss} is when $g=5$, 
under certain congruence conditions on $p$.
Here are the congruence conditions for which the answer to Question~\ref{Qopenss} is currently known to be yes for $g=5$:
for $p \equiv -1 \bmod 8, 11, 12, 20$ or for $p \equiv -1, -4 \bmod 15$ by \cite{ekedahl87};
for $p$ a quadratic non-residue modulo $11$ or for $p \equiv 11 \bmod 20$ by the Shimura--Taniyama formula 
(see \cite{LMPT1} or \cite{booherpriesproc}), or by \cite{yuiFermat, Aoki};
and for $p \equiv 3 \bmod 4$ by \cite{booherpries3mod4}.

For every $g \geq 6$, 
curves of genus $g$ can be constructed as cyclic covers of $\mathbb{P}^1$ branched at $\{0,1,\infty\}$, for certain degrees $m$,
and these are supersingular under certain congruence conditions on $p$ modulo $m$, 
leading to a positive answer to Question~\ref{Qopenss} for those pairs of $g$ and $p$.
\end{remark}

\subsection{Non-ordinary curves over finite fields} \label{Snonordstat}

We consider
the statistical behavior of the number of non-ordinary curves defined over a finite field.

We focus on the \emph{almost ordinary} case, defined as follows:
the almost ordinary Newton polygon is $\ooo^{g-1} \oplus \sss$, namely
the Newton polygon with $g-1$ slopes of $0$, two slopes of $1/2$, and $g-1$ slopes of $1$.
There is a unique Ekedahl--Oort type for $\ooo^{g-1} \oplus \sss$, which is $(\ZZ/p\ZZ \oplus \mu_p)^{g-1} \oplus I_{1,1}$.

\begin{question} \label{Qnumbercomponentsnonord}
For $p$ prime, $g \geq 4$, and $r \geq 1$,
what is the number of isomorphism classes of curves of genus $g$
with Newton polygon $\ooo^{g-1} \oplus \sss$ defined over $\FF_{p^r}$, in terms of $p$, $g$, and $r$?
\end{question}

The answer to this question is not even known for $g=4$. 
A curve $C$ is non-ordinary if and only if its Cartier--Manin matrix has determinant $0$.
Given an affine equation $f(x,y)$ for $C$,
the entries of this matrix are polynomials, depending on the coefficients for $f(x,y)$, whose degree
grows with $p$.
This makes it difficult to solve Question~\ref{Qnumbercomponentsnonord} computationally even in simple cases.

\begin{remark}
The dimension of the non-ordinary locus in $\cM_g$ is $3g-4$ for $g \geq 2$.
Because of this, for $p$ and $r$ sufficiently large, 
one expects that the answer to Question~\ref{Qnumbercomponentsnonord} grows asymptotically like 
$C p^{r(3g-4)}$, for some constant $C$ that depends on $g$ but not on $p$ or $r$.
By the Lang--Weil theorem, the value of $C$ gives information about one case of Question~\ref{Qnumbercomponents}.
\end{remark}

\begin{remark} Given a predetermined finite set of abelian varieties over $\FF_q$, 
the authors of \cite{elkieshoweR} give an upper bound on the genus of a curve over $\FF_q$
whose Jacobian decomposes into a product of powers of abelian varieties from this set.
\end{remark}

\begin{remark}
For one of the Arizona Winter School projects in 2024, a group studied the 
number of hyperelliptic curves with given $p$-rank \cite{AWSstats}.  
This answers a variation of Question~\ref{Qnumbercomponentsnonord} for the hyperelliptic locus.
\end{remark}

\chapter{Moduli spaces and the boundary} \label{C2Aboundary}

\section{Overview: moduli spaces of abelian varieties and curves}

Moduli spaces provide a powerful framework for studying families of abelian varieties and curves.

Let $g \geq 1$.
Section~\ref{Snotationmoduli} provides background about the moduli space $\cA_g$ of p.p.\ abelian 
varieties of dimension $g$.
By Theorem~\ref{TdimensionAgnew}, the dimension of $\cA_g$ is $g(g+1)/2$.
Section~\ref{SmodulispaceM}
provides background about the moduli space $\cM_g$ of smooth curves of genus $g$. 
By Theorem~\ref{TdimensionMg}, the dimension of $\cM_g$ is $3g-3$ for $g \geq 2$.

Section~\ref{Storellimorphism} is about the Torelli morphism $\tau_g : \cM_g \to \cA_g$, which takes the isomorphism class of 
a curve of genus $g$ to the isomorphism class of its Jacobian.
By a dimension count, for $g \geq 4$, we can see that most p.p.\ abelian varieties of dimension $g$ are not Jacobians of curves. 

Section~\ref{Sboundary} is about the Deligne--Mumford compactification of $\cM_g$ 
and the boundary $\partial \cM_g$, whose
points represent singular stable curves of genus $g$.  

In Section~\ref{Scomplete}, we describe some results about compact subvarieties of $\cA_g$ and $\cM_g$.
We end in Section~\ref{Sopencompact} with an open question about the maximal dimension of a compact subvariety of $\cM_g$.

\section{Dimensions of moduli spaces and the Torelli morphism} \label{Smodulispace}

\subsection{Moduli spaces of abelian varieties} \label{Snotationmoduli}

\begin{notation}
For $g \geq 1$, let $\cA_g$ denote the moduli space of principally polarized abelian varieties of dimension $g$.
\end{notation}

\begin{theorem} \label{TdimensionAgnew} \cite[Section~I.4]{faltingschaidegeneration}
The moduli space $\cA_g$ is a smooth Deligne--Mumford stack over $\ZZ$ 
with geometrically integral fibers of dimension $g(g+1)/2$.
\end{theorem}

(This is an algebraic version of Theorem~\ref{TdimensionAg},
whose proof used that $\cA_{g, \CC}$ is a quotient of the Siegel upper half-space.)

For more details on the definition and properties of $\cA_g$, we refer the reader to 
Section~1.10 of the prelude of this volume and \cite[Section~I.4]{faltingschaidegeneration}.
There is a tautological abelian variety $\mathcal{X}_g$ over the moduli stack $\cA_g$.

For this manuscript, the reader will lose little by replacing $\cA_g$ by its coarse moduli space, 
which is a quasiprojective scheme. 
There are several constructions of the coarse moduli space for $\cA_g$ in \cite{mumfordfogarty},
using geometric invariant theory, covariants of points, and theta constants.
Alternatively, one can add sufficiently fine level structure so that the moduli
problem is representable by a scheme. (We note that a
monodromy calculation is then necessary to show that an irreducible locus stays
irreducible after level structure is introduced.)

\begin{remark} \label{NcompactAg}
The moduli space $\cA_g$ is not compact, because there are families of abelian varieties that 
specialize to semi-abelian varieties.  
Let $\tilde{\cA_g}$ be a smooth toroidal compactification of $\cA_g$ from \cite[Chapter IV]{faltingschaidegeneration}.\footnote{
There are many results about different compactifications of $\cA_g$ that we do not cover here;
a good reference is the survey of Hulek and Tommasi \cite{hulektommasi}.}
The points of $\tilde{\cA_g}$ represent principally polarized semi-abelian varieties of dimension $g$.  
In fact, $\cA_g$ is open and dense in $\tilde{\cA}_g$.
This material is not needed in this manuscript.
\end{remark}

\subsection{Moduli spaces of curves} \label{SmodulispaceM}

\begin{notation}
For $g \geq 1$ and $r \geq 0$ (with $r \geq 1$ when $g=1$), 
let $\cM_{g;r}$ denote the moduli space of (smooth, projective, irreducible) curves of genus $g$ with $r$ labeled points.
For $g \geq 2$, let $\cM_g = \cM_{g;0}$.
\end{notation}

\begin{theorem} \label{TdimensionMg}  \cite[Theorem~5.2]{delignemumford}, \cite[Theorem~2.7]{knudsen2}
The moduli space $\cM_{g;r}$ is a smooth Deligne--Mumford stack over $\ZZ$ with geometrically integral fibers.
It has dimension $3g-3+r$ for $g \geq 2$, and dimension $r$ for $g=1$ and $r \geq 1$. 
\end{theorem}

For more details on the definition and properties of $\cM_{g;r}$, 
we refer the reader to \cite{knudsen2}.  There is a tautological curve $\cC_g$ over the moduli stack
$\cM_g$ \cite[Theorem~2.4]{knudsen2}.

Again, the reader will lose little by replacing $\cM_{g:r}$ by its coarse moduli space, 
which is a quasiprojective scheme, studied in numerous references e.g.\
\cite{mumfordPicard}, \cite[Section~5]{delignemumford}, \cite{mumfordfogarty}, and \cite[Chapter XII]{ACGH2}.
The construction of this scheme using geometric invariant theory works in arbitrary characteristic.
Again, one can add sufficiently fine level structure so that the moduli
problem is representable by a scheme.

A key point in proving Theorem~\ref{TdimensionMg} is that every relative stable curve $C$ of genus $g$ 
can be realized as a family
of curves in $\mathbb{P}^{5g-6}$ with Hilbert polynomial $P_g(n)=(6n-1)(g-1)$.
For a descriptive proof, see \cite[VII, Section~2]{miranda}.  Another proof uses deformation theory and the Riemann--Roch theorem.

\begin{remark}
We occasionally mention the `moduli space'
$\mathcal{H}_g$ of (smooth, projective, irreducible) hyperelliptic curves of genus $g$, 
viewed as a subspace of $\cM_g$.  It is irreducible and has dimension $2g-1$.  Intuitively, this is because a hyperelliptic curve is 
determined by the $2g+2$ branch points of its hyperelliptic cover to $\mathbb{P}^1$, and three of these 
branch points can be moved to $0$, $1$, and $\infty$.
\end{remark}

\subsection{The Torelli morphism} \label{Storellimorphism}

Recall the material from Section~\ref{SJacobianabelvar}.
By \cite[Section~8]{milneJacobian}, the Jacobian can be constructed for a relative curve of genus $g$.
The Torelli morphism $\tau_g: \cM_g \to \cA_g$ takes the isomorphism class of a curve $C$ to the isomorphism class of 
its Jacobian $\mathrm{Jac}(C)$.

\begin{theorem} \label{Ttorelli2} Torelli's Theorem, see \cite[Section~7.4]{mumfordfogarty}
If $k$ is an algebraically closed field, then the Torelli map $\tau_g: \cM_g(k) \to \cA_g(k)$ is injective.
\end{theorem}

\begin{definition} \label{Dtorelli}
The \emph{open Torelli locus} $\cT_g^\circ$ is the image of $\cM_g$ under $\tau_g$.
\end{definition}

The points in the open Torelli locus $\cT_g^\circ$ represent p.p.\ abelian varieties of dimension $g$ that are Jacobians of smooth curves.

\begin{remark}
See e.g., \cite[Section~4.3]{Landesman}.
For $g \geq 3$, the Torelli map $\tau_g$ is generically two to one (not an immersion) as a morphism of stacks, 
because a generic curve has no automorphisms, while an abelian variety always has the automorphism $[-1]$ of order $2$.
Also, the Torelli map is ramified on the hyperelliptic locus. 
\end{remark}

\section{The boundary of the moduli space of curves}  \label{Sboundary}

The moduli space $\cM_g$ is not compact, because there are families of smooth curves that specialize to singular curves.
In this section, we describe the Deligne--Mumford compactification of $\cM_g$, the clutching morphisms,
and the components of the boundary of $\cM_g$.
A good reference for this topic is \cite[Section~3]{knudsen2}.

\subsection{Deligne--Mumford compactification of $\cM_g$} \label{Scompactmoduli}

\begin{notation} \label{NcompactMg}
For $g \geq 1$ and $r \geq 0$ (with $r \geq 1$ when $g=1$),
let $\bar\cM_{g;r}$ be the Deligne-Mumford compactification of $\cM_{g;r}$ 
as in \cite{delignemumford} and \cite{knudsen2}. 
For $g \geq 2$, let $\bar{\cM}_{g} = \bar{\cM}_{g;0}$.
\end{notation}

The points of $\bar{\cM}_{g}$ represent stable curves of genus $g$.
For $r \geq 1$, the points of $\bar{\cM}_{g;r}$ represent $r$-labeled stable curves of genus $g$,
see \cite[Def.\ 1.1, 1.2]{knudsen2}. 
Then $\bar{\cM}_g$, and $\bar{\cM}_{g;r}$ are
smooth proper Deligne-Mumford stacks over $\ZZ$ \cite[Theorem~2.7]{knudsen2}.  
In fact, $\cM_{g;r}$ is open and dense in $\bar\cM_{g;r}$.

\begin{definition} \label{Dtorelli2}
Let $\cM_g^{ct}$ denote the subspace of $\bar\cM_g$ whose points represent stable curves of compact type.
\end{definition}

The Torelli morphism extends to a morphism $\tau_g: {\cM}^{ct}_g \to \cA_g$.
It is no longer injective, as seen in Fact~\ref{Ftorellifalse}.

\begin{definition}
The \emph{closed Torelli locus} $\cT_g$ is the image of $\cM_g^{ct}$ under $\tau_g$.
\end{definition}

The points in the closed Torelli locus $\cT_g$ represent p.p.\ abelian varieties of dimension $g$ that are 
Jacobians of smooth curves, or products of Jacobians of smooth curves, with the product polarization.

\begin{remark}
For compactifications of $\mathcal{H}_g$ and other Hurwitz spaces,  
we refer the reader to \cite{EkedahlBoundary}, \cite{wewersthesis}, and \cite{AvritzerLange}.
\end{remark}

\subsection{Clutching maps} 

Given two curves (with labeled points), it is possible to clutch them together to obtain a 
singular curve of higher genus.
To set some notation, suppose $g_1, g_2, r_1, r_2$ are positive integers. 
There is a clutching map
\begin{equation} \label{Eclutch1}
\xymatrix{
\kappa_{g_1;r_1, g_2;r_2}:\bar\cM_{g_1;r_1}\times\bar\cM_{g_2;r_2} \ar[r] &\bar\cM_{g_1+g_2;r_1+r_2-2}.
}
\end{equation}
Suppose $s_1 \in \bar\cM_{g_1;r_1}$ is the moduli point of a
labeled curve $(C_1;P_1, \ldots, P_{r_1})$, and suppose $s_2 \in
\bar\cM_{g_2;r_2}$ is the moduli point of a labeled curve $(C_2;Q_1, \ldots,
Q_{r_2})$.  Then $\kappa_{g_1;r_1,g_2;r_2}(s_1,s_2)$ is defined to be the moduli
point of the labeled curve $(D; P_1, \ldots, P_{r_1-1},
Q_2, \ldots Q_{r_2})$, where the underlying curve $D$ has components 
$C_1$ and $C_2$, the sections $P_{r_1}$ and $Q_1$ are identified in an
ordinary double point, and this nodal section is dropped from the
labeling.  The clutching map is a closed immersion if $g_1\not = g_2$
or if $r_1+r_2 \ge 3$, and is always a finite, unramified map \cite[Corollary~3.9]{knudsen2}.

The Jacobian of the resulting curve $D$ is the product of the Jacobians of $C_1$ and $C_2$. 
Specifically, by \cite[Ex.\ 9.2.8]{BLR},
\begin{equation}
\label{eqblr}
\mathrm{Pic}^0(D) \cong \mathrm{Pic}^0(C_1) \times \mathrm{Pic}^0(C_2).
\end{equation}

\bigskip

Alternatively, given a curve with two labeled points, it is possible to clutch these points together to obtain a singular 
curve of higher genus.  To set some notation, suppose $g$ and $r$ are positive integers and $r \ge 2$. 
There is a
clutching map
\[
\xymatrix{
\kappa_{g;r}:\bar\cM_{g;r} \ar[r] &\bar\cM_{g+1;r-2}.
}
\]
If $s \in \bar \cM_{g;r}$ is the moduli point of a labeled curve 
$(C; P_1, \ldots, P_r)$ then $\kappa_{g;r}(s)$ is the moduli point of the
labeled curve $(\tilde{C}; P_1, \ldots, P_{r-2})$ where $\tilde{C}$ is obtained by identifying 
the sections $P_{r-1}$ and $P_r$ in an ordinary double point, and these 
sections are dropped from the labeling. 
The 
morphism $\kappa_{g;r}$ is finite and unramified \cite[Corollary~3.9]{knudsen2}.

In this situation, $\mathrm{Pic}^0(\tilde{C})$ is a semi-abelian variety but not an abelian variety.
By \cite[Example~9.2.8]{BLR}, $\mathrm{Pic}^0(\tilde{C})$ is an extension of the form
\begin{equation}
\label{eqblr0}
\xymatrix{
0 \ar[r] & T \ar[r] & \mathrm{Pic}^0(\tilde{C})  \ar[r] & \mathrm{Pic}^0(C) \ar[r] & 0},
\end{equation}
where $T$ is a one-dimensional torus.  
The toric rank of $\mathrm{Pic}^0(\tilde{C})$ is one more than the toric rank of $\mathrm{Pic}^0(C)$. 
The maximal projective quotient of $\mathrm{Pic}^0(\tilde{C})$ is the maximal quotient of $\mathrm{Pic}^0(\tilde{C})$ 
which is an abelian variety; 
the maximal projective quotients of $\mathrm{Pic}^0(\tilde{C})$ and $\mathrm{Pic}^0(C)$ are isomorphic.

\subsection{Components of the boundary}

The boundary of $\cM_g$ is $\partial \cM_g = \bar{\cM}_g - \cM_g$.
In most cases, $\partial \cM_g$ has multiple irreducible components.

The first boundary components we consider are labeled $\Delta_i$ for $1 \leq i \leq g-1$, where
the generic point of $\Delta_i$ represents a stable curve of compact type.

\begin{definition}
Let $1 \le i \le g-1$ and write $g_1=i$ and $g_2=g-i$.
Define $\Delta_i = \Delta_{i}[\bar\cM_g]$ to be the image of 
$\bar\cM_{i;1}\times\bar\cM_{g-i;1}$ under the morphism $\kappa_{i,1;g-i,1}$, with the reduced induced structure.  
\end{definition}

The generic geometric point of $\Delta_i$ represents a curve $D$ with two irreducible components $C_1$ and $C_2$,
having genera $g_1$ and $g_2$, that intersect in an ordinary double point.  Note that $\Delta_i$ and $\Delta_{g-i}$ are the
same substack of $\bar\cM_g$.
By \eqref{eqblr}, the Jacobian of the curve $D$ is a p.p.\ abelian variety that decomposes, 
together with the product polarization.

We also consider the boundary component $\Delta_0$.
A point of $\Delta_0$ represents a stable curve that is not of compact type.

\begin{definition}
Define $\Delta_0 = \Delta_0[\bar\cM_g]$ to be the image
of $\bar \cM_{g-1;2}$ under the morphism $\kappa_{g-1;2}$, with the reduced induced structure. 
\end{definition}

Then $\cM_g^{ct} = \bar{\cM}_g - \Delta_0$.
The generic geometric point of $\Delta_0$ represents an irreducible curve $\tilde{C}$
that self-intersects in an ordinary double point.
By \eqref{eqblr0}, the Jacobian of $\tilde{C}$ is a semi-abelian variety, 
rather than an abelian variety.

\begin{theorem} \label{Tboundarydivisor}
\cite[page~190]{knudsen2} 
The locus $\Delta_i$ is an
irreducible divisor in $\bar\cM_g$, and $\partial \cM_g$ is the union
of $\Delta_i$ for $0 \le i \le g/2$.
\end{theorem}

\section{Theorems about complete subvarieties} \label{Scomplete}

In this section, we detail the historical development of results about complete subvarieties of the moduli spaces 
$\cA_g$, $\cM_g$, and $\cM_g^{ct} = \bar{\cM}_g - \Delta_0$.
Many of the proofs use the structure of the Chow ring, which we do not cover here.

\begin{theorem} \label{Tdiaz}
\cite[Theorem~4]{diaz} (for positive characteristic, see \cite[page 412]{looijenga})
Suppose $g \geq 3$.  If $Z \subset \cM_g$ is complete, then $\mathrm{dim}(Z) \leq g-2$.
\end{theorem}

\begin{theorem} \cite[page 80]{diaz} \label{Tdiaz2}
Suppose $g \geq 3$.  If $Z \subset \cM_g^{ct}$ is complete, then 
$\mathrm{codim}(Z, \cM_g^{ct}) \geq g$ (so $\mathrm{dim}(Z) \leq 2g-3$).
\end{theorem}

\begin{theorem} \cite[Corollary~1.7]{V:cycles}
Suppose $g \geq 3$.  If $Z \subset \cA_g$ is complete, then $\mathrm{codim}(Z, \cA_g) \geq g$
(so $\mathrm{dim}(Z) \leq g(g-1)/2$).
\end{theorem}

Keel and Sadun solved a conjecture of Oort \cite[Conjecture~3.5]{geerOort99}.

\begin{theorem} \label{Tkeelsadun} \cite[Corollary~1.2, 1.2.1]{keelsadun}
For $g \geq 3$, there is no compact codimension $g$ subvariety of $\cA_{g, \CC}$, or of $\cM_{g,\CC}^{ct}$.
\end{theorem}

\begin{remark}
Both claims in Theorem~\ref{Tkeelsadun} are false in positive characteristic.
In characteristic $p >0$, 
the $p$-rank $0$ locus is compact and has codimension $g$ in both $\cA_{g,k}$ and $\cM_{g,k}^{ct}$, 
see Sections~\ref{Sdimstrata} and \ref{SFVdG}.
\end{remark}

Here is a recent result.  

\begin{theorem} \cite[Theorem~B]{GrMSMT}
The maximal dimension of a compact subvariety of $\cA_{g,\CC}$ is $g-1$ if $g \leq 16$, 
is $\lfloor g^2/16 \rfloor$ if $g \geq 16$ is even, and is $\lfloor (g-1)^2/16 \rfloor$ if $g \geq 17$ is odd. 
\end{theorem}

See \cite[Corollary~C]{GrMSMT} for additional results, including the maximal dimension of a compact subvariety of $\cM_{g,\CC}^{ct}$.

\section{Open question on dimensions of compact subvarieties} \label{Sopencompact}

The next result was proven over $\CC$, but the proof applies more generally.

\begin{theorem} \cite{diezharvey}
Suppose $\mathrm{char}(k) \not = 2,3$.
If $g \geq 3$, there exists a complete curve in $\cM_g$.
\end{theorem}

\begin{proof}
Construction: Take the elliptic curve $E: y_1^2=x_1^3-1$ and 
the genus two curve $X: y_2^2=x_2^6-1$.  Consider the double cover $\pi: X \to E$, given by
$\pi(x_2,y_2)=(x_2^2, y_2)$.
Then $\pi$ is branched above $(0,i)$ and $(0,-i)$.
Let $r$ be even.  
Choose points $Q_1, Q_2, \ldots, Q_r \in E$ such that $Q_i - Q_j$ is not a 2-torsion point for $1 \leq i < j \leq r$.
Let $W = \{(P+_E Q_1, P +_E Q_2, \ldots, P+_EQ_r) \mid P \in E\}$.  Note that $W \subset E^r-\Delta$ and $W \cong E$.
Let $T \subset X^r - \Delta$ be the set of points $\vec{x} = (x_1, \ldots, x_r)$ such that $\pi(x_i) = \pi(x_1) +_E Q_i$.
Then $T$ is complete and $\mathrm{dim}(T) = 1$.  

Take $r=2(g-3)$.
For each $t \in T$, consider the cover $Z_t \to X$ branched at the $r$ coordinates of $t$.
By the Riemann--Hurwitz formula, $Z_t$ has genus $g$.  
Since $\#\mathrm{Aut}(Z_t)$ is finite, 
$Z_t$ is isomorphic to at most finitely many $Z_{t'}$ for $t' \in T$.
Thus the image of $T$ in $\cM_g$ is a complete curve.
\end{proof}

\begin{question} \label{Qcomplete}
If $g \geq 4$, what is the maximum dimension of a compact subspace of $\cM_g$?
\end{question}

It is possible that the answer to Question~\ref{Qcomplete} depends on the characteristic.
The first open case of Question~\ref{Qcomplete} is $g=4$, because it is not known 
if there exists a complete surface in $\cM_4$. 

\begin{remark}
If $\mathrm{char}(k) \not = 2$, Choi proved that $\cM_g$ contains a complete
subvariety of dimension $d$, for $g \geq 3 \cdot 2^{d-1}$, \cite{Choi}, generalizing work of \cite{Zaal}.
\end{remark}

\chapter[Intersection]{Intersection of the Torelli locus with $p$-torsion strata} \label{C2Bboundary}

\section{Overview: $p$-torsion stratifications of moduli spaces}

Recall that the \emph{$p$-torsion invariants} refer to the invariants of the 
$p$-rank, $a$-number, Newton polygon, and Ekedahl Oort type of a p.p.\ abelian variety defined over an 
algebraically closed field $k$ of positive characteristic $p$.
In this chapter, we take a geometric approach to the question of which $p$-torsion invariants occur for Jacobians of curves.

Let ${\mathcal A}_g$ denote the 
moduli space of principally polarized abelian varieties of dimension $g$ over $k$.
There are deep results about the stratifications of ${\mathcal A}_g$ by the $p$-torsion invariants.
However, there are not many results about how the open Torelli locus intersects these strata.
Many of the techniques used to study the stratifications on $\cA_g$ are not available on the Torelli locus. 
This includes techniques about deformation (Serre-Tate theory and Dieudonn\'e theory) and Hecke operators.

This leads to a geometric analogue of Question \ref{Qmot2}.
\begin{question} \label{Qmot3}
Given a prime $p$ and an integer $g \geq 4$, does the open Torelli locus $\cT^\circ_g$ intersect the stratum of 
$\cA_g$ parametrizing p.p.\ abelian varieties 
of a given $p$-rank, $a$-number, Newton polygon, or Ekedahl-Oort type? If so, what are the geometric properties of the intersection?
\end{question}

Section~\ref{S2BNPbackground} provides information about the stratifications of ${\mathcal A}_g$
by the $p$-torsion invariants.  Some facts are that the Newton polygon can only go up (become less ordinary) under specialization, 
and there is a purity result about the sublocus where the Newton polygon goes up.
The dimensions of the $p$-torsion strata in $\cA_g$ are known.
In Section \ref{Sunlikely}, we describe how finding curves with an unusual Newton polygon can be viewed as an unlikely
intersection problem. 

Section~\ref{S2BNPmain} contains 
several theorems about the geometry of the $p$-torsion stratifications of the Torelli locus.  This leads to
a geometric proof of
the existence of supersingular curves of genus $4$ for every prime $p$ \cite{Pssg=4};
the dimensions of the $p$-rank strata of $\cM_g$ \cite{FVdG};
and an inductive result, similar in spirit to earlier results in the literature,
but which provides more flexibility with the Newton polygon and Ekedahl-Oort type \cite{priesCurrent}.
The proofs of these results rely on information about the boundary $\partial \cM_g$.

Section~\ref{Srelatedstrata} contains some corollaries that emerge from Section~\ref{S2BNPmain}, 
along with some material about a mass formula for the number of non-ordinary curves in a one-dimensional family.
In Section~\ref{Squestionstrata}, we include some open questions about the geometry of the $p$-torsion strata in $\cM_g$, 
about the number of irreducible components of the $p$-rank strata and the generic Newton polygon 
on the $p$-rank $0$ stratum.

\section{Theorems about $p$-torsion stratifications of $\cA_g$} \label{S2BNPbackground}

\subsection{Specialization and purity} \label{Spurity}

This section includes two major facts about the behavior of Newton polygons in families.

The first is that the Newton polygon can only go up (become less ordinary) under specialization.
Specifically, building on Grothendieck's specialization theorem, Katz proved the following:

\begin{theorem} \label{katzslope} \cite[Corollary~2.3.2]{katzslope}
If $A$ is an $\FF_p$-algebra,
the set of points in $\mathrm{Spec}(A)$ at which the Newton polygon (of an $F$-crystal) goes up is Zariski-closed, and is locally on 
$\mathrm{Spec}(A)$ the zero-set of a finitely generated ideal.
\end{theorem}

Theorem~\ref{katzslope} provides a way to study Newton polygons in families.  This was used by Koblitz in \cite{koblitz75}.

The second major fact is the purity theorem for Newton polygons proved by de Jong and Oort.
Here is the statement.

\begin{theorem} \label{Tpurity} Purity theorem \cite[Theorem~4.1]{oortpurity}
Let $(A, m_A)$ be a Noetherian local ring of characteristic $p$. 
Let $G$ be a $p$-divisible group (or more generally, an $F$-crystal) over $\mathrm{Spec}(A)$. 
Assume that the Newton polygon of $G$ is constant over $\mathrm{Spec}(A) \backslash \{m_A \}$. 
Then either $\mathrm{dim}(A) < 1$ or the Newton polygon of $G$ is constant over $\mathrm{Spec}(A)$.
\end{theorem}

In practice, the purity theorem is used as follows.  
Suppose $X \to W$ is an abelian scheme of relative dimension $g$ defined over a reduced and irreducible scheme $W$.
Given a point $w \in W$, let $X_w$ denote the fiber of $X \to W$ over $w$.

\begin{corollary} \label{Cpurity2}
Suppose that $\xi$ is the Newton polygon of $X_\eta$ over the generic geometric point $\eta$ of $W$.
Consider the sublocus $L$ of points $w \in W$ for which the Newton polygon of $X_w$ is not $\xi$. 
If $L$ is not empty, then it has codimension $1$ in $W$.
\end{corollary}

More generally, suppose $\xi, \xi'$ are symmetric Newton polygons of height $2g$ with $\xi' < \xi$ in the partial ordering. 
Let $d(\xi', \xi)$ be the distance between $\xi'$ and $\xi$ in the 
partially ordered graph of symmetric Newton polygons of height $2g$. 
Then: 

\begin{corollary} \label{Cpurity3}
With hypotheses as above: if $\xi'$ is a symmetric Newton polygon with $\xi' < \xi$, then
the sublocus of points $w \in W$ for which the Newton polygon of $X_w$ is $\xi'$ 
is either empty or has codimension at most $d(\xi', \xi)$ in $W$.
\end{corollary}

The codimension might not be exactly $d(\xi', \xi)$ in Corollary~\ref{Cpurity3}
because some of the 
Newton polygons $\xi''$ between $\xi$ and $\xi'$ might not occur on $W$.

Currently, there is no analogue of Theorem~\ref{Tpurity} for the E--O stratification.

\begin{question} \label{QpurityEO}
Is there a purity result for the E--O stratification on $\cA_g$?
\end{question}

\subsection{Notation for the strata}

Let $g \geq 1$.  Let $\xi$ denote a $p$-torsion invariant 
(such as the $p$-rank, $a$-number, Newton polygon, or Ekedahl--Oort type) occurring for p.p.\ abelian varieties of dimension $g$.  

\begin{definition} \label{Ddefstrata}
Let $X \to W$ be an abelian scheme of relative dimension $g$ over
a (reduced, irreducible, smooth) Deligne--Mumford stack $W$. 
Define $W[\xi]$ to be the locally closed reduced substack of $W$ such that 
$w \in W[\xi](k)$ if and only if $\xi$ is the $p$-torsion invariant of $X_w$.
\end{definition}

The $p$-rank $f$ stratum is often denoted with a superscript:
e.g., $\cA_g^f$ (resp.\ $\cM_g^f$) denotes the
locally closed reduced substack of $\cA_g$ (resp.\ $\cM_g$)
whose points represent p.p.\ abelian varieties of dimension $g$ (resp.\ curves of genus $g$) with $p$-rank $f$.  

\subsection{Dimensions of the strata} \label{Sdimstrata} 

Here is some information about the dimensions of the $p$-torsion strata in $\cA_g$.
Recall that $\mathrm{dim}(\cA_g)=g(g + 1)/2$.  We include a partial list of valuable references.

\smallskip

{\bf (A) The $p$-rank strata:}  See Oort \cite{O:sub} and Norman--Oort \cite{normanoort}.

For $0 \leq f \leq g$, let $\cA_g^f$ denote the $p$-rank $f$ stratum
whose points represent p.p.\ abelian varieties of dimension $g$ and $p$-rank $f$.

\begin{theorem} 
\cite{normanoort} For $0 \leq f \leq g$, the $p$-rank stratum $\cA_g^f$ is non-empty and pure of codimension $g-f$ in $\cA_g$.
\end{theorem}

{\bf (B) Newton polygon strata:}
See Koblitz \cite{koblitz75}, Katz \cite{katzslope}, de Jong--Oort \cite{oortpurity}, Oort \cite{OortNPstrata}, and Chai--Oort  \cite{CO11}.

Let $\xi$ be a symmetric Newton polygon of height $2g$.
Consider the stratum $\cA_g[\xi]$ of $\cA_{g}$ 
whose points represent p.p.\ abelian varieties with Newton polygon $\xi$.
As in \cite[3.3]{OortNPformalgroups} or \cite[1.9]{OortNPstrata}, define 
$\mathrm{sdim}(\xi):=\# \Delta(\xi)$,
where
\[\Delta(\xi) = \{(x,y) \in \ZZ \times \ZZ \mid y < x \leq g, \ (x,y) \text{ on or above } \xi\}.\]

\begin{example} \label{Ecidim}
When $g=4$ and $\xi$ has slopes $1/4, 3/4$, then $\mathrm{sdim}(\xi)= 6$.
When $g=5$ and $\xi$ has slopes $2/5, 3/5$, then $\mathrm{sdim}(\xi) = 7$. 
\end{example}

\begin{theorem} \label{Tnpdim}
\cite[Theorem 4.1]{OortNPstrata} The dimension of the Newton polygon stratum $\cA_g[\xi]$ is 
$\mathrm{dim}(\cA_g[\xi]) = \mathrm{sdim}(\xi)$.
In particular, the supersingular locus has dimension $\mathrm{dim}(\cA_g[\sigma_g]) = \lfloor g^2/4 \rfloor$.
\end{theorem}

By \cite{CO11}, $\cA_g[\xi]$ is irreducible as long as $\xi$ is not the supersingular Newton polygon $\sigma_g$.
This implies that $\cA_g^f$ is irreducible, except when $g=1,2$ and $f=0$.

\smallskip

{\bf (C) Ekedahl-Oort strata:}
See Kraft \cite{Kraft}, Oort \cite{O:strat}, Ekedahl--Van der Geer \cite{EVdG}, and Moonen--Wedhorn \cite{moonenwedhorn}.

\begin{theorem} \label{Teodim} \cite[Theorem 1.2]{O:strat} The stratum of $\cA_g$ 
whose points represent p.p.\ abelian varieties with Ekedahl-Oort type $\nu=[\nu_1, \ldots, \nu_g]$
is locally closed and quasi-affine with dimension $\sum_{i=1}^g \nu_i$.
\end{theorem}

The closure relations for the Ekedahl--Oort strata are determined using information 
gleaned from the Weyl group of the symplectic group \cite{EVdG}.

\subsection{Unlikely intersections} \label{Sunlikely} 

The dimension of $\cA_g$ is $g(g + 1)/2$.
Its supersingular locus $\cA_g[\sigma_g]$ has dimension $\lfloor g^2/4 \rfloor$.
The difference $\delta_g:=g(g + 1)/2 - \lfloor g^2/4 \rfloor$ is the length of a maximal chain 
between the ordinary and supersingular Newton polygons in the partially ordered set of symmetric Newton polygons of height $2g$.

Oort observed the following in \cite[Expectation 8.5.4]{oortpadova05}.

\begin{remark} \label{Rlongchain}
If $g \geq 9$, then $\delta_g > 3g - 3 =\mathrm{dim}(\cM_g)$.
\end{remark}

Because of Remark \ref{Rlongchain}, for every $g \geq 9$ and every prime $p$, at least one of the following is true:
\begin{enumerate}
\item Either $\cM_g$ does not admit a perfect stratification by Newton polygon in characteristic $p$: this means that 
there are two symmetric Newton polygons $\xi_1$ and $\xi_2$ of height $2g$ such that 
$\xi_1 < \xi_2$ 
but $\cM_g[\xi_1]$ is not in the closure of $\cM_g[\xi_2]$;
\item or there is a symmetric Newton polygon of height $2g$ which does not occur for the Jacobian of a smooth curve of genus $g$
over $\bar{\FF}_p$.
\end{enumerate}

At this time, there is no example known of either (1) or (2) for any pair of $g$ and $p$.
In particular, no symmetric Newton polygon has been excluded from occurring for a Jacobian in any characteristic.

\begin{definition}
Let $\xi$ denote a symmetric Newton polygon of height $2g$.
The intersection of the open Torelli locus $\cT^\circ_g$ and the Newton polygon stratum $\cA_g[\xi]$ is \emph{unlikely} if 
$\mathrm{codim}(\cA_g[\xi], \cA_g) > 3g-3$.
Also, $\cT^\circ_g$ and $\cA_g[\xi]$ have an \emph{unlikely intersection} if their intersection is unlikely and also 
not empty.
\end{definition}

\begin{example}
For $g \geq 9$, the intersection of the open Torelli locus $\cT^\circ_g$ and the supersingular locus $\cA_g[\sigma_g]$ is unlikely.

To see this, note that $\mathrm{dim}(\cA_g) = g(g+1)/2$ (e.g.\ $45$ when $g=9$).
By Theorem~\ref{Tnpdim}, $\mathrm{dim}(\cA_g[\sigma_g]) = \lfloor g^2/4 \rfloor$ (e.g.\ $20$ when $g=9$).
Also $\mathrm{dim}(\cM_g) = 3g-3$ (e.g.\ $24$ when $g=9$).
The intersection of $\cT_g^\circ$ and $\cA_g[\sigma_g]$ is unlikely 
for $g \geq 9$ since $\mathrm{dim}(\cM_g) < \mathrm{codim}(\cA_g[\sigma_g], \cA_g)$.
\end{example}

\begin{remark}
This manuscript contains numerous examples that unlikely intersections do occur.  
By Section~\ref{Sartinschreier}, for fixed $p$, there exist supersingular curves over $\bar{\FF}_p$ 
whose genus is arbitrarily large.  For fixed $g$, under congruence conditions on $p$, 
there exist curves of genus $g$ that are supersingular (whose Jacobians have complex multiplication).
For example, the technique in Section~\ref{Scomplexmultiplication} shows that
the genus $9$ curve with equation $y^{19} = x(x-1)$ is supersingular if $p$ is not a quadratic residue modulo $19$.
Theorem~\ref{main-thm-ord} provides more examples of unlikely intersections of Newton polygon strata in $\cA_g$ with 
$\cT^\circ_g$ for $g$ arbitrarily large.  
\end{remark}

\section{Theorems about $p$-torsion stratifications of $\cM_g$} \label{S2BNPmain}

In this section, we describe several results about the geometry of the $p$-torsion strata of $\cM_g$, 
whose points represent smooth curves of genus $g$ whose Jacobians have a given $p$-torsion invariant.  

Section~\ref{Sg4ssnewproof} contains a geometric proof that
there exists a smooth curve of genus $4$ that is supersingular for every prime $p$ \cite{Pssg=4}. 
While this result was already known \cite[Corollary~1.2,1.3]{khs20} (see Theorem~\ref{Tkhs}), 
the geometric proof leads to some new applications, which are included in Section~\ref{Srealizing}.

Section~\ref{SFVdG} contains a proof of \cite[Theorem 2.3]{FVdG} by Faber and Van der Geer,
about the dimensions of the $p$-rank strata of $\cM_g$.  
This answers Question~\ref{Qmot3} for the $p$-rank strata.  
We highlight this result because it inspired a lot of research.

Section \ref{Sinductive} contains a result from \cite{priesCurrent}
which gives a method to reduce questions about the geometry of the 
Newton polygon and Ekedahl-Oort strata of $\cM_g$ to the case of $p$-rank $0$.

Recall that the open (resp.\ closed) Torelli locus $\cT^\circ_g$ (resp.\ $\cT_g$) 
is the image of $\cM_g$ (resp.\ $\cM_g^{ct}$) under the Torelli morphism $\tau_g: \cM_g^{ct} \to \cA_g$.  
The techniques in Section~\ref{S2BNPmain} make use of $\tau_g$ and $\partial \cM_g$.
Some complications are that, in most cases:
it is not known whether the $p$-torsion strata of $\cT^\circ_g$ are irreducible; the closure $\cT_g$ of $\cT^\circ_g$ in $\cA_g$
contains points representing Jacobians of singular curves of compact type; and $\tau_g$ is not flat, 
since its fibers have positive dimension over $\partial \cM_g$. 

\subsection{A geometric proof for supersingular genus $4$ curves} \label{Sg4ssnewproof}

This result was inspired by a conversation with Oort, in which we discussed
a geometric method for studying the Newton polygons $\xi$ that occur on $\mathcal{M}_g$. 
This method applies when $\mathrm{codim}(\mathcal{A}_g[\xi], \mathcal{A}_g)$ is small.  

As an illustration of this method, here is a new proof of \cite[Corollary~1.2]{khs20}.
Let $\mathcal{M}_g[\sigma_g]$ (resp.\ $\mathcal{A}_g[\sigma_g]$) denote the supersingular locus 
of $\mathcal{M}_g$ (resp.\ $\mathcal{A}_g$).

\begin{theorem} \label{T4supersingular}  \cite{Pssg=4}
For every prime $p$, there exists a smooth curve of genus $4$ that is supersingular. 
Thus $\mathcal{M}_4[\sigma_4]$ is non-empty and its irreducible components have dimension at least $3$ for every prime $p$.
\end{theorem}

Note that \cite[Theorem~1.1]{khs20} (whose proof uses a different technique)
also gives information about the $a$-number,
namely there exists a supersingular smooth curve of genus $4$ with $a$-number $a \geq 3$ for every prime $p > 3$.

\begin{proof}[Proof of Theorem~\ref{T4supersingular}]
Working over $\bar{\mathbb{F}}_p$, consider a chain of four supersingular elliptic curves, clutched together
at ordinary double points.
This yields a stable curve $C_\circ$ of compact type that has genus $4$ and is supersingular. 
By \eqref{eqblr}, the Jacobian $J_\circ$ of $C_\circ$ is isomorphic to the product of four supersingular elliptic curves and 
thus is supersingular.
As such, $J_\circ$ is represented by a point in ${\mathcal A}_4[\sigma_4] \cap \cT_4$, which shows that this intersection is non-empty. 
 
The codimension of $\mathcal{A}_4[\sigma_4]$ in $\mathcal{A}_4$ is $10-4=6$.
The codimension of $\cT_4 \cap \mathcal{A}_4$ in $\mathcal{A}_4$ is $10-9=1$.
Since $\mathcal{A}_4$ is a smooth stack, the codimension of an intersection of two substacks is at
most the sum of their codimensions \cite[page 614]{V:stack}.  Thus
$\mathrm{codim}(\mathcal{A}_4[\sigma_4] \cap \cT_4, \mathcal{A}_4) \leq 7$.  
To summarize, each of the irreducible components of the non-empty intersection $\mathcal{A}_4[\sigma_4] \cap \cT_4$ 
has dimension at least $3$.

Let $\Delta$ denote the locus in $\mathcal{A}_4[\sigma_4] \cap \cT_4$
whose points represent the Jacobian of a curve $C_s$ that is stable but not smooth.
By hypothesis, these points represent abelian varieties, so the curve $C_s$ has compact type.
From Section~\ref{Sboundary}, $\mathrm{Jac}(C_s)$ is a p.p.\ abelian fourfold that decomposes,
with the product polarization.  

Then $\mathrm{dim}(\Delta) \leq 2$.
This is because points in $\Delta$ represent abelian fourfolds either of the form (i) $E \oplus X$ where $E$ is a supersingular 
elliptic curve and $X$ is a supersingular abelian threefold, or of the form (ii) $X \oplus X'$ where $X,X'$ are supersingular 
abelian surfaces.  In case (i), 
$\mathrm{dim}(\Delta) \leq \mathrm{dim}(\mathcal{A}_1[\sigma_1] \times \mathcal{A}_3[\sigma_3]) = 2$.
In case (ii), $\mathrm{dim}(\Delta) \leq \mathrm{dim}(\mathcal{A}_2[\sigma_2] \times \mathcal{A}_2[\sigma_2]) =2$.
Since $2<3$, every generic geometric point of $\mathcal{A}_4[\sigma_4] \cap \cT_4$ represents 
the Jacobian of a supersingular curve of genus $4$ which is smooth.

Thus ${\mathcal M}_4[\sigma_4]$ is non-empty for every $p$; equivalently,
there exists a smooth curve of genus $4$ that is supersingular.
If $W$ is an irreducible component of 
${\mathcal M}_4[\sigma_4]$, then the image of $W$ under $\tau_g$ is open and dense 
in an irreducible component of $\mathcal{A}_4[\sigma_4] \cap \cT_4$.
Thus $\mathrm{dim}(W) \geq 3$ by purity (Corollary~\ref{Cpurity2}).
\end{proof}

\begin{remark}
One expects that the dimension of every irreducible component of $\cM_4[\sigma_4]$ is three.  
For $7 < p < 20,000$ or $p \equiv 5 \bmod 6$, this is true for at least one irreducible component of $\cM_4[\sigma_4]$ by 
\cite{harashita22}. 
It is true for every irreducible component of $\cM_4[\sigma_4]$ when $p=2,3$ by \cite{dusan1, dusan2}. 
\end{remark}

\subsection{Results about the $p$-rank stratification} \label{SFVdG} 

We describe a theorem of Faber and Van der Geer about the dimensions of the $p$-rank strata of $\cM_g$.

Fix a prime $p$ and integers $g \geq 2$ and $f$ such that $0 \leq f \leq g$. 

Recall that $\cA_g^f$ is irreducible unless $g=1,2$ and $f=0$.
In most cases, it is not known whether $\cM_g^f$ is irreducible. 

The definition of the $p$-rank in \eqref{Edefprank} applies also for stable curves.
The $p$-rank stratification extends to the moduli space $\bar{\cM}_g$ by \cite[Section~2]{FVdG}.

We denote by $\bar{\cM}_g^f$ the stratum of $\bar{\cM}_g$ 
of points that represent stable curves of genus $g$ and $p$-rank $f$.\footnote{
There are some technical details here that we are choosing to suppress so that the main ideas of the proof are more clear.
If $\bar{W}$ is the closure of $W$, there is a difference between $\bar{W}[\xi]$ and $\overline{W[\xi]}$.
For example, $\bar\cM_g^f:=(\bar \cM_g)^f$ is the $p$-rank $f$ stratum of $\bar \cM_g$, while
$\overline{(\cM_g^f)}$ is the closure of the $p$-rank $f$ stratum of $\cM_g$.
A consequence of the proof is that the former is contained in the latter.
For $f >0$, the containment is proper since the 
latter contains points representing curves whose $p$-rank is less than $f$.
More precisely, as part of the proof one can show that the generic geometric point of every irreducible component of 
$(\bar{\cM}_g)^{\leq f}$ represents a smooth curve whose $p$-rank is exactly $f$. 
We refer the reader to \cite{FVdG} and \cite{AP:mono} for more details.}

\begin{theorem} \label{TFVdG} \cite[Theorem 2.3]{FVdG}
Let $g \geq 2$.  Every irreducible component of $\bar{\cM}_g^f$ has dimension $2g-3+f$ 
(codimension $g-f$ in $\bar{\cM}_g$);
in particular, there exists a smooth curve over $\bar{\FF}_p$ with genus $g$ and $p$-rank $f$. 
\end{theorem}

By \eqref{eqblr}, it is easy to construct a \emph{singular} curve of genus $g$ with $p$-rank $f$, by constructing 
a chain of $f$ ordinary and $g-f$ supersingular elliptic curves, joined at ordinary double points.
This singular curve can be deformed to a smooth one, but it takes additional work to show that
the $p$-rank stays constant in this deformation.
To prove that there is a \emph{smooth} curve of genus $g$ with $p$-rank $f$, 
singular curves are still useful, but the argument must be made carefully.

\begin{proof} (Sketch of proof of Theorem \ref{TFVdG})
The proof is by induction on $g$.
When $g =2,3$, the result is true since $\cT_g^\circ$ is open and dense in $\cA_g$.

Suppose $g \geq 4$.
The dimension of $\bar{\cM}_g$ is $3g-3$.  There are singular stable curves that are ordinary, namely
chains of $g$ ordinary elliptic curves.  Since $\bar{\cM}_g$ is irreducible and 
the $p$-rank is lower semi-continuous (Theorem~\ref{katzslope}), the generic geometric point of $\bar{\cM}_g$
is ordinary, with $p$-rank $g$.

Let $\Gamma$ be an irreducible component of $\bar{\cM}_g^f$ in $\bar{\cM}_g$.
Consider the maximal chain which connects the ordinary Newton polygon $\mathrm{ord}^g$ to 
the maximal Newton polygon having $(f,0)$ as a break point; it has length $g-f$.
Using purity of the Newton polygon stratification \cite{oortpurity} (see Theorem~\ref{Tpurity}), 
\[\mathrm{dim}(\Gamma) \geq (3g-3) -(g-f)=2g-3+f.\]

Building on Theorem~\ref{Tdiaz2}, consider the intersection of $\Gamma$ with $\partial \cM_g$.
By \cite[Lemma 2.5]{FVdG}, $\Gamma$ intersects $\Delta_i$ for each $1 \leq i \leq g-1$.
By Theorem~\ref{Tboundarydivisor}, $\mathrm{codim}(\Delta_i, \bar{\cM}_g) = 1$.
Thus, by \cite[page 614]{V:stack}, 
$\mathrm{dim}(\Gamma) \leq \mathrm{dim}(\Gamma \cap \Delta_i) +1$.

Using \eqref{eqblr}, as seen in \cite[Proposition 3.4]{AP:mono}, $\Gamma \cap \Delta_i$ is the union of the images of the 
clutching morphisms (restricted to the $p$-rank strata) from Section~\ref{Sboundary}:
\begin{equation} \label{Ekapparestrictprank}
\kappa_{i;1,g-i;1}: \bar{\cM}^{f_1}_{i;1} \times \bar{\cM}^{f_2}_{g-i;1} \to \bar{\cM}^{f_1+f_2}_g,
\end{equation}
for pairs of non-negative integers $(f_1,f_2)$ with $f_1 + f_2 = f$.

So $\mathrm{dim}(\Gamma \cap \Delta_i) \leq 
\mathrm{dim}(\bar{\cM}^{f_1}_{i;1}) + \mathrm{dim}(\bar{\cM}^{f_2}_{g-i;1})$, 
for some pair $(f_1,f_2)$.
If $i \not = 1, g-1$, then adding a marked point adds one to the dimension. 
By the inductive hypothesis (or an explicit computation when $i=1,g-1$), 
$\mathrm{dim}(\bar{\cM}^{f_1}_{i;1})=2i - 3 + f_1 +1$
and $\mathrm{dim}(\bar{\cM}^{f_2}_{g-i;1})=2(g-i) -3 +f_2+1$.
It follows that $\mathrm{dim}(\Gamma \cap \Delta_i) \leq 2g - 4 + f$. 
Thus $\mathrm{dim}(\Gamma) \leq 2g-3 +f$, which completes the proof.
\end{proof}

\begin{remark}
In \cite{AP:mono}, there are additional results about how $\bar{\cM}_g^f$ intersects $\partial \cM_g$. 
This leads to applications about the $\ell$-adic monodromy of the irreducible components of $\bar{\cM}_g^f$, which are described in 
Section~\ref{Smonoprank}.  
\end{remark}

\begin{corollary} \cite[Corollary~4.4]{APgen}
Let $\Gamma$ be an irreducible component of $\cM_g^f$.
If $f < g$, then $\Gamma$ is in the closure of $\cM_g^{f+1}$ in $\cM_g$.
If $f >0$, then the closure of $\Gamma$ in $\cM_g$ contains an irreducible component of $\cM_g^{f-1}$.
\end{corollary}

\subsection{Reducing to the case of $p$-rank $0$} \label{Sinductive}

We describe a technique to reduce Question~\ref{Qmot3} to the case of $p$-rank $0$.
This uses an inductive process on $g$ found in earlier work, 
e.g., \cite[Theorem~2.3]{FVdG}, 
\cite[Section~3]{AP:mono}, \cite[Proposition~3.7]{Pr:large}, and \cite[Proposition~5.4]{APgen}.
Theorem~\ref{Tinductive} is stronger than these results because 
it provides more control over the Newton polygon and Ekedahl-Oort type; 
however, the hypotheses needed to apply the result can be hard to verify.

First, we fix some notation about Newton polygons and $\mathrm{BT}_1$ group schemes. 

\begin{notation} \label{Nstrata}
Let $\xi$ denote a symmetric Newton polygon of height $2g$ (or a symmetric $\mathrm{BT}_1$ group scheme of rank $2g$). 
Let $\cA_g[\xi]$ be the stratum in $\cA_g$ whose points represent p.p.\ abelian varieties of 
dimension $g$ and $p$-torsion invariant $\xi$.
Let $cd_\xi=  \mathrm{codim}({\mathcal A}_g[\xi], {\mathcal A}_g)$.
Let ${\mathcal M}_g[\xi]$ be the stratum in ${\mathcal M}_g$ whose points represent curves
of genus $g$ with $p$-torsion invariant $\xi$.
\end{notation}

\begin{notation} \label{Nadde}
Let $e$ be a positive integer.

When $\xi$ is a Newton polygon of height $2g$,
let $\xi\oplus \mathrm{ord}^{e}$ be the Newton polygon of height $2(g+e)$
obtained by adjoining $e$ slopes of $0$ and $e$ slopes of $1$ to $\xi$.

When $\xi$ is a $\mathrm{BT}_1$ group scheme of rank $2g$, 
let $\xi\oplus \mathrm{ord}^{e}$ be the $\mathrm{BT}_1$ group scheme of rank $2(g+e)$
obtained by adjoining $e$ copies of $L = \ZZ/p \oplus \mu_p$ to $\xi$.
\end{notation}

\begin{theorem} \label{Tinductive}  \cite[Theorem~6.4]{priesCurrent}
With notation as in \ref{Nstrata} and \ref{Nadde}, 
suppose that there exists an irreducible component $\Gamma=\Gamma_0$ of ${\mathcal M}_g[\xi]$ such that
$\mathrm{codim}(\Gamma, {\mathcal M}_g)=cd_\xi$.
Then, for all positive integers $e$, there exists an irreducible component $\Gamma_e$ of ${\mathcal M}_{g+e}[\xi\oplus \mathrm{ord}^{e}]$
such that $\mathrm{codim}(\Gamma_e, {\mathcal M}_{g+e})=cd_\xi$.
\end{theorem}

The proof uses the boundary component $\Delta_1$ from Section~\ref{Sboundary}.

\section{Related results on families of non-ordinary curves} \label{Srelatedstrata}

\subsection{$p$-ranks of curves with involutions} \label{Srealizing}

The next two results are proven using the strategy of Theorem~\ref{TFVdG}.

\begin{theorem} (for $p$ odd) \cite[Theorem 1]{glasspries}, see also \cite[Lemma 3.1]{AP:hyp}, 
(for $p=2$) \cite[Corollary 1.3]{prieszhu}
Every irreducible component of $\bar{\cH}_g^f$ has dimension $g-1+f$ 
(codimension $g-f$ in $\bar{\cH}_g$);
in particular, there exists a smooth hyperelliptic curve over $\bar{\FF}_p$ 
with genus $g$ and $p$-rank $f$.
\end{theorem}

\begin{remark}
For one of the Arizona Winter School projects in 2024, a group 
studied the $p$-rank stratification of double covers of a fixed elliptic curve \cite{AWSdoublecover}.
\end{remark}

\subsection{Realizing other $p$-torsion invariants for curves} \label{Srealizing}

We obtain more results about Question~\ref{Qmot3} by combining the methods described in Sections~\ref{Sg4ssnewproof}, 
\ref{SFVdG}, and \ref{Sinductive}.
Here are some applications.

\begin{corollary} \cite[Corollary~4.3]{Pssg=4}
For every prime $p$, every symmetric Newton polygon of height $2g$ having $p$-rank $f \geq g-4$ occurs on $\cM_g$.
\end{corollary}

\begin{corollary} \cite{dusanrachel}
For every prime $p$, there exists a smooth curve of genus $g$ with $p$-rank $0$ and $a$-number at least $2$.
\end{corollary}

See \cite[Corollary~6.4]{dusan2} for some additional results when $p=2$.

\subsection{Mass formula for non-ordinary curves} \label{Smassformula}

In this section, we describe another way to study the geometry of the non-ordinary locus of $\cM_g$.

Suppose $\xi$ is a $p$-torsion invariant such that $\cM_g[\xi]$ is non-empty,
and suppose that we know the dimension of (the irreducible components of) $\cM_g[\xi]$.
To fully answer Question~\ref{Qmot3}, we would like to know more about the geometry of $\cM_g[\xi]$.

One framework to address this question is with a mass formula.
Here we are motivated by the Eichler--Deuring mass formula which states that 
\begin{equation} \label{Eeichlerdeuring}
\sum_{[E] ss} \frac{1}{\#\mathrm{Aut}(E)}= \frac{p-1}{24}, 
\end{equation}
where the sum is over the isomorphism classes of supersingular curves $E$ over $\bar{\mathbb{F}}_p$.  

One generalization of \eqref{Eeichlerdeuring} is in \cite{YuYu} and \cite{KaremakerYobukoYu}, where the authors provide 
mass formulae for the number of supersingular p.p.\ abelian varieties of dimensions $2$ and $3$.

Generalizing in a different direction, we consider this question.

\begin{question}  \label{Qmassformula}
Given a one dimensional family $\mathcal{F}$ of curves of genus $g$, whose generic point represents an ordinary curve,
determine a mass formula for the number of non-ordinary curves in $\mathcal{F}$.
\end{question}

In \cite[Proposition~3.2, Theorem~3.3]{iko}, the authors essentially answer Question~\ref{Qmassformula} when $\mathcal{F}$
is the one dimensional family of curves of genus two that have an automorphism of order $3$ (resp.\ $4$).

\begin{remark} \label{CavPriesmoonen}
In \cite{cavalieripries}, Cavalieri and I answer Question~\ref{Qmassformula} when $\mathcal{F}$
is a one dimensional family of curves that are cyclic covers of $\mathbb{P}^1$.
We intersect $\mathcal{F}$ with the codimension one locus of non-ordinary curves, 
using cycle classes in the tautological ring of $\cA_g$.
The cycle class for the non-ordinary locus is $(p-1) \lambda_1$ where $\lambda_1$ is the first Chern class of the Hodge bundle.
This yields a mass formula for the number of non-ordinary curves in $\mathcal{F}$.

This mass formula depends on the intersection of $\mathcal{F}$ with $\partial \cM_g$.
In particular, we compute it when the covers in $\mathcal{F}$ are branched at $4$ points 
(namely, of the form $y^m=x^{a_1}(x-1)^{a_2}(x-t)^{a_3}$) \cite[Corollary~6.1]{cavalieripries} 
or when $\mathcal{F}$ is a family of hyperelliptic curves where only one branch point varies 
(namely, of the form $y^2=h(x)(x-t)$) \cite[Corollary~5.4]{cavalieripries}.
In \cite[Corollary~6.4]{cavalieripries}, we provide additional information for the special families found in Example~\ref{Mspecial}.
\end{remark}

\section{Open questions on $p$-torsion strata of $\cM_g$} \label{Squestionstrata}

Earlier chapters include several open questions about the existence of smooth curves having a given $p$-torsion type;
see Questions~\ref{Qmot2}, \ref{Qmot2general}, and \ref{Qopenss}. 
Question~\ref{QpurityEO} is about purity for the E--O stratification on $\cA_g$.

Question~\ref{Qmot3} is a general question about the geometry of the $p$-torsion strata of $\cM_g$.
We discuss refinements of Question~\ref{Qmot3} which may be more tractable.

\subsection{The generic Newton polygon for $p$-rank $0$ curves}

Let $g \geq 2$.  The $p$-divisible group $G_{g-1,1} \oplus G_{1,g-1}$, 
has Newton polygon slopes $1/g$ and $(g-1)/g$.  
This is the maximal symmetric Newton polygon of height $2g$ with $p$-rank $0$.  If it occurs for a smooth curve of genus $g$, then 
it is the generic Newton polygon for at least one irreducible component of the $p$-rank $0$ stratum $\cM_g^0$.

\begin{question} \label{Qslopesprank0}
Suppose $p$ is prime and $g \geq 4$.
Suppose $\Gamma$ is an irreducible component of the $p$-rank $0$ stratum $\cM_g^0$ of $\cM_g$.
Does the Newton polygon of the generic geometric point of $\Gamma$ have slopes $1/g$ and $(g-1)/g$?
\end{question}

Question~\ref{Qslopesprank0} is closely related to Question~\ref{Qanumber1}.
There are partial results for $g=4$ in \cite{APgen} and 
for $g=5,6$ under congruence conditions on $p$ in \cite{LMPT2}.

\subsection{Number of irreducible components of the strata} \label{Squestioncomponents}

If $\xi$ is a Newton polygon that is not supersingular,
then the locus $\cA_g[\xi]$ is irreducible \cite{CO11}.
Similarly, if $\xi$ is an Ekedahl--Oort type that is not fully contained in the supersingular locus $\cA_g[\sigma_g]$, then 
the locus $\cA_g[\xi]$ is irreducible \cite{EVdG}.

Suppose that $\cM_g[\xi]$ is non-empty, 
meaning that there exists a smooth curve of genus $g$ defined over $\bar{\FF}_p$ having $p$-torsion invariant $\xi$.
In this case, we can ask more refined questions about the geometry of $\cM_g[\xi]$, such as the number of its irreducible 
components and their dimensions.
In almost all cases, the number of irreducible components of $\cM_g[\xi]$ is not known.

\begin{question} \label{Qnumbercomponents}
Suppose $g \geq 4$.  Let $\xi$ denote a Newton polygon (resp. Ekedahl--Oort type) 
that is not ordinary, and that is not supersingular 
(resp.\ whose stratum is not fully contained in the supersingular locus $\cA_g[\sigma_g]$).
If $\cM_g[\xi]$ is not empty, then what is the number of its irreducible components?
\end{question}

The answer to Question~\ref{Qnumbercomponents} might depend on $p$.
This question might be most tractable for the $p$-rank $f$ stratum, because the dimension of $\cM_g^f$ is known ($2g-3+f$).
The answer is not known even for the case $f=g-1$.

\begin{example} \label{Enonordcomp}
The non-ordinary locus of $\cM_g$
is closed of codimension $1$ in $\cM_g$.
When $g=2$ (resp.\ $g=3$), this locus is irreducible.
For $g \geq 4$, the number of irreducible components of the non-ordinary locus of $\cM_g$ is not known. 
\end{example}

Question~\ref{Qnumbercomponentsnonord} is related to Example~\ref{Enonordcomp} which 
is about the order of magnitude of $\#\cM_g[\ooo^{g-1} \oplus \sss](\FF_{p^r})$, 
where $\ooo^{g-1} \oplus \sss$ is the almost ordinary Newton polygon.

\subsection{Intersection of Newton polygon strata and the boundary}

In Section~\ref{Squesexistss}, we mentioned \cite[Conjecture~8.5.7]{oortpadova05}, which is about the existence of smooth curves 
whose Newton polygons have slopes with small denominators.  Here we provide some underlying motivation for that conjecture.
The idea is to try to deform singular curves of compact type to smooth curves while staying within the same Newton polygon stratum.
It turns out to be difficult to do this and the first open case is when $g=3$.

\begin{question} \label{Qdeformg3}
Determine the intersection of the closure of $\cM_3[\sigma_3]$ (the supersingular locus of $\cM_3$) with the boundary of $\cM_3$.
\end{question}

Note that each irreducible component $\Gamma$ of $\cM_3[\sigma_3]$ has dimension $2$.  
The closure of $\Gamma$ intersects $\Delta_1$ (by Theorem~\ref{Tdiaz}) and such an intersection has dimension $1$.

However, $R:=\bar{\cM}_{1;1}[ss] \times \bar{\cM}_{2;1}[\sigma_2]$ has dimension two. 
A point of $R$ represents a supersingular elliptic curve $E$ clutched together (at its marked point $\mathcal{O}_E$) 
with a supersingular curve $C$ of genus $2$ (at a marked point $P$).  
Note that $\cM_{2,1}[\sigma_2]$ has dimension $2$, representing the choice of the marked curve $(C,P)$.
So a starting point with Question~\ref{Qdeformg3} is to determine which points of $R$ can be deformed 
to a smooth curve of genus $3$ that is supersingular.

\subsection{Mass formula}

There are many possible ways to generalize the material in Section~\ref{Smassformula}.  For example, one can 
compute the mass formula in Question~\ref{Qmassformula} for other one dimensional families $\mathcal{F}$ of curves 
(whose generic point represents an ordinary curve).
Alternatively, one can generalize Question~\ref{Qmassformula} to the case of intersections of $d$-dimensional families of curves 
with Newton polygon strata that have codimension $d$ in $\cM_g$.

\chapter{Curves and abelian varieties with cyclic actions} \label{C4Aspecial}

\section{Overview: cyclic actions on curves and abelian varieties}

In this chapter, we focus on curves $C$ and abelian varieties $X$ that have an automorphism of order $m$.
We assume throughout that $m \geq 2$ and $\mathrm{char}(k) \nmid 2m$.

In Section~\ref{Scurvecyclic}, we consider curves $C$ that are branched cyclic covers of the projective line. 
The moduli spaces for these covers of curves are examples of Hurwitz spaces.
The irreducible components of these Hurwitz spaces are indexed by the monodromy data of the cover
(the degree $m$, the number of branch points $N$, and the inertia type $a$).
The dimension of each component of the Hurwitz space is $N-3$.

In Section~\ref{Sabvarcyclic},
we consider abelian varieties $X$ having an automorphism of order $m$, with the restriction that the
trivial eigenspace for the $\mu_m$-action is zero.
The moduli spaces for these abelian varieties are examples of Deligne--Mostow Shimura varieties.
For $m >2$, they are unitary Shimura varieties of PEL-type.  A more technical perspective on this material can be found in the optional 
Section~\ref{SShimura}.

Using a generalization of the Torelli morphism, it is possible to map each such Hurwitz space to one of these Shimura varieties.
When the image of the Torelli morphism is open and dense in a component of the Shimura variety, the family is called \emph{special}.
Section~\ref{Smoonenspecial} contains some examples of special families.

\section{Hurwitz spaces for cyclic covers of the projective line} \label{Scurvecyclic}

Let $\zeta_m$ be a fixed $m$th root of unity in $k$.
Let $\pi: C \to \mathbb{P}^1$ be a branched cyclic cover of degree $m$.
Let $\tau \in \mathrm{Aut}(\pi: C \to \mathbb{P}^1)$ be an automorphism of order $m$. 

\subsection{Equations of cyclic covers of the projective line}\label{prelim_curve} 

\begin{lemma}
Suppose $C$ is a curve that admits a $\mu_m$-cover $\pi:C \to \PP^1$.
Let $N$ be the number of branch points of $\pi$.
Then $C$ has an equation of the form 
\begin{equation} \label{Esuperelliptic}
y^m = \prod_{i = 1}^{N}(x-b_i)^{a_i},
\end{equation}
for some distinct values $b_1, \ldots, b_N \in k$ and 
some integers $a_1, \ldots, a_N$ such that $1 \leq a_i < m$ and $\sum_{i=1}^N a_i \equiv 0 \bmod m$.
Also, the chosen automorphism $\tau$ of order $m$ acts by $\tau((x,y)) = (x, \zeta_m y)$.
\end{lemma}

\begin{proof}
The cover $\pi$ is determined by an inclusion of $k(x)$ into the function field of $C$. 
By Kummer theory, there is an affine equation for $C$ of the form $y^m=f(x)$, for some rational function $f(x) \in k(x)$.
After some changes of coordinates, we can suppose that $f(x) \in k[x]$ is a polynomial and
that each root of $f(x)$ has multiplicity less than $m$.
After a fractional linear transformation, we can suppose that $\pi$ is not branched at $\infty$.
Then the roots of $f(x)$ are the set $B$ of branch points of $\pi$ and we label these
as $b_1, \ldots, b_N$.  
Thus there are integers $a_1, \ldots, a_N$ such that $1 \leq a_i < m$ such that \eqref{Esuperelliptic} is satisfied.
The fact that $\sum_{i=1}^N a_i \equiv 0 \bmod m$ comes from 
the topological description of the fundamental group of ${\mathbb{P}^1} - B$.
\end{proof}

\begin{definition}
Fix integers $m\geq 2, N\geq 2$ and an $N$-tuple of positive integers $a=(a_1,\dots, a_N)$.
Then $a$ is an {\em inertia type} for $m$ and $\gamma=(m,N,a)$ is a {\em monodromy datum} if
\begin{enumerate}
\item $a_i\not\equiv 0\bmod m$, for each $1 \leq i \leq N$, 
\item $\gcd(m, a_1,\dots, a_N)=1$, and
\item $\sum_{i=1}^N a_i \equiv 0 \bmod m$.
\end{enumerate}
\end{definition}

Fix a monodromy datum $\gamma=(m,N,a)$. 
Let $U\subset (\mathbb{A}^1)^N$ be the locus of points where no two of the coordinates are equal.
Let $C \to U$ be the relative curve
whose fiber at each point $t=(b_1,\dots, b_N)\in U$ is the smooth projective curve with affine equation 
\begin{equation} \label{EformulaC1}
y^m=\prod_{i=1}^N(x-b_i)^{a_i}.
\end{equation}

The function $x$ on $C$ yields a map $C \to {\mathbb P}^1_U$
and there is a $\mu_m$-action on $C$ over $U$ given by $\zeta\cdot (x,y)=(x,\zeta\cdot y)$ for all $\zeta\in\mu_m$.
Thus $C \to {\mathbb P}^1_U$ is a $\mu_m$-cover of (relative) curves. 

In addition, 
one can move three of the branch points to $0,1,\infty$.
Then we take $U \subset (\mathbb{A}^1 - \{0,1\})^{N-3}$ to be the 
locus of points where no two of the coordinates are equal.
In that case, \eqref{EformulaC1} simplifies to:
\begin{equation} \label{EformulaC2}
y^m=x^{a_1}(x-1)^{a_2}\prod_{i=3}^{N-1}(x-b_i)^{a_i}.
\end{equation}

For a closed point $t\in U$, let $C_t$ denote the smooth projective curve 
with affine equation \eqref{EformulaC1} (or \eqref{EformulaC2}).
There is a $\mu_m$-cover $C_t\to \PP^1$ taking $(x,y) \mapsto x$; it is
branched at $N$ points $b_1,\dots , b_N$ in $\PP^1$, and 
has local monodromy $a_i$ at $b_i$.
Let $J_t$ be the Jacobian of $C_t$.

If $a_i > 1$, then the affine curve \eqref{EformulaC1} (or \eqref{EformulaC2}) has a singularity at the point $(b_i,0)$.
Finding the equation for the blow-up is a long process which is best to avoid.

\subsection{The genus and the signature}

\begin{lemma} \label{Lspecialgenus} [Riemann--Hurwitz formula]
For all $t \in U$, the curve $C_t$ is irreducible.  
Its genus $g$ is $(m-1)(N-2)/2$ if $m$ is prime.
More generally, the genus is:
\begin{equation} \label{Egenus}
g=g(m,N,a)=1+\frac{(N-2)m-\sum_{i=1}^N\gcd(a_i,m)}{2}.
\end{equation}
\end{lemma}

The Jacobian $J_t$ and all the cohomology groups of $C_t$ are modules for the group ring $\ZZ[\mu_m]$.
Thus they decompose into eigenspaces under the $\mu_m$-action.\footnote{We are implicitly using 
the fact that the family in \eqref{EformulaC1} can be defined over $\ZZ[1/m]$, and that
one can work over $\CC$ to calculate the dimensions of these eigenspaces.} 
Let $V$ be the first Betti cohomology group $H^1(C_t(\CC), \QQ)$. 
Let $V^+=H^0(C_t(\CC), \Omega^1_{C_t})$.

The data of a $\mu_m$-cover includes an inclusion of $\mu_m$ in $\mathrm{Aut}(C_t)$.
There is an induced action of $\mu_m$ on $V^+$.
Recall that we fixed an $m$th root of unity $\zeta_m$.
For $0 \leq n \leq m-1$, let $L_n$ denote the subspace of $\omega \in V^+$ such that 
$\zeta_m \cdot \omega = \zeta_m^n \omega$.  
The dimension of $L_n$ is independent of the choice of $t\in U$. 
The subspace $L_0$ is trivial since $C_t$ is a $\mu_m$-cover of $\PP^1$.
There is a decomposition:
\[V^+=\oplus_{1 \leq n \leq m-1} L_n.\]

\begin{definition}
Let $\cf_n = \mathrm{dim}(L_n)$. 
The \emph{signature type} of the monodromy datum $\gamma=(m,N,a)$ is 
\[\cf=(\cf_1, \ldots, \cf_{m-1}).\]
\end{definition}

Note that $\sum_{n=1}^{m-1} \cf_n= g$.

For any $q\in \QQ$, let $\langle q\rangle$ denote the fractional part of $q$. 

\begin{lemma} \label{Lsignature} [Hurwitz, Chevalley-Weil] 
If $1 \leq n \leq m-1$, then
\begin{equation}\label{DMeqn}
\cf_n= -1+\sum_{i=1}^N\langle\frac{-na(i)}{m}\rangle
\end{equation}
\end{lemma}

\subsection{Hurwitz spaces}

Let $\gamma=(m,N,a)$ be a monodromy datum with $N \geq 4$.
The Hurwitz space $\cH_\gamma$ is the moduli space of $\mu_m$-covers $\pi:C \to \PP^1$ having 
monodromy datum $\gamma$.
There is a forgetful map $\cH_\gamma \to \cM_g$ that takes the isomorphism class of $\pi$ to the isomorphism class of $C$.

\begin{theorem} \label{TdimHurwitz} \cite[Corollary 7.5]{fultonhur}, \cite[Corollary 4.2.3]{wewersthesis}
The Hurwitz space $\cH_\gamma$ is irreducible.
It has dimension $\mathrm{dim}(\cH_\gamma)=N-3$.
\end{theorem}

\section{Jacobians of cyclic covers of the projective line} \label{Sabvarcyclic} 

We would like to understand the subspace of $\cA_g$ whose points represent Jacobians of curves 
that are cyclic covers of $\PP^1$.
In this section, we take a more accessible approach to this topic.  
In the next section, we approach the same topic from the perspective of unitary Shimura varieties.

Let $\gamma=(m,N,a)$ be a monodromy datum with $N \geq 4$, let $g$ be the associated genus given by 
Lemma~\ref{Lspecialgenus}, and let
$\cf$ be the associated signature type given by (\ref{DMeqn}). 

Given a $\mu_m$-cover $\pi: C \to \PP^1$ with monodromy datum $\gamma$, then the Jacobian $\mathrm{Jac}(C)$ is a
p.p.\ abelian variety of dimension $g$, with an induced action of the group ring $\ZZ[\mu_m]$, 
such that the signature of the action is given by $\cf$.

Composition with the Torelli map yields a morphism
\begin{equation}
j=j_\gamma: \cH_\gamma  \to {\mathcal M_g}  \to \cA_g,
\end{equation}
taking the isomorphism class of $\pi: C \to \PP^1$ to that of $\mathrm{Jac}(C)$.

\begin{definition} \label{DZmNa}
Given a monodromy datum $\gamma=(m,N,a)$, let $\cT^{\circ}_\gamma$ be the image of $j_\gamma$ in $\cA_g$ 
(with the reduced induced structure),
and let
$\cT_\gamma$ be the closure of $\cT^{\circ}_\gamma$ in $\cA_g$.
\end{definition}

By definition, $\cT_\gamma$ is a closed, reduced substack of $\cA_g$.  

\begin{remark} \label{Requivalence}
Let $\pi: C \to \PP^1$ be a $\mu_m$-cover with monodromy datum $\gamma$.
Changing the generator of $\mu_m$ does not change $C$ or $\mathrm{Jac}(C)$.
Changing the ordering of the branch points does not change $C$ or $\mathrm{Jac}(C)$.
So $\cT_\gamma$ depends only on the equivalence class of $\gamma=(m,N,a)$, 
where $(m,N,a)$ and $(m',N',a')$ are equivalent if $m=m'$, $N=N'$, 
and the images of $a,a'$ in $(\ZZ/m\ZZ)^N$ are in the same orbit under
$(\ZZ/m\ZZ)^*\times \mathrm{Sym}_N$, where $c \in (\ZZ/m\ZZ)^*$ acts by multiplication on each coordinate and 
$\sigma \in \mathrm{Sym}_N$ acts by permuting the coordinates.
\end{remark}

In \cite{deligne-mostow}, 
Deligne and Mostow construct the smallest unitary Shimura variety whose image in $\cA_g$ contains $\cT_\gamma$;
we denote it by $\mathcal{S}_\gamma=\Sh(\mu_m,\cf)$.
Section~\ref{SShimura} contains the basic definitions and facts about PEL-type Shimura varieties, 
and the construction of \cite{deligne-mostow}, following \cite{moonen}.
Naively speaking, the points of $\mathcal{S}_\gamma$ represent p.p.\ abelian varieties of dimension $g$, 
equipped with an action of $\ZZ[\mu_m]$, with the signature of the action given by $\cf$, together with some extra structure.
Also $\cT_\gamma$ can be viewed as (the closure of) the intersection in $\cA_g$ of the Torelli locus and the (image of) $\mathcal{S}_\gamma$.

Here is a schematic diagram of the moduli spaces:

\begin{equation}
\xymatrix{\cH_\gamma  \ar[d] \ar[r]^{\tau_\gamma} & \mathcal{S}_\gamma   \ar[d] &  \\
\cM_g \ar[r]^{\tau_g} & \cT_\gamma  & \subset \cA_g.}
\end{equation}

The main result we need is the dimension of $\mathcal{S}_\gamma$, which is given as follows.

\begin{proposition} \label{PdimShimura} \cite[Proposition~5.13]{moonenoort}
Let $\gamma = (m,N,a)$ be a monodromy datum with associated signature $\cf$.

If $m$ is even, let $m_1=m/2$ and $\epsilon_\gamma = \cf_{m_1}(\cf_{m_1}+1)/2$; if $m$ is odd, let $\epsilon_\gamma=0$.

Then the dimension of the Shimura variety $\mathcal{S}_\gamma = \Sh(\mu_m,\cf)$ is
\begin{equation} \label{EdimShimura}
\mathrm{dim}(\mathcal{S}_\gamma) = \epsilon_\gamma + \sum_{n=1}^{\lfloor m/2 \rfloor} \cf_n \cf_{-n}.
\end{equation}
\end{proposition}

The main ideas of the proof of Proposition~\ref{PdimShimura} are to study the 
Hodge structure with $\mu_m$-action and symplectic form, and to compute the dimension of the tangent space.

\begin{definition} 
The family of $\mu_m$-covers $\pi:C \to \mathbb{P}^1$ with monodromy datum $\gamma = (m,N,a)$ is \emph{special}
when the image of $\cH_\gamma$ is open and dense in $\mathcal{S}_\gamma$.
\end{definition}

Equivalently, the family is special when $\mathrm{dim}(\cH_\gamma) = N-3$ equals $\mathrm{dim}(\mathcal{S}_\gamma)$ as in \eqref{EdimShimura}.
Section~\ref{Smoonenspecial} contains examples of special families of cyclic covers of $\mathbb{P}^1$.

\begin{remark}
The precise definition of special is more technical; a good reference is \cite{moonenoort}. 
The following selection of papers illustrates some of the perspectives on this topic: Shimura \cite{Shimurapurely}, 
\cite{deJongNoot}, \cite{Rohde}, \cite{moonen}, and \cite{Frediani}.
\end{remark}

\section{Related topics: Deligne--Mostow Shimura varieties} \label{SShimura}

This section is more technical and can be postponed for another reading.

\subsection{Shimura datum for the moduli space of abelian varieties}\label{sec_Sh_GSp} 

Let $V=\QQ^{2g}$, 
and let $\Psi:V\times V\to\QQ$ denote the standard symplectic form.
Let $G:=\mathrm{GSp}(V,\Psi)$ denote the group of symplectic similitudes over ${\mathbb Q}$. 
Let $\hh$ denote the space of
homomorphisms $h:{\mathbb S}:=\mathrm{Res}_{\CC/\RR} \GG_m \to G_\RR$ which define a Hodge structure of type 
$(-1,0)+(0,-1)$ on $V_\ZZ$ such that $\pm(2\pi i)\Psi$ is a polarization on $V$.
The pair $(G,\hh)$ is the Shimura datum for $\cA_g$.

Let $H\subset G$ be an algebraic subgroup over $\QQ$ such that the subspace
\[\hh_H:=\{h\in\hh\mid h \text{ factors through } H_\RR\}\] is non-empty. Then $H(\RR)$ acts on $\hh_H$ by conjugation, and for each $H(\RR)$-orbit $Y_H\subset \hh_H$, 
the Shimura datum $(H,Y_H)$ defines an algebraic substack $\Sh(H,Y_H)$ of $\cA_g$. 
We write $\Sh(H,\hh_H)$ for the finite union of the Shimura stacks $\Sh(H,Y_H)$, as $Y_H$ varies among the $H(\RR)$-orbits in $\hh_H$.

\subsection{Shimura data of PEL-type}\label{sec_Sh_PEL} 

We focus on Shimura data of PEL-type. 
Let $B$ be a semisimple $\QQ$-algebra, with an involution $*$. 
Suppose there is an action of $B$ on $V$ such that 
$\Psi(bv,w)=\Psi(v,b^*w)$, for all $b\in B$ and all $v,w\in V$.
Let
\[H_B:=\mathrm{GL}_B(V)\cap \mathrm{GSp}(V,\Psi).\]  
We assume that $\hh_{H_B}\neq \emptyset$.

For each $H_B(\RR)$-orbit $Y_B:=Y_{H_B}\subset \hh_{H_B}$, the Shimura stack 
$\Sh(H_B,Y_B)$ arises as the moduli space of polarized abelian varieties 
endowed with a $B$-action.
We say that $\Sh(H_B,Y_B)$ has PEL-type and also write $\Sh(B):=\Sh(H_B,\hh_{H_B})$.

Each homomorphism $h\in Y_B$ defines a decomposition of $B_\CC$-modules \[V_\CC=V^+\oplus V^-,\]
where $V^+$ (resp.\ $V^-$) is the subspace of $V_\CC$ on which $h(z)$ acts by $z$ 
(resp.\ by the complex conjugate $\bar{z}$).
The isomorphism class of the $B_\CC$-module $V^+$ depends only on $Y_B$. Moreover, $Y_B$ is determined by the isomorphism class of $V^+$ as a $B_\CC$-submodule of $V_\CC$.
In the following, we describe $Y_B$ in terms of the $B_\CC$-module $V^+$. 
By construction, $\dim_\CC V^+=g$.

\subsection{Shimura subvariety attached to a monodromy datum} \label{sec_Sh_cyclic} 

We consider cyclic covers of $\mathbb{P}^1$ branched at more than three points.
We fix a monodromy datum $\gamma=(m,N,a)$ with $N\geq 4$ and $m>2$. Take $B=\QQ[\mu_m]$ with involution $*$ being complex conjugation.

As in Section \ref{prelim_curve}, let $C \to U$ denote the universal family of $\mu_m$-covers of ${\mathbb P}^1$ branched at $N$ points with inertia type $a$. 
Let $j=j_\gamma: U \to \cA_g$ be the composition of the Torelli map with the morphism $U \to \cM_g$.
From Definition \ref{DZmNa}, recall that $\cT_\gamma$ is the closure in $\cA_g$ of the image of $j_\gamma$. 

The pullback of the universal abelian scheme $\CX_g$ on $\cA_g$ via $j$ is the relative Jacobian $\CJ$ of $C\rightarrow U$. Since $\mu_m$ acts on $C$, there is a natural action of the group algebra $\ZZ[\mu_m]$ on $\CJ$. 
We also use $\CJ$ to denote the pullback of $\CX_g$ to $\cT_\gamma$.  The action of $\ZZ[\mu_m]$ extends naturally to $\CJ$ over $\cT_\gamma$.
Hence the substack $\cT_\gamma$ is contained in $\Sh(\QQ[\mu_m])$ 
for an appropriate choice of a structure of $\QQ[\mu_m]$-module on $V$.

More precisely, fix $t\in \cT_\gamma(\CC)$, and 
let $(\CJ_t,\theta)$ be the corresponding Jacobian with its principal polarization $\theta$.  
Choose a symplectic similitude, i.e., an isomorphism
\[\alpha: (H_1(\CJ_t,\QQ), \psi_\theta)\to (V ,\Psi),\]
such that the pull back of the symplectic form $\Psi$ to $H_1(\CJ_t, \QQ)$ 
is a scalar multiple of $\psi_\theta$,
where $\psi_\theta$ denotes the Riemann form on $H_1(\CJ_t,\QQ)$ for the polarization $\theta$.
Via $\alpha$, the $\QQ[\mu_m]$-action on $\CJ_t$ induces an action on $V$. This action satisfies \[\hh_{\QQ[\mu_m]}\neq \emptyset, \text{ and }\Psi(bv,w)=\Psi(v,b^*w),\] for all $b\in \QQ[\mu_m]$, and all $v,w\in V$.

The isomorphism class of $V^+$ as a $\QM\otimes_\QQ\CC$-module is determined by and determines 
the signature type of its $\mu_m$-action.
By \cite[2.21, 2.23]{deligne-mostow}, 
the $H_\QM(\RR)$-orbit $Y_{\QQ[\mu_m]}$ in $\hh_{H_\QM}$ such that \[\cT_\gamma \subset \Sh(H_\QM, Y_\QM) \] corresponds to the isomorphism class of $V^+$ with $\cf$ given by \eqref{DMeqn}. We note that $\Sh(H_\QM, Y_\QM)$ depends only on $\mu_m$ and $\cf$, so we denote it by $\mathcal{S}_\gamma = \Sh(\mu_m,\cf)$.

The irreducible component of $\mathcal{S}_\gamma$ which contains $\cT_\gamma$ is the largest closed, reduced and irreducible substack 
$\mathcal{S}$ of $\cA_g$ containing $\cT_\gamma$ such that the action of $\ZZ[\mu_m]$ on $\CJ$ extends to the universal abelian scheme over $\mathcal{S}$. 

\section{Special families of cyclic covers} \label{Smoonenspecial}

Consider a family $\mathcal{F}$ of degree $m$ cyclic covers of $\mathbb{P}^1$, 
branched at $N \geq 4$ points, with monodromy datum $\gamma = (m, N, a)$.
Moonen \cite{moonen} proved that there are exactly 20 monodromy data $\gamma$ for which $\mathcal{F}$ is special 
(up to equivalence, see Remark~\ref{Requivalence}).
We refer to his paper for references to earlier work on some of these families. 
A family of cyclic covers of $\mathbb{P}^1$ is special if and only if 
the Hurwitz space $\cH_\gamma$ and the Shimura variety $\mathcal{S}_\gamma$ have the same dimension.

\begin{example} \label{Mspecial}
Here are the 14 examples of one dimensional special families of cyclic covers of $\mathbb{P}^1$.
The table includes the label from \cite[Table~1]{moonen}\footnote{The unusual numbering is because we do not 
include the special families whose dimension is greater than $1$.}, then
the degree $m$, the inertia type $a$, the genus $g$, and the signature $\cf$.
The affine equation for the family is
\[y^m=x^{a_1}(x-1)^{a_2}(x-t)^{a_3}.\]
\begin{center}
		\begin{tabular}{  |c|c|c|c|c|  }
\hline
Label & $m$ & $a$ & $g$ & $\cf$ 
\\ \hline \hline
$M[1]$ & $2$ & $(1,1,1,1)$ & $1$ & $(1)$ 
\\ \hline
$M[3]$ & $3$ & $(1,1,2,2)$ & $2$ & $(1,1)$  
\\ \hline
$M[4]$ & $4$ & $(1,2,2,3)$ & $2$ & $(1,0,1)$  
\\ \hline
$M[5]$ & $6$ & $(2,3,3,4)$ & $2$ & $(1,0,0,0,1)$  
\\ \hline	
$M[7]$ & $4$ & $(1,1,1,1)$ & $3$ & $(2,1,0)$ 
\\ \hline
$M[9]$ & $6$ & $(1,3,4,4)$ & $3$ & $(1,1,0,0,1)$ 
\\ \hline
$M[11]$& $5$ & $(1,3,3,3)$ & $4$ & $(1,2,0,1)$ 
\\ \hline
$M[12]$ & $6$ & $(1,1,1,3)$ & $4$ & $(2,1,1,0,0)$ 
\\ \hline
$M[13]$ & $6$ & $(1,1,2,2)$ & $4$ & $(2,1,0,1,0)$ 
\\ \hline
$M[15]$& $8$ & $(2,4,5,5)$ & $5$ & $(1,1,0,0,2,0,1)$ 
\\ \hline
$M[17]$ & $7$ & $(2,4,4,4)$ & $6$ & $(1,2,0,2,0,1)$ 
\\ \hline
$M[18]$ & $10$ & $(3,5,6,6)$ & $6$ & $(1,1,0,1,0,0,2,0,1)$ 
\\ \hline
$M[19]$ & $9$ & $(3,5,5,5)$ & $7$ & $(1,2,0,2,0,1,0,1)$  
\\ \hline
$M[20]$ & $12$ & $(4,6,7,7)$ & $7$ 
& $(1,1,0,1,0,0,2,0,1,0,1)$ 
\\ \hline
		\end{tabular}
	\end{center}
\end{example}

\begin{remark}
Note that the family $M[1]$ is the Legendre family and the families $M[3,4,5]$ are studied in \cite{iko}.
\end{remark}

We refer to this table in Sections~\ref{Skottwitz} and \ref{Sopenfield}.

\section{Open questions on the Coleman--Oort conjecture}

Suppose $g \geq 4$. 
Coleman conjectured that there are only finitely many (smooth, projective, irreducible) curves $C$ of genus $g$ such that 
$\mathrm{Jac}(C)$ has complex multiplication \cite{coleman}.  The table in Example~\ref{Mspecial}
contains special families of cyclic covers of $\mathbb{P}^1$ that provide
counterexamples to the Coleman conjecture for $4 \leq g \leq 7$. 

For large $g$, Oort stated the expectation that there is no positive-dimensional special subvariety of $\cA_g$ 
contained in the Torelli locus, with generic point contained in the open Torelli locus \cite{oortcolemanconj}.
Because of the Andr\'e--Oort Conjecture for $\cA_g$, proved by Tsimerman \cite{tsimermanandreoort}, 
Oort's expectation is equivalent to Coleman's conjecture for large $g$; this problem is now called the Coleman--Oort conjecture.

\begin{question} \label{ColemanOort}
What is the largest $g$ for which there is a counterexample to the Coleman-Oort conjecture?
\end{question}

This is a difficult question.  
We refer the reader to the wide array of results, which either prove the Coleman-Oort conjecture in 
special situations or classify families of curves of (small) genus that provide counterexamples to the original Coleman conjecture.

\chapter[Newton polygons in the context of cyclic actions]{Newton polygons in the context of cyclic actions} \label{Chapter4BNPspecial}

\section{Overview: $p$-torsion invariants for cyclic actions}

Suppose $X$ is a principally polarized abelian variety of dimension $g$ 
defined over an algebraically closed field $k$ with $\mathrm{char}(k) = p$.
Let $m \geq 3$ be an integer with $p \nmid 2m$.
Suppose $X$ has an action by the group ring $\ZZ[\mu_m]$. 
Then the interaction between the Frobenius action and the $\mu_m$-action 
places constraints on the $p$-rank, $a$-number, Newton polygon, and 
Ekedahl--Oort type of $X$.

In this chapter, we focus on the constraints on the Newton polygon in this situation.
We address the open question of whether the Newton polygons satisfying these constraints can occur 
for Jacobians of curves having an action by a cyclic group.
We restrict to curves that are cyclic covers of the projective line and continue with the notation from Chapter~\ref{C4Aspecial}.

\section{History of complex multiplication examples} \label{Scomplexmultiplication}

Historically, many interesting phenomena were discovered by studying abelian varieties with complex multiplication.

For example, if $m$ is an odd prime, then the curve $C: y^m=x(x-1)$ has genus $(m-1)/2$ and $\mathrm{Jac}(C)$ 
has complex multiplication by the field $\QQ(\zeta_m)$.  These curves were studied by many authors
\cite{weil, honda, GR, yuiFermat, Aoki, alvarez14}.

Such curves $C$ provide many examples of Jacobians with unusual Newton polygons.
The eigenvalues of Frobenius for $\mathrm{Jac}(C)$ can be expressed using Jacobi sums.
The Newton polygons of these curves (also for composite $m$) can also be computed using the Shimura--Taniyama formula, 
see \cite{LMPT1} or \cite{booherpriesproc} for explanation.

\begin{proposition} \label{PCMsupersingular}
Let $m$ be odd and not necessarily prime. 
Let $f$ be the order of $p$ modulo $m$.  If $f$ is even 
and $p^{f/2} \equiv -1 \bmod m$, then $C:y^m=x(x-1)$ is supersingular.
\end{proposition}

See Section~\ref{Squesexistss} for additional references and examples when $g=5$.

\begin{example}
The genus $6$ curve $y^{13} = x(x-1)$ is supersingular if and only if $p \not \equiv 1,3,9 \bmod 13$.

The genus $6$ curve $y^{14} = x(x-1)$ is supersingular if and only if $p \equiv 3,5,6 \bmod 7$. 
\end{example}

By similar methods, there exists a supersingular curve of genus $7$ when $p \equiv 14 \bmod 15$ or $p \equiv 15 \bmod 16$ \cite{yuiFermat}.

\section{Kottwitz method for Newton polygons}

Suppose $X$ has an action by the group ring $\ZZ[\mu_m]$, with signature $\cf$, with $p \nmid 2m$. 
In this section, we describe some restrictions that this places on the $p$-torsion invariants ($p$-rank, $a$-number, 
Newton polygon, and Ekedahl--Oort type).

\subsection{Constraints arising from the Frobenius action on eigenspaces}

The first step of understanding these constraints is to consider the orbits $o$ of $[\times p]$ on $\ZZ/m - \{0\}$.
Both the Dieudonn\'e module and the $p$-torsion group scheme of $X$ 
decompose into summands indexed by the orbits.

For example, if $m$ is prime with $p \nmid 2m$, then the number of orbits equals the number of primes $\mathfrak{p}$ of
$\QQ(\zeta_m)$ above $p$, and the length of the orbits equals the degree of the residue field of $\mathfrak{p}$ over $\FF_p$.
The residue field of $\mathfrak{p}$ acts on the summands.
Similar, but slightly more complicated, statements hold when $m$ is not prime.

The constraints on the $p$-rank in this situation can be found in \cite{Bouw}.
Specifically, the $p$-rank is bounded by the sum (over the orbits) of the length of the orbit multiplied by 
the minimal dimension of an eigenspace $L$ in that orbit.
In particular, if $m$ is prime, then the $p$-rank is divisible by the order of $p$ modulo $m$.
Constraints on the $a$-numbers of superelliptic curves can be found in \cite{elkin11}.

\subsection{The Kottwitz method} \label{Skottwitz}

Furthermore, there are constraints on the Newton polygon in this situation.
These were analyzed by Kottwitz \cite{kottwitz1, kottwitz2}, Rapoport and Richartz \cite{rapoport-richartz}.

\begin{definition} \label{Dkottwitz} For $m$, $\cf$, and $p$ as above: \hfill
\begin{itemize}
\item The Dieudonn\'e module $M$ decomposes into pieces $M_o$ indexed by the orbits
$o$ of $[\times p]$ on $\ZZ/m - \{0\}$.

\item For any Newton polygon on $M_o$, the multiplicity $n_\lambda$ of each slope is divisible by $\#o$.

\item The Rosati involution $*$ acts on $\QQ[\mu_m]$ by involution: if $o$ is invariant under $*$ then $M_o$ is symmetric;
if not, then $M_o \oplus M_{o^*}$ is symmetric. 

\item The maximal Newton polygon $\mu_o$ for $M_o$ has $s$ distinct slopes 
where $s$ is the number of distinct values 
across the orbit of ${\rm dim}(L_i)$ in the range $[1,\cf_i + \cf_{m-i} -1]$.

\item All Newton polygons $\xi_o$ on $M_o$ satisfy $\xi_o \leq \mu_o$ (along with the other conditions above).
\end{itemize}
\end{definition}

The maximal Newton polygon means the one satisfying these constraints which is closest to the ordinary Newton polygon.
The notation $\xi_o \leq \mu_o$ means that $\xi_o$ is on or above $\mu_o$ when drawn as a Newton polygon.
If $\xi_o \not = \mu_o$, this means that $\mu_o$ is closer to being ordinary than $\xi_o$.

\begin{definition}
Given $m$ and $\cf$, in the set of Newton polygons satisfying the constraints in Definition~\ref{Dkottwitz}, 
the maximal element is called \emph{$\mu$-ordinary}, and the minimal element is called \emph{basic}.
\end{definition}

The following examples and others can be found in \cite[Section~6]{LMPT2}.
Recall the notation from Example~\ref{Mspecial}.

\begin{example} \label{Em16}
$M[16]$: Let $m=5$, $N=5$, and $a=(1,1,1,1,1)$.
By Lemma~\ref{Lsignature}, $g=6$ and $\cf= (f_1,f_2,f_3,f_4) = (3,2,1,0)$.

This data is realized by the special family $C: y^5=x(x-1)(x-t_1)(x-t_2)$.

Consider the orbit of $[\times p]$ on $(\ZZ/5\ZZ)-\{0\}$.
If $p \equiv 2,3 \bmod 5$, there is one orbit of size $4$.  The maximum $p$-rank is $0$
because one eigenspace is trivial. 
The $\mu$-ordinary Newton polygon has $3$ distinct slopes;
specifically, it is $(1/4, 3/4) \oplus (1/2,1/2)^2$, 
with $p$-divisible group $(G_{3,1} \oplus G_{1,3}) \oplus G_{1,1}^2$. 
The basic Newton polygon is supersingular.

If $p \equiv 4 \bmod 5$, the orbits are $\{1,4\}$ and $\{2,3\}$.
The maximum $p$-rank is $2$.
The $\mu$-ordinary Newton polygon is $\mathrm{ord}^2 \oplus \mathrm{ss}^4$, and the basic Newton polygon is supersingular.
\end{example}

\begin{example} \label{Em19}
$M[19]$: Let $m=9$, $N=4$, and $a=(1,1,1,6)$.  
By Lemma~\ref{Lsignature}, $g=7$ and $\cf = (2,2,1,1,1,0,0,0)$.

This data is realized by the special family $C: y^9=x(x-1)(x-t)$.

Consider the orbit of $[\times p]$ on $(\ZZ/9\ZZ)-\{0\}$.
If $p \equiv 2,5 \bmod 9$ (resp.\ $p \equiv 8 \bmod 9$), then 
the $\mu$-ordinary Newton polygon is $(1/3,2/3)^2 \oplus ss$ (resp.\ $\mathrm{ord}^2 \oplus \mathrm{ss}^5$)
and the basic Newton polygon is supersingular.
\end{example}

\subsection{Newton polygons on unitary Shimura varieties}

All Newton polygons satisfying these constraints occur.

\begin{theorem} Viehmann/Wedhorn \cite{viehmann-wedhorn}: given $m$ and $\cf$, 
each Newton polygon satisfying the conditions in Definition~\ref{Dkottwitz} occurs on the Shimura variety $\mathcal{S}_\gamma$.
The Newton polygon stratification of $\mathcal{S}_\gamma$ is well-understood. 
\end{theorem}

\section{Related results: Newton polygons of curves with cyclic action}

\subsection{Newton polygons for cyclic covers of curves}

Now we reframe the Newton polygon question for cyclic covers.
Let $\gamma = (m, N, a)$ be a monodromy datum as in Notation~\ref{Nsuperelliptic} 
and let $\cf$ be its associated signature as in Lemma~\ref{Lsignature}.

\begin{question} \label{QcurveShiNP}
Let $\xi$ be a Newton polygon satisfying the conditions in Definition~\ref{Dkottwitz} for $m$ and $\cf$ with respect to $p$.
Does $\xi$ occur as the Newton polygon for a smooth curve $C$ having 
a $\mu_m$-cover $\pi:C \to {\mathbb P}^1$ with monodromy datum $\gamma$? 
\end{question}

Here is a geometric version of this question.
Consider the Torelli morphism $\tau_\gamma : \cH_\gamma \to \mathcal{S}_\gamma$.
and the image $\cT^\circ_\gamma$ of $\cH_\gamma$ in $\cA_g$.

\begin{question} \label{QgeoShiNP}
With notation as in Question~\ref{QcurveShiNP}:
Does the image of $\tau_\gamma$ intersect the Newton polygon stratum $\mathcal{S}_\gamma[\xi]$?
\end{question}

If the answer to Question~\ref{QgeoShiNP} is yes, then 
$\cT^\circ_\gamma$ intersects the Newton polygon stratum $\cA_g[\xi]$.
This means that there is a $\mu_m$-cover of curves $\pi: C \to \PP^1$ with monodromy datum $\gamma$ 
such that $C$ has Newton polygon $\xi$.

These questions are most accessible for the $\mu$-ordinary Newton polygon.
Otherwise, the problem is more difficult for two reasons: 
first, because $\mathcal{S}_\gamma[\xi]$ has positive codimension in $\mathcal{S}_\gamma$; and second, because
there can be (many) singular curves of compact type whose Jacobians are represented by points of  
$\mathcal{S}_\gamma[\xi]$. 

For the special families of cyclic covers of $\PP^1$, 
in \cite{LMPT2} and \cite{LMPT3}, the authors prove that all Newton polygons that occur on $\mathcal{S}_\gamma$ 
also occur on $\cT^\circ_\gamma$.
This yields the following application.

\begin{corollary} \cite{LMPT2}
If $p \gg 0$, there exists a \emph{smooth} supersingular curve with genus:


$g = 6$, when $p \equiv 2,3,4 \bmod 5$, in the family $M[16]$ (and in the family $M[18]$); 

$g = 7$ when $p \equiv 2 \bmod 3$, in the family $M[19]$ (and in the family $M[20]$).
\end{corollary}

\begin{remark}
Under mild conditions on $\gamma$, Bouw proved that the maximal $p$-rank occurs on $\cT^\circ_\gamma$ \cite{Bouw}.
When the number of branch points is small ($N=4$ or $N=5$), for all choices of $m$ and $a$,
Lin, Mantovan, and Singal proved in \cite{LMS} that the $\mu$-ordinary 
Newton polygon occurs on $\cT^\circ_\gamma$.  See also \cite{LMS2}.
\end{remark}

\subsection{Inductive results}

In \cite{LMPT3}, we developed an inductive method, analogous to that described in Section~\ref{Sinductive},
for curves having an action by a cyclic group.
The idea is to study families of $\mu_m$-covers of $\mathbb{P}^1$, letting the number of branch points (and the genus) grow.
(Earlier inductive results about the $p$-ranks of cyclic covers of $\mathbb{P}^1$ can be found in \cite{Bouw} and \cite{POW}).

Working with inductive systems of Hurwitz spaces $\cH_\gamma$, 
we analyze how the Torelli locus intersects the Newton polygon stratification in 
some PEL-type unitary Shimura varieties $\mathcal{S}_\gamma$.
Fixing the degree $m$ and the prime $p$ with $p \nmid 2m$, 
the main idea is that we find infinitely many $N$ and $a$ for which
the open Torelli locus $\cT^\circ_\gamma$ intersects the $\mu$-ordinary locus of $\mathcal{S}_\gamma$; and
for which $\cT^\circ_\gamma$ intersects the non-$\mu$-ordinary locus of $\mathcal{S}_\gamma$.

The full statements of the results are too long to include here 
because they require some subtle conditions on the signatures.
Here is a sample application. 

\begin{theorem}\label{main-thm-ord} \cite[Theorem~1.2]{LMPT3}
	Let $\gamma = (m,N,a)$ be a monodromy datum.  
	Let $p$ be a prime such that $p \nmid 2m$.
	Let $u$ be the $\mu$-ordinary Newton polygon associated to $\gamma$.
	Suppose there exists a $\mu_m$-cover of $\PP^1$ defined over $\overline{\FF}_p$ 
	with monodromy datum $\gamma$ and Newton polygon $u$.
	Then, for any $n\in \ZZ_{\geq 1}$, there exists a smooth curve over $\overline{\mathbb{F}}_p$ 
	with Newton polygon $\xi_n:=u^n\oplus (0,1)^{(m-1)(n-1)}$.
\end{theorem}


Theorem~\ref{main-thm-ord} is geometrically intriguing because
if $u$ is not ordinary, then for sufficiently large $n$, it demonstrates 
an unlikely intersection of the Newton polygon stratification and the open Torelli locus in $\cA_g$. 

From an arithmetic perspective,
this result produces many new applications about Newton polygons that occur for Jacobians of smooth curves.
For example, take any Newton polygon with interesting slopes from this section 
(e.g.\ 
from Proposition~\ref{PCMsupersingular} or Examples~\ref{Em16}, \ref{Em19}).  Then this technique makes it possible 
to find a curve whose Newton polygon has those slopes with arbitrary high multiplicity, along with some extra slopes of $0$ and $1$.
This also allows us to verify certain cases of Oort's conjecture \cite[Conjecture 8.5.7]{oortpadova05} about Newton polygons.

\section{Open questions on supersingular curves in special families}

In full generality, Questions~\ref{QcurveShiNP} and \ref{QgeoShiNP} appear to be difficult.  
Here are some more tractable questions.

Consider a family $\mathcal{F}$ of cyclic covers of $\mathbb{P}^1$, branched at $N \geq 4$ points, with monodromy datum $\gamma$.
Recall from Section~\ref{Smoonenspecial} that $\mathcal{F}$ is special if 
the Hurwitz space $\cH_\gamma$ and the Shimura variety $\mathcal{S}_\gamma$ have the same dimension.
The following questions about supersingular curves in special families can be approached from several perspectives.

\subsection{Field of definition of supersingular curves} \label{Sopenfield}

Recall the list of one dimensional special families ($N=4$) from Example~\ref{Mspecial}.
For each such family, the basic Newton polygon depends on the congruence class of $p$ modulo $m$.
In many cases, this basic Newton polygon is supersingular.

\begin{question} \label{Qfieldofdef}
What is the field of definition of the basic points 
for each of the 1-dimensional special families of cyclic covers of $\mathbb{P}^1$?
\end{question}

For $M[1]$, it is well-known that supersingular elliptic curves can be defined over $\FF_{p^2}$, and there is a formula 
for the number of isomorphism classes of supersingular curves that can be defined over $\FF_p$.
For $M[3,4,5]$, the field of definition of the basic points is also $\FF_{p^2}$ \cite{iko}.

\subsection{Supersingular locus in two dimensional families}

\begin{example} \label{Especialdim2}
Suppose $N=5$ (so the dimension of the family is two).  
There are exactly four such families $\mathcal{F}$, up to equivalence, 
which are labeled $M[6]$, $M[8]$, $M[14]$, and $M[16]$ in \cite[Table 1]{moonen}.
Under the congruence condition listed below, the basic Newton polygon for $\mathcal{F}$ is supersingular.  
\begin{enumerate}
\item $M[6]$: the family is $\mathcal{F}:y^3 = x(x-1)(x-t_1)(x-t_2)$, so $g=3$, with $p \equiv 2 \bmod 3$.
\item $M[8]$: the family is $\mathcal{F}:y^4 = x(x-1)(x-t_1)^2(x-t_2)^2$, so $g=3$, with $p \equiv 3 \bmod 4$.
\item $M[14]$: the family is $\mathcal{F}: y^6=x^2(x-1)^2(x-t_1)^2(x-t_2)^3$, so $g=4$, with $p \equiv 5 \bmod 6$.
\item $M[16]$: the family is $\mathcal{F}: y^5=x(x-1)(x-t_1)(x-t_2)$, so $g=6$, with $p \equiv 2,3,4 \bmod 5$.
\end{enumerate}
\end{example}


For each of these families $\mathcal{F}$, in \cite[Theorem~7.1]{LMPT3}, the authors gave a non-constructive proof that
there is a smooth curve in $\mathcal{F}$ which is supersingular under the given congruence condition on $p$ (for $p$ sufficiently large).  
For the families $M[6,8,16]$, 
in \cite{strengss}, Streng showed that one of the supersingular curves in $\mathcal{F}$ is defined over $\FF_p$ 
under the given congruence condition on $p$ (and removed the condition that $p$ is sufficiently large).
In \cite{booherpries3mod4}, for the family $M[8]$, 
Booher and I found the explicit polynomial condition on $t_1,t_1$ (of bidegree $(A,A)$ where $A=(p^2-1)/4$)
which cuts out the supersingular locus.

\begin{question} \label{Qgeossdim2}
For the four cases in Example~\ref{Especialdim2}, determine the geometry of the basic locus (which is the supersingular locus under the 
given congruence condition on $p$).  It has dimension one.  Is it irreducible?  If not, how many components does it have?
Is it smooth?  If not, describe its singularities.
\end{question}

\subsection{Basic locus for special families of non-cyclic covers}

Questions similar to Questions~\ref{Qfieldofdef} and \ref{Qgeossdim2} can be asked about 
special families of abelian non-cyclic covers \cite{moonenoort} or special families of non-abelian covers \cite{Frediani}.

\chapter{Torsion points and monodromy}
\label{C3Acovers}

\section{Overview: the monodromy representation on torsion points}

This chapter is about monodromy groups, which measure the 
action of fundamental groups of moduli spaces on torsion points on abelian varieties.
Recall that $k$ is an algebraically closed field.  
For a prime $\ell \not = \mathrm{char}(k)$, 
the $\ell$-torsion of an abelian variety $X$ of dimension $g$ is a vector space of dimension $2g$ over $\ZZ/\ell \ZZ$.
Section~\ref{Smono} introduces the main theme: 
given a family of abelian varieties $X \to W$ over an irreducible base $W$ and a geometric point $w \in W$, the  
fundamental group $\pi_1(W, w)$ acts linearly on the $\ell$-torsion points of the fiber $X_w$.
This produces the monodromy represention, whose image in $\mathrm{GL}_{2g}(\ZZ/\ell\ZZ)$ is called the monodromy group.

We discuss an important result of Chai in Theorem~\ref{Tchaimonodromy} about the $\ell$-adic monodromy group being big for 
an irreducible subspace $W$ of $\cA_g$ which is stable under all $\ell$-adic Hecke correspondences 
and whose generic point is not supersingular.
As an application, in Section~\ref{Smonoprank}, we prove that
the $\ell$-adic monodromy of each irreducible component of the $p$-rank strata of $\mathcal{A}_g$ and $\mathcal{M}_g$ is big.
As a consequence, there are stark differences between the $p$-rank $0$ stratum and the supersingular stratum 
of the moduli spaces $\cA_g$ and $\cM_g$ when $g \geq 3$.

\section{Linear action of fundamental groups on torsion points} \label{Smono}

Let $\ell \not = \mathrm{char}(k)$ be prime.
Let $X$ be a p.p.\ abelian variety over $k$ of dimension $g$.  Let $X[\ell]$ denote the $\ell$-torsion subgroup scheme of $X$.
There are $\ell^{2g}$ points on $X$ that are $\ell$-torsion points and $X[\ell](k) \cong (\ZZ/\ell \ZZ)^{2g}$.

More generally, if $n \geq 1$, then $X[\ell^n](k) \cong (\ZZ/\ell^n\ZZ)^{2g}$.  The $\ell$-adic Tate module is 
$T_\ell(X) = \varprojlim_n X[\ell^n](k) \cong \ZZ_\ell^{2g}$.

Suppose $X \to W$ is a family of p.p.\ abelian varieties over an irreducible $k$-scheme $W$.  
Let $w \in W$ be a geometric point and consider the 
fiber $X_w$ of $X$ above $w$.  The $\ell$-torsion subgroup of $X_w$
is a vector space of dimension $2g$ over $\ZZ/\ell \ZZ$;
we choose a basis for it.

The following result can be found in the appendix to \cite{FreitagKiehl}.
\begin{theorem} With notation as above, for $n \geq 1$:
\begin{enumerate}
\item The fundamental group $\pi_1(W, w)$ acts linearly on the $\ell^n$-torsion 
$X_w[\ell^n] \cong (\ZZ/\ell^n \ZZ)^{2g}$.
\item There is a homomorphism $\rho_{\ell^n}: \pi_1(W, w) \to \mathrm{GL}_{2g}(\ZZ/\ell^n \ZZ)$, 
and a homomorphism $\tilde{\rho}_\ell: \pi_1(W, w) \to \mathrm{GL}_{2g}(\ZZ_\ell)$.
\item Because the principal polarization induces a symplectic form on $X_w[\ell]$ and on the Tate module, 
the image of $\rho_{\ell^n}$ is contained in $\mathrm{Sp}_{2g}(\ZZ/\ell^n \ZZ)$ 
and the image of $\tilde{\rho}_\ell$ is contained in $\mathrm{Sp}_{2g}(\ZZ_\ell)$.
\end{enumerate}
\end{theorem}

\begin{definition}
The \emph{mod-$\ell$} (resp.\ \emph{$\ell$-adic}) 
\emph{monodromy group} of $X \to W$ is the image of $\rho_\ell$ (resp.\ $\tilde{\rho}_\ell$).
The family $X \to W$ has {\emph{big}} mod-$\ell$ (resp.\ $\ell$-adic) monodromy
if the image of $\rho_\ell$ (resp.\ $\tilde{\rho}_\ell$) is $\mathrm{Sp}_{2g}(\ZZ/\ell \ZZ)$ 
(resp.\ $\mathrm{Sp}_{2g}(\ZZ_\ell)$).
\end{definition}

If $g \geq 3$, then the conditions of having big $\ell$-adic monodromy and having big mod-$\ell$ monodromy are 
equivalent.

Now we consider monodromy groups associated with families of curves.
By \cite[Proposition~9.1]{milneJacobian},
if $C$ is a curve of genus $g$, there is a connection between the $\ell$-torsion points on $\mathrm{Jac}(C)$ and
unramified $\ZZ/\ell \ZZ$-covers $C' \to C$.

If $C \to W$ is a family of curves, then we define its monodromy by considering the associated family of 
p.p.\ abelian varieties $\mathrm{Jac}(C) \to W$.
Naively speaking, if a family of curves of genus $g$ has big $\ell$-adic monodromy, 
this implies that the generic curve in the family behaves in some ways like 
a generic curve of genus $g$, 
and its Jacobian behaves in some ways like a generic p.p.\ abelian variety of dimension $g$, at least in terms of
properties of its torsion points of $\ell$-power order.

\section{A big monodromy theorem}

\begin{remark}
Historically, monodromy groups were studied from a topological perspective.  
Results were proved in a more arithmetic setting by Katz--Sarnak \cite{KatzSarnak}, 
Chai \cite{Chaimonodromy}, Hall \cite{hallmono}, and Chai--Oort \cite{CO11}.
We include only one result here from the extensive literature.
\end{remark}

\begin{theorem} \label{Tchaimonodromy} [Chai] \cite{Chaimonodromy} 
Suppose $W$ is an irreducible subspace of $\cA_g$ which is stable under all $\ell$-adic Hecke correspondences. 
If the generic point of $W$ is not in the supersingular locus, then $S$ has big $\ell$-adic monodromy.
\end{theorem}


\section{Related results on monodromy of the $p$-rank strata} \label{Smonoprank}

Let $X$ be a p.p.\ abelian variety of dimension $g$ over $k$.
Suppose $0 \leq f \leq g$. 
Recall that $X$ has $p$-rank $f$ means that $\#X[p](k) = p^f$;
and that this is equivalent to the Newton polygon having exactly $f$ slopes of $0$ and $f$ slopes of $1$.

Recall from \cite[Theorem~2.3]{FVdG} (see Theorem~\ref{TFVdG}) that
the $p$-rank $f$ locus of $\bar{\cM}_g$ has codimension $g-f$ in $\bar{\cM}_g$.
Likewise, the $p$-rank $f$ locus of $\cM_g$ has codimension $g-f$ in $\cM_g$.
This implies that there exists a smooth curve of $p$-rank $f$ in $\cM_g$ for all $p$ and $g$, and $0 \leq f \leq g$.

In this context, Achter and I proved that conditions on the $p$-rank do not produce any restrictions on the 
$\ell$-adic monodromy group.

\begin{theorem} \label{Tbigmonoprank} [Achter--Pries] \cite{AP:mono}
Let $g \geq 3$.  Let $0 \leq f \leq g$.  Let $\ell \not = p$.
Let $W$ be an irreducible component of the $p$-rank $f$ stratum of $\cM_g$.
Then $W$ has big mod-$\ell$ monodromy and big $\ell$-adic monodromy.
\end{theorem}

\begin{proof} [Sketch of proof of Theorem~\ref{Tbigmonoprank}]
For $g=3$, the image of $\cM_3$ is open and dense in $\cA_3$. 
Thus the $p$-rank $f$ stratum is stable under all Hecke correspondences 
and has a generic point which not supersingular. 
So the result follows from Chai's result (Theorem~\ref{Tchaimonodromy}).
For $g >3$, it is not possible to apply Theorem~\ref{Tchaimonodromy} because the Torelli locus is not stable under 
Hecke correspondences.  Instead, the strategy is to work inductively on $g$.

Suppose $g \geq 4$ and $W$ is an irreducible component of the $p$-rank $f$ stratum of $\cM_g$.
By \cite[Theorem 2.3]{FVdG} $\mathrm{codim}(W, \cM_g) = g-f$.
By Theorem~\ref{Tdiaz}, this implies that the closure $\bar{W}$ of $W$ intersects a boundary component $\Delta_i$ 
of $\partial \cM_g$.
By \cite[Lemma~2.5]{FVdG}, 
when $f=0$, this intersection contains points representing chains of (supersingular) elliptic curves,
and so the closure of every irreducible component of the $p$-rank $0$ locus intersects $\Delta_i$ for all $1 \leq i \leq g/2$.
We strengthen \cite[Lemma~2.5]{FVdG} to show that, for any $0 \leq f \leq g$, 
the closure $\bar{W}$ contains points representing chains of elliptic curves, exactly $f$ of which are ordinary; 
also, we can choose the location of the ordinary elliptic curves in the chain.
This shows that $\bar{W}$ contains the image of the clutching morphism
$\kappa_{g_1;1,g_2;1}$ on an irreducible component of
$\bar{\cM}_{g_1;1}^{f_1} \times \bar{\cM}_{g_2;1}^{f_2}$ (for any pair $g_1,g_2$ such that $g_1+g_2=g$, 
and pair $f_1,f_2$ such that $f_1 + f_2 = f$ and $0 \leq f_i \leq g_i$).

The induction step is a little subtle for $g=4$ and $g=5$.  Putting aside those cases for a moment, suppose $g \geq 6$ 
and the result is true for all $3 \leq g' < g$.
For $g \geq 6$, we use that $\bar{W}$ intersects $\Delta_1$ and $\Delta_2$ to show that 
the mod-$\ell$ monodromy group of $\bar{W}$
contains subgroups isomorphic to $\mathrm{Sp}_{2g-2}(\ZZ/\ell\ZZ)$ and $\mathrm{Sp}_{2g-4}(\ZZ/\ell\ZZ)$, 
with non-trivial intersection, but with the latter not contained in the former.
The conclusion follows from group theoretic results about maximal subgroups of the symplectic group.

The $g=4$ case is more complicated.  When $g=4$ and $f=0$ then
the intersection of $\bar{W}$ with $\Delta_2$ is contained in the supersingular locus whose mod-$\ell$ monodromy group is not big.  
Instead, we degenerate the family further into the boundary to show that the 
mod-$\ell$ monodromy group contains distinct subgroups of block form $1,2,1$ and $1,1,2$ with non-trivial intersection.
A similar idea works for $g=5$.
\end{proof}

\begin{remark}
If $C$ is a curve of genus $g$ and $\pi: C' \to C$ is an unramified double cover, then the Prym of $\pi$ is a p.p.\ abelian variety 
of dimension $g-1$.
In \cite{ozmanpriesPrym}, the authors used Theorem~\ref{Tbigmonoprank} to determine information about the $p$-ranks of 
the Prym varieties of unramified double covers of a generic curve of genus $g$ and $p$-rank $f$, for $0 \leq f \leq g$.
\end{remark}

\subsection{Applications of having big monodromy} 

There are some concrete applications of Theorem~\ref{Tbigmonoprank} for curves over finite fields.
Let $\FF=\FF_{p^r}$ be a finite field of characteristic $p$.  Let $k = \bar{\FF}_p$.  

\begin{corollary} \label{Cachterpriesmonodromy} [Achter--Pries] \cite{AP:mono} 
Let $g \geq 3$.  Let $\ell \not = p$ be prime.
\begin{enumerate}
\item There exists a curve $C/k$ of genus $g$ and $p$-rank $0$ s.t.\ $\mathrm{Aut}_k(C) = {\mathrm{id}}$;

\item there exists a curve $C/k$ of genus $g$ and $p$-rank $0$ s.t.\ $\mathrm{Jac}(C)$ is absolutely simple;

\item if $|\FF| \equiv 1 \bmod \ell$, about $\ell/(\ell^2-1)$ of the 
$\FF$-curves of genus $g$ and $p$-rank $0$ have a point of order $\ell$ in $\mathrm{Jac}(C)(\FF)$.

\item for most $\FF$-curves $C$ of genus $g$ and $p$-rank $0$, 
the splitting field of the $L$-polynomial $L(C/\FF, T)$ has degree $2^g g!$ over $\QQ$.
\end{enumerate}
\end{corollary}

\begin{proof}
\begin{enumerate}
\item This follows from Theorem~\ref{Tbigmonoprank} and work of Chai \cite[Corollary~4.3]{Chaimonodromy}.
The Chebotarev density theorem is used to show that the endomorphism algebra of the Jacobian of a generic curve of genus $g$ and 
$p$-rank $0$ is large. 
\item This follows from Theorem~\ref{Tbigmonoprank} and part 1.  This result was proven earlier in 
\cite{AchterGlassPries}.
\item This follows from Theorem~\ref{Tbigmonoprank} and an equidistribution theorem of Katz and Sarnak from 
\cite[9.7.13]{KatzSarnak}.
\item This follows from Theorem~\ref{Tbigmonoprank} and work of Kowalski \cite[Theorem~6.1 and Remark~3.2.(4)]{Kowalskimonodromy}.
\end{enumerate}
\end{proof}

\begin{remark}
In contrast with Theorem~\ref{Tbigmonoprank}, 
irreducible components of the supersingular locus $\cA_g[\sigma_g]$ have trivial $\ell$-adic monodromy,
because of the PFTQ description \cite{lioort}, see \cite[Section~3.3]{karemakerAWS}. 
Furthermore, properties (2)-(4) of Corollary~\ref{Cachterpriesmonodromy} are false for supersingular curves.

Property (1) of Corollary~\ref{Cachterpriesmonodromy} is expected to be true for irreducible components of the 
supersingular locus.
This was proven for $g=2$ and odd $p$ in \cite[Proposition~7.6]{KaremakerPries} and independently in 
\cite{IbukiyamaPP}.
There are several recent results showing property (1) is true for irreducible components of the 
supersingular locus for larger $g$: for $g=3$, see \cite{KaremakerYobukoYu}; 
for $g=4$, for $\cM_4[\sigma_4]$ and $\cA_4[\sigma_4]$, see \cite{dusan3};
for $g \geq 4$ with $g$ even, for $\cA_g[\sigma_g]$, see \cite{KaremakerYu}.
\end{remark}

\section{Open questions on $a$-number strata and monodromy} 

For $X$ a p.p.\ abelian variety of dimension $g$ over $k$, recall from \eqref{Edefanumber} that the $a$-number of $X$ is  
$a = \mathrm{dim}_k\mathrm{Hom} (\alpha_p, X)$.
Generically, the $a$-number is $1$ when $X$ is non-ordinary.  
More generally, for a generic p.p.\ abelian variety $X$ of dimension $g$ and $p$-rank $f$, with $0 \leq f \leq g-1$, 
the $a$-number of $X$ is $1$.  The property of having $a$-number $1$ plays a key role when studying explicit
deformations of the $p$-divisible group of a p.p.\ abelian variety \cite{OortNPformalgroups}. 

For a curve $C$ over $k$ of genus $g$, the $a$-number of $C$ is the 
$a$-number of $\mathrm{Jac}(C)$.
One can ask whether the property of having $a$-number $1$ 
is also true for Jacobians of non-ordinary curves of fixed $p$-rank $f$.  
This is clearly true when $f=g-1$ since $a$ is positive when $f< g$ and since $a \leq g-f$.
It is also true for $f=g-2$ and $f=g-3$ by \cite[Theorems~4.1-4.2]{Pr:large}.  
For simplicity, we restrict to the key case $f=0$ in the next question; the cases for other $p$-ranks should be similar. 

\begin{question} \label{Qanumber1}
Suppose $g \geq 4$.
Suppose $\Gamma$ is an irreducible component of the $p$-rank $0$ stratum $\cM_g^0$ of $\cM_g$.
Is the $a$-number $1$ for the curve represented by the generic geometric point of $\Gamma$? 
\end{question}

Question~\ref{Qslopesprank0} is closely related to Question~\ref{Qanumber1}.

One reason this question is difficult is that there are many singular curves with $p$-rank $0$ and $a$-number greater than $1$.
To see this, consider a singular curve $C_s$ represented by a point of the boundary component
$\Delta_i[\cM_g^0]$, for some $1 \leq i \leq g-1$.
By \eqref{eqblr}, $\mathrm{Jac}(C_s) = \mathrm{Pic}^0(C_1) \times \mathrm{Pic}^0(C_2)$, where 
$C_1,C_2$ both have $p$-rank $0$, and thus both have $a$-number at least $1$.
Thus the $a$-number of $C_s$ is at least $2$.

We end this section with a question about the $\ell$-adic monodromy of the $a$-number $2$ locus. 
Again, for simplicity, we restrict to the key case $f=0$; the cases for other $p$-ranks should be similar.  
For an irreducible component $\Gamma$ of the $p$-rank $0$ stratum $\cM_g^0$, there are some hypotheses 
that are necessary for this question to be meaningful:
first, in light of Theorem~\ref{Tbigmonoprank}, we suppose the answer to Question~\ref{Qanumber1} is yes, meaning 
that the $a$-number is $1$ for a generic curve in $\Gamma$;
second, we suppose that there are curves having $a$-number at least two in $\Gamma$.

\begin{question}
\label{Qmonodromya2}
Let $g \geq 4$.
Suppose $\Gamma$ is an irreducible component of the $p$-rank $0$ stratum $\cM_g^0$ of $\cM_g$.
Suppose that the generic geometric point of $\Gamma$ represents a curve with $a$-number $1$.
Suppose the intersection of $\Gamma$ with the $a$-number $\geq 2$ stratum is non-empty and 
let $W$ be an irreducible component of this intersection.
If $\ell \not = p$ is prime, is the $\ell$-adic monodromy of $W$ big? 
\end{question}

\appendix

\backmatter
\bibliographystyle{amsalpha}
\bibliography{supersingular}

\newcommand{\etalchar}[1]{$^{#1}$}
\def\cprime{$'$}
\providecommand{\bysame}{\leavevmode\hbox to3em{\hrulefill}\thinspace}
\providecommand{\MR}{\relax\ifhmode\unskip\space\fi MR }
\providecommand{\MRhref}[2]{%
  \href{http://www.ams.org/mathscinet-getitem?mr=#1}{#2}
}
\providecommand{\href}[2]{#2}
\begin{thebibliography}{vdGvdV95}

\bibitem[\'{A}14]{alvarez14}
A.~\'{A}lvarez, \emph{The {$p$}-rank of the reduction {${\rm mod}\, p$} of
  {J}acobians and {J}acobi sums}, Int. J. Number Theory \textbf{10} (2014),
  no.~8, 2097--2114. \MR{3273477}

\bibitem[ACG11]{ACGH2}
Enrico Arbarello, Maurizio Cornalba, and Phillip~A. Griffiths, \emph{Geometry
  of algebraic curves. {V}olume {II}}, Grundlehren der mathematischen
  Wissenschaften [Fundamental Principles of Mathematical Sciences], vol. 268,
  Springer, Heidelberg, 2011, With a contribution by Joseph Daniel Harris.
  \MR{2807457}

\bibitem[ACGH85]{ACGH1}
E.~Arbarello, M.~Cornalba, P.~A. Griffiths, and J.~Harris, \emph{Geometry of
  algebraic curves. {V}ol. {I}}, Grundlehren der mathematischen Wissenschaften
  [Fundamental Principles of Mathematical Sciences], vol. 267, Springer-Verlag,
  New York, 1985. \MR{770932}

\bibitem[AGP08]{AchterGlassPries}
Jeffrey~D. Achter, Darren Glass, and Rachel Pries, \emph{Curves of given
  {$p$}-rank with trivial automorphism group}, Michigan Math. J. \textbf{56}
  (2008), no.~3, 583--592. \MR{2490647}

\bibitem[AH19]{achterhowe}
Jeffrey~D. Achter and Everett~W. Howe, \emph{Hasse-{W}itt and {C}artier-{M}anin
  matrices: a warning and a request}, Arithmetic geometry: computation and
  applications, Contemp. Math., vol. 722, Amer. Math. Soc., [Providence], RI,
  [2019] \copyright 2019, pp.~1--18. \MR{3896846}

\bibitem[AL02]{AvritzerLange}
D.~Avritzer and H.~Lange, \emph{The moduli spaces of hyperelliptic curves and
  binary forms}, Math. Z. \textbf{242} (2002), no.~4, 615--632. \MR{1981190}

\bibitem[AM67]{andreottimayer}
A.~Andreotti and A.~L. Mayer, \emph{On period relations for abelian integrals
  on algebraic curves}, Ann. Scuola Norm. Sup. Pisa Cl. Sci. (3) \textbf{21}
  (1967), 189--238. \MR{220740}

\bibitem[Aok08]{Aoki}
Noboru Aoki, \emph{On the zeta function of some cyclic quotients of {F}ermat
  curves}, Comment. Math. Univ. St. Pauli \textbf{57} (2008), no.~2, 163--185.
  \MR{2492579}

\bibitem[AP08]{AP:mono}
Jeffrey~D. Achter and Rachel Pries, \emph{Monodromy of the {$p$}-rank strata of
  the moduli space of curves}, Int. Math. Res. Not. IMRN (2008), no.~15, Art.
  ID rnn053, 25. \MR{2438069}

\bibitem[AP11]{AP:hyp}
\bysame, \emph{The {$p$}-rank strata of the moduli space of hyperelliptic
  curves}, Adv. Math. \textbf{227} (2011), no.~5, 1846--1872. \MR{2803789}

\bibitem[AP14]{APgen}
\bysame, \emph{Generic {N}ewton polygons for curves of given {$p$}-rank},
  Algebraic curves and finite fields, Radon Ser. Comput. Appl. Math., vol.~16,
  De Gruyter, Berlin, 2014, pp.~1--21. \MR{3287680}

\bibitem[Bak00]{Baker}
Matthew~H. Baker, \emph{Cartier points on curves}, Internat. Math. Res. Notices
  (2000), no.~7, 353--370. \MR{1749740 (2001g:11096)}

\bibitem[BC20]{boohercais}
Jeremy Booher and Bryden Cais, \emph{{$a$}-numbers of curves in
  {A}rtin-{S}chreier covers}, Algebra Number Theory \textbf{14} (2020), no.~3,
  587--641. \MR{4113776}

\bibitem[BDG{\etalchar{+}}]{AWSstats}
Thomas Bouchet, Erik Davis, Steven Groen, Zachary Porat, and Benjamin York,
  \emph{Heuristics for (ir)reducibility of $p$-rank strata of the moduli space
  of hyperelliptic curves}, https://arxiv.org/abs/2506.06457.

\bibitem[Bea89]{beauville}
Arnaud Beauville, \emph{Prym varieties: a survey}, Theta functions---{B}owdoin
  1987, {P}art 1 ({B}runswick, {ME}, 1987), Proc. Sympos. Pure Math., vol. 49,
  Part 1, Amer. Math. Soc., Providence, RI, 1989, pp.~607--620. \MR{1013156}

\bibitem[BHM{\etalchar{+}}16]{bouwzeta}
Irene Bouw, Wei Ho, Beth Malmskog, Renate Scheidler, Padmavathi Srinivasan, and
  Christelle Vincent, \emph{Zeta functions of a class of {A}rtin-{S}chreier
  curves with many automorphisms}, Directions in number theory, Assoc. Women
  Math. Ser., vol.~3, Springer, [Cham], 2016, pp.~87--124. \MR{3596578}

\bibitem[BL04]{birkenhakelange}
Christina Birkenhake and Herbert Lange, \emph{Complex abelian varieties},
  second ed., Grundlehren der mathematischen Wissenschaften [Fundamental
  Principles of Mathematical Sciences], vol. 302, Springer-Verlag, Berlin,
  2004. \MR{2062673}

\bibitem[Bla12]{blache}
R\'egis Blache, \emph{Valuation of exponential sums and the generic first slope
  for {A}rtin-{S}chreier curves}, J. Number Theory \textbf{132} (2012), no.~10,
  2336--2352. \MR{2944758}

\bibitem[BLR90]{BLR}
Siegfried Bosch, Werner L{\"u}tkebohmert, and Michel Raynaud, \emph{N\'eron
  models}, Ergebnisse der Mathematik und ihrer Grenzgebiete (3) [Results in
  Mathematics and Related Areas (3)], vol.~21, Springer-Verlag, Berlin, 1990.
  \MR{1045822 (91i:14034)}

\bibitem[Bou01]{Bouw}
Irene~I. Bouw, \emph{The {$p$}-rank of ramified covers of curves}, Compositio
  Math. \textbf{126} (2001), no.~3, 295--322. \MR{1834740 (2002e:14045)}

\bibitem[BPa]{booherpries3mod4}
Jeremy Booher and Rachel Pries, \emph{Producing supersingular curves of genus
  five}, https://arxiv.org/abs/2410.20262.

\bibitem[BPb]{booherpriesproc}
\bysame, \emph{Supersingular curves via the {S}himura--{T}aniyama method},
  https://arxiv.org/abs/2503.04878.

\bibitem[Car57]{Cartier}
Pierre Cartier, \emph{Une nouvelle op\'eration sur les formes
  diff\'erentielles}, C. R. Acad. Sci. Paris \textbf{244} (1957), 426--428.
  \MR{0084497 (18,870b)}

\bibitem[CCO14]{chaiconradoort}
Ching-Li Chai, Brian Conrad, and Frans Oort, \emph{Complex multiplication and
  lifting problems}, Mathematical Surveys and Monographs, vol. 195, American
  Mathematical Society, Providence, RI, 2014. \MR{3137398}

\bibitem[CDG{\etalchar{+}}]{AWSdoublecover}
Kevin Chang, Dusan Dragutinovic, Steven Groen, Yuxin~Lin Lin, Natalia
  Pacheco-Tallaj, and Deepesh Singhal, \emph{The $p$-rank stratification of the
  moduli space of double covers of a fixed elliptic curve},
  https://arxiv.org/abs/2502.09540.

\bibitem[CG72]{clemensgriffiths}
C.~Herbert Clemens and Phillip~A. Griffiths, \emph{The intermediate {J}acobian
  of the cubic threefold}, Ann. of Math. (2) \textbf{95} (1972), 281--356.
  \MR{302652}

\bibitem[Cha05]{Chaimonodromy}
Ching-Li Chai, \emph{Monodromy of {H}ecke-invariant subvarieties}, Pure Appl.
  Math. Q. \textbf{1} (2005), no.~2, 291--303. \MR{2194726}

\bibitem[Cho]{Choi}
Daebeom Choi, \emph{Complete subvarieties of $\mathcal{M}_{g,n}$ and a lifting
  problem}, https://arxiv.org/abs/2304.08568.

\bibitem[CO11]{CO11}
Ching-Li Chai and Frans Oort, \emph{Monodromy and irreducibility of leaves},
  Ann. of Math. (2) \textbf{173} (2011), no.~3, 1359--1396. \MR{2800716}

\bibitem[CO12]{chaioortnotisog}
\bysame, \emph{Abelian varieties isogenous to a {J}acobian}, Ann. of Math. (2)
  \textbf{176} (2012), no.~1, 589--635. \MR{2925391}

\bibitem[Col84]{collino}
Alberto Collino, \emph{A new proof of the {R}an-{M}atsusaka criterion for
  {J}acobians}, Proc. Amer. Math. Soc. \textbf{92} (1984), no.~3, 329--331.
  \MR{759646}

\bibitem[Col87]{coleman}
Robert~F. Coleman, \emph{Torsion points on curves}, Galois representations and
  arithmetic algebraic geometry ({K}yoto, 1985/{T}okyo, 1986), Adv. Stud. Pure
  Math., vol.~12, North-Holland, Amsterdam, 1987, pp.~235--247. \MR{948246}

\bibitem[CP25]{cavalieripries}
Renzo Cavalieri and Rachel Pries, \emph{Mass formula for non-ordinary curves in
  one dimensional families}, Manuscripta Math. \textbf{176} (2025), no.~1,
  Paper No. 10, 30. \MR{4852257}

\bibitem[CU]{caisulmer}
Bryden Cais and Douglas Ulmer, \emph{$p$-torsion for unramified
  {A}rtin--{S}chreier covers of curves}, https://arxiv.org/abs/2307.16346.

\bibitem[Deb05]{Debarre}
Olivier Debarre, \emph{Complex tori and abelian varieties}, french ed., SMF/AMS
  Texts and Monographs, vol.~11, American Mathematical Society, Providence, RI;
  Soci\'et\'e{} Math\'ematique de France, Paris, 2005. \MR{2158864}

\bibitem[Deu41]{deuring}
Max Deuring, \emph{Die {T}ypen der {M}ultiplikatorenringe elliptischer
  {F}unktionenk\"orper}, Abh. Math. Sem. Univ. Hamburg \textbf{14} (1941),
  no.~1, 197--272. \MR{3069722}

\bibitem[dGFS]{notisogFortmanSchreider}
Olivier de~Gaay~Fortman and Stefan Schreieder, \emph{Abelian varieties with no
  power isogenous to a {J}acobian}, https://arxiv.org/pdf/2401.06577.

\bibitem[Dia84]{diaz84}
Steven Diaz, \emph{A bound on the dimensions of complete subvarieties of
  {${\mathcal M}_{g}$}}, Duke Math. J. \textbf{51} (1984), no.~2, 405--408.
  \MR{747872 (85j:14042)}

\bibitem[Dia87]{diaz}
\bysame, \emph{Complete subvarieties of the moduli space of smooth curves},
  Algebraic geometry, {B}owdoin, 1985 ({B}runswick, {M}aine, 1985), Proc.
  Sympos. Pure Math., vol. 46, Part 1, Amer. Math. Soc., Providence, RI, 1987,
  pp.~77--81. \MR{927950}

\bibitem[dJN91]{deJongNoot}
Johan de~Jong and Rutger Noot, \emph{Jacobians with complex multiplication},
  Arithmetic algebraic geometry ({T}exel, 1989), Progr. Math., vol.~89,
  Birkh\"auser Boston, Boston, MA, 1991, pp.~177--192. \MR{1085259}

\bibitem[dJO00]{oortpurity}
A.~J. de~Jong and F.~Oort, \emph{Purity of the stratification by {N}ewton
  polygons}, J. Amer. Math. Soc. \textbf{13} (2000), no.~1, 209--241.
  \MR{1703336}

\bibitem[DM69]{delignemumford}
P.~Deligne and D.~Mumford, \emph{The irreducibility of the space of curves of
  given genus}, Inst. Hautes \'{E}tudes Sci. Publ. Math. (1969), no.~36,
  75--109. \MR{262240}

\bibitem[DM86]{deligne-mostow}
P.~Deligne and G.~D. Mostow, \emph{Monodromy of hypergeometric functions and
  nonlattice integral monodromy}, Inst. Hautes \'Etudes Sci. Publ. Math.
  (1986), no.~63, 5--89. \MR{849651}

\bibitem[DP]{dusanrachel}
Dusan Dragutinovi\'c and Rachel Pries, \emph{Newton polygons and $a$-numbers
  for curves of $p$-rank $0$}, in progress.

\bibitem[Draa]{dusan2}
Dusan Dragutinovi\'c, \emph{{E}kedahl--{O}ort types of stable curves},
  https://arxiv.org/abs/2307.13445.

\bibitem[Drab]{dusan3}
\bysame, \emph{Oort's conjecture and automorphisms of supersingular curves of
  genus four}, arxiv.org/abs/2405.01282.

\bibitem[Dra24]{dusan1}
\bysame, \emph{Supersingular curves of genus four in characteristic two}, Proc.
  Amer. Math. Soc. \textbf{152} (2024), no.~6, 2333--2347. \MR{4741231}

\bibitem[EHR14]{elkieshoweR}
Noam~D. Elkies, Everett~W. Howe, and Christophe Ritzenthaler, \emph{Genus
  bounds for curves with fixed {F}robenius eigenvalues}, Proc. Amer. Math. Soc.
  \textbf{142} (2014), no.~1, 71--84. \MR{3119182}

\bibitem[Eke87]{ekedahl87}
Torsten Ekedahl, \emph{On supersingular curves and abelian varieties}, Math.
  Scand. \textbf{60} (1987), no.~2, 151--178. \MR{914332}

\bibitem[Eke95]{EkedahlBoundary}
\bysame, \emph{Boundary behaviour of {H}urwitz schemes}, The moduli space of
  curves ({T}exel {I}sland, 1994), Progr. Math., vol. 129, Birkh\"auser Boston,
  Boston, MA, 1995, pp.~173--198. \MR{1363057}

\bibitem[Elk11]{elkin11}
Arsen Elkin, \emph{The rank of the {C}artier operator on cyclic covers of the
  projective line}, J. Algebra \textbf{327} (2011), 1--12. \MR{2746026}

\bibitem[EP13]{EP13}
Arsen Elkin and Rachel Pries, \emph{Ekedahl--{O}ort strata of hyperelliptic
  curves in characteristic 2}, Algebra Number Theory \textbf{7} (2013), no.~3,
  507--532, arXiv:1007.1226. \MR{3095219}

\bibitem[ES93]{ekedahlserre93}
Torsten Ekedahl and Jean-Pierre Serre, \emph{Exemples de courbes
  alg\'{e}briques \`a jacobienne compl\`etement d\'{e}composable}, C. R. Acad.
  Sci. Paris S\'{e}r. I Math. \textbf{317} (1993), no.~5, 509--513.
  \MR{1239039}

\bibitem[EvdG09]{EVdG}
Torsten Ekedahl and Gerard van~der Geer, \emph{Cycle classes of the {E}-{O}
  stratification on the moduli of abelian varieties}, Algebra, arithmetic, and
  geometry: in honor of {Y}u. {I}. {M}anin. {V}ol. {I}, Progr. Math., vol. 269,
  Birkh\"auser Boston Inc., Boston, MA, 2009, pp.~567--636. \MR{2641181
  (2011e:14080)}

\bibitem[EvdGM]{EvdGMbook}
Bas Edixhoven, Gerald van~der Geer, and Ben Moonen, \emph{Abelian varieties},
  http://van-der-geer.nl/~gerard/AV.pdf.

\bibitem[Far25]{Farb}
Benson Farb, \emph{Irrationality of the general smooth quartic 3-fold using
  intermediate {J}acobians}, Adv. Math. \textbf{465} (2025), Paper No. 110160,
  6. \MR{4864700}

\bibitem[FC90]{faltingschaidegeneration}
Gerd Faltings and Ching-Li Chai, \emph{Degeneration of abelian varieties},
  Ergebnisse der Mathematik und ihrer Grenzgebiete (3) [Results in Mathematics
  and Related Areas (3)], vol.~22, Springer-Verlag, Berlin, 1990, With an
  appendix by David Mumford. \MR{1083353}

\bibitem[FGP15]{Frediani}
Paola Frediani, Alessandro Ghigi, and Matteo Penegini, \emph{Shimura varieties
  in the {T}orelli locus via {G}alois coverings}, Int. Math. Res. Not. IMRN
  (2015), no.~20, 10595--10623. \MR{3455876}

\bibitem[FK88]{FreitagKiehl}
Eberhard Freitag and Reinhardt Kiehl, \emph{\'etale cohomology and the {W}eil
  conjecture}, Ergebnisse der Mathematik und ihrer Grenzgebiete (3) [Results in
  Mathematics and Related Areas (3)], vol.~13, Springer-Verlag, Berlin, 1988,
  Translated from the German by Betty S. Waterhouse and William C. Waterhouse,
  With an historical introduction by J. A. Dieudonn\'e. \MR{926276}

\bibitem[Ful69]{fultonhur}
William Fulton, \emph{Hurwitz schemes and irreducibility of moduli of algebraic
  curves}, Ann. of Math. (2) \textbf{90} (1969), 542--575. \MR{0260752}

\bibitem[FvdG04]{FVdG}
Carel Faber and Gerard van~der Geer, \emph{Complete subvarieties of moduli
  spaces and the {P}rym map}, J. Reine Angew. Math. \textbf{573} (2004),
  117--137. \MR{2084584}

\bibitem[Gab01]{gabber}
O.~Gabber, \emph{On space filling curves and {A}lbanese varieties}, Geom.
  Funct. Anal. \textbf{11} (2001), no.~6, 1192--1200. \MR{1878318}

\bibitem[GDH91]{diezharvey}
Gabino Gonz\'{a}lez~D\'{\i}ez and William~J. Harvey, \emph{On complete curves
  in moduli space. {I}, {II}}, Math. Proc. Cambridge Philos. Soc. \textbf{110}
  (1991), no.~3, 461--466, 467--472. \MR{1120481}

\bibitem[GMSMT]{GrMSMT}
Samuel Grushevsky, Gabriele Mondello, Riccardo Salvati~Manni, and Jakob
  Tsimerman, \emph{Compact subvarieties of the moduli space of complex abelian
  varieties}, https://arxiv.org/abs/2404.06009.

\bibitem[Gor02]{G:book}
E.~Goren, \emph{Lectures on {H}ilbert modular varieties and modular forms}, CRM
  Monograph Series, vol.~14, American Mathematical Society, Providence, RI,
  2002, With M.-H. Nicole. \MR{2003c:11038}

\bibitem[GP05]{glasspries}
Darren Glass and Rachel Pries, \emph{Hyperelliptic curves with prescribed
  {$p$}-torsion}, Manuscripta Math. \textbf{117} (2005), no.~3, 299--317.
  \MR{2154252}

\bibitem[GR78]{GR}
Benedict~H. Gross and David~E. Rohrlich, \emph{Some results on the
  {M}ordell-{W}eil group of the {J}acobian of the {F}ermat curve}, Invent.
  Math. \textbf{44} (1978), no.~3, 201--224. \MR{0491708}

\bibitem[Gru12]{grushevskysurvey}
Samuel Grushevsky, \emph{The {S}chottky problem}, Current developments in
  algebraic geometry, Math. Sci. Res. Inst. Publ., vol.~59, Cambridge Univ.
  Press, Cambridge, 2012, pp.~129--164. \MR{2931868}

\bibitem[Gun82]{gunning}
R.~C. Gunning, \emph{Some curves in abelian varieties}, Invent. Math.
  \textbf{66} (1982), no.~3, 377--389. \MR{662597}

\bibitem[Hal08]{hallmono}
Chris Hall, \emph{Big symplectic or orthogonal monodromy modulo {$l$}}, Duke
  Math. J. \textbf{141} (2008), no.~1, 179--203. \MR{2372151}

\bibitem[Han92]{HansenDL}
Johan~P. Hansen, \emph{Deligne-{L}usztig varieties and group codes}, Coding
  theory and algebraic geometry ({L}uminy, 1991), Lecture Notes in Math., vol.
  1518, Springer, Berlin, 1992, pp.~63--81. \MR{1186416 (94e:94024)}

\bibitem[Har07a]{harashita07}
Shushi Harashita, \emph{Ekedahl-{O}ort strata and the first {N}ewton slope
  strata}, J. Algebraic Geom. \textbf{16} (2007), no.~1, 171--199. \MR{2257323}

\bibitem[Har07b]{harvey07}
David Harvey, \emph{Kedlaya's algorithm in larger characteristic}, Int. Math.
  Res. Not. IMRN (2007), no.~22, Art. ID rnm095, 29. \MR{2376210}

\bibitem[Har10]{harashita10}
Shushi Harashita, \emph{Generic {N}ewton polygons of {E}kedahl-{O}ort strata:
  {O}ort's conjecture}, Ann. Inst. Fourier (Grenoble) \textbf{60} (2010),
  no.~5, 1787--1830. \MR{2766230}

\bibitem[Har22]{harashita22}
\bysame, \emph{Supersingular abelian varieties and curves, and their moduli
  spaces, with a remark on the dimension of the moduli of supersingular curves
  of genus 4}, Theory and {A}pplications of {S}upersingular {C}urves and
  {S}upersingular {A}belian {V}arieties, RIMS K\^{o}ky\^{u}roku Bessatsu, vol.
  B90, Res. Inst. Math. Sci. (RIMS), Kyoto, 2022, pp.~1--16. \MR{4521510}

\bibitem[Has35]{Hasse35}
Helmut Hasse, \emph{Existenz separabler zyklischer unverzweigter
  {E}rweiterungsk\"orper vom {P}rimzahlgrade {$p$} \"uber elliptischen
  {F}unktionenk\"orpern der {C}harakteristik p}, J. Reine Angew. Math.
  \textbf{172} (1935), 77--85. \MR{1581440}

\bibitem[HM98]{harrismorrison}
Joe Harris and Ian Morrison, \emph{Moduli of curves}, Graduate Texts in
  Mathematics, vol. 187, Springer-Verlag, New York, 1998. \MR{1631825}

\bibitem[Hon66]{honda}
Taira Honda, \emph{On the {J}acobian variety of the algebraic curve
  {$y^{2}=1-x^{l}$}\ over a field of characteristic {$p>0$}}, Osaka J. Math.
  \textbf{3} (1966), 189--194. \MR{0225777}

\bibitem[HPS]{hanselmanpiepersam}
Jeroen Hanselman, Andreas Pieper, and Sam Schiavone, \emph{Equations of genus 4
  curves from their theta constants}, https://arxiv.org/pdf/2402.03160.

\bibitem[HS14]{harveysutherland1}
David Harvey and Andrew~V. Sutherland, \emph{Computing {H}asse-{W}itt matrices
  of hyperelliptic curves in average polynomial time}, LMS J. Comput. Math.
  \textbf{17} (2014), 257--273. \MR{3240808}

\bibitem[HS16]{harveysutherland2}
\bysame, \emph{Computing {H}asse-{W}itt matrices of hyperelliptic curves in
  average polynomial time, {II}}, Frobenius distributions: {L}ang-{T}rotter and
  {S}ato-{T}ate conjectures, Contemp. Math., vol. 663, Amer. Math. Soc.,
  Providence, RI, 2016, pp.~127--147. \MR{3502941}

\bibitem[HT18]{hulektommasi}
Klaus Hulek and Orsola Tommasi, \emph{The topology of {$\mathcal{A}_g$} and its
  compactifications}, Geometry of moduli, Abel Symp., vol.~14, Springer, Cham,
  2018, With an appendix by Olivier Ta\"{\i}bi, pp.~135--193. \MR{3968046}

\bibitem[Hus04]{Husemoller}
Dale Husem\"oller, \emph{Elliptic curves}, second ed., Graduate Texts in
  Mathematics, vol. 111, Springer-Verlag, New York, 2004, With appendices by
  Otto Forster, Ruth Lawrence and Stefan Theisen. \MR{2024529}

\bibitem[Ibu20]{IbukiyamaPP}
Tomoyoshi Ibukiyama, \emph{Principal polarizations of supersingular abelian
  surfaces}, J. Math. Soc. Japan \textbf{72} (2020), no.~4, 1161--1180.
  \MR{4165927}

\bibitem[Igu58]{igusa}
Jun-ichi Igusa, \emph{Class number of a definite quaternion with prime
  discriminant}, Proc. Nat. Acad. Sci. U.S.A. \textbf{44} (1958), 312--314.
  \MR{0098728}

\bibitem[IKO86]{iko}
Tomoyoshi Ibukiyama, Toshiyuki Katsura, and Frans Oort, \emph{Supersingular
  curves of genus two and class numbers}, Compositio Math. \textbf{57} (1986),
  no.~2, 127--152. \MR{827350}

\bibitem[Kar]{karemakerAWS}
Valentijn Karemaker, \emph{Geometry and arithmetic of moduli spaces of abelian
  varieties in positive characteristic}, Abelian varieties, Lectures from the
  Arizona Winter School 2024, AMS Mathematical Surveys and Monographs.

\bibitem[Kat79]{katzslope}
Nicholas~M. Katz, \emph{Slope filtration of {$F$}-crystals}, Journ\'ees de
  {G}\'eom\'etrie {A}lg\'ebrique de {R}ennes ({R}ennes, 1978), {V}ol. {I},
  Ast\'erisque, vol.~63, Soc. Math. France, Paris, 1979, pp.~113--163.
  \MR{563463}

\bibitem[Ked01]{kedlaya01}
Kiran~S. Kedlaya, \emph{Counting points on hyperelliptic curves using
  {M}onsky-{W}ashnitzer cohomology}, J. Ramanujan Math. Soc. \textbf{16}
  (2001), no.~4, 323--338. \MR{1877805}

\bibitem[Ked04]{kedlaya04}
\bysame, \emph{Computing zeta functions via {$p$}-adic cohomology}, Algorithmic
  number theory, Lecture Notes in Comput. Sci., vol. 3076, Springer, Berlin,
  2004, pp.~1--17. \MR{2137340}

\bibitem[KHH20]{kudoharashitahowe}
Momonari Kudo, Shushi Harashita, and Everett~W. Howe, \emph{Algorithms to
  enumerate superspecial {H}owe curves of genus 4}, A{NTS}
  {XIV}---{P}roceedings of the {F}ourteenth {A}lgorithmic {N}umber {T}heory
  {S}ymposium, Open Book Ser., vol.~4, Math. Sci. Publ., Berkeley, CA, 2020,
  pp.~301--316. \MR{4235120}

\bibitem[KHS20]{khs20}
Momonari Kudo, Shushi Harashita, and Hayato Senda, \emph{The existence of
  supersingular curves of genus 4 in arbitrary characteristic}, Res. Number
  Theory \textbf{6} (2020), no.~4, Paper No. 44, 17. \MR{4170348}

\bibitem[Kle05]{kleimanPicard}
Steven~L. Kleiman, \emph{The {P}icard scheme}, Fundamental algebraic geometry,
  Math. Surveys Monogr., vol. 123, Amer. Math. Soc., Providence, RI, 2005,
  pp.~235--321. \MR{2223410}

\bibitem[KM85]{KatzMazur}
Nicholas~M. Katz and Barry Mazur, \emph{Arithmetic moduli of elliptic curves},
  Annals of Mathematics Studies, vol. 108, Princeton University Press,
  Princeton, NJ, 1985. \MR{772569}

\bibitem[Knu83]{knudsen2}
Finn~F. Knudsen, \emph{The projectivity of the moduli space of stable curves.
  {II}. {T}he stacks {$M_{g,n}$}}, Math. Scand. \textbf{52} (1983), no.~2,
  161--199. \MR{702953}

\bibitem[Kob75]{koblitz75}
Neal Koblitz, \emph{{$p$}-adic variation of the zeta-function over families of
  varieties defined over finite fields}, Compositio Math. \textbf{31} (1975),
  no.~2, 119--218. \MR{414557}

\bibitem[Kot85]{kottwitz1}
Robert~E. Kottwitz, \emph{Isocrystals with additional structure}, Compositio
  Math. \textbf{56} (1985), no.~2, 201--220. \MR{809866}

\bibitem[Kot97]{kottwitz2}
\bysame, \emph{Isocrystals with additional structure. {II}}, Compositio Math.
  \textbf{109} (1997), no.~3, 255--339. \MR{1485921}

\bibitem[Kow06]{Kowalskimonodromy}
E.~Kowalski, \emph{The large sieve, monodromy and zeta functions of curves}, J.
  Reine Angew. Math. \textbf{601} (2006), 29--69. \MR{2289204}

\bibitem[KP19]{KaremakerPries}
Valentijn Karemaker and Rachel Pries, \emph{Fully maximal and fully minimal
  abelian varieties}, J. Pure Appl. Algebra \textbf{223} (2019), no.~7,
  3031--3056. \MR{3912957}

\bibitem[KR89]{kanirosen}
E.~Kani and M.~Rosen, \emph{Idempotent relations and factors of {J}acobians},
  Math. Ann. \textbf{284} (1989), no.~2, 307--327.

\bibitem[Kra]{Kraft}
Hanspeter Kraft, \emph{Kommutative algebraische $p$-{G}ruppen (mit
  {A}nwendungen auf $p$-divisible {G}ruppen und abelsche {V}ariet\"aten)},
  Sonderforsch. Bereich Bonn, September 1975, 86 pp.

\bibitem[Kri06]{krichever06}
I.~Krichever, \emph{Integrable linear equations and the {R}iemann-{S}chottky
  problem}, Algebraic geometry and number theory, Progr. Math., vol. 253,
  Birkh\"{a}user Boston, Boston, MA, 2006, pp.~497--514. \MR{2263198}

\bibitem[Kri10]{krichever10}
Igor Krichever, \emph{Characterizing {J}acobians via trisecants of the {K}ummer
  variety}, Ann. of Math. (2) \textbf{172} (2010), no.~1, 485--516.
  \MR{2680424}

\bibitem[KS99]{KatzSarnak}
Nicholas~M. Katz and Peter Sarnak, \emph{Random matrices, {F}robenius
  eigenvalues, and monodromy}, American Mathematical Society Colloquium
  Publications, vol.~45, American Mathematical Society, Providence, RI, 1999.
  \MR{1659828}

\bibitem[KS03]{keelsadun}
Sean Keel and Lorenzo Sadun, \emph{Oort's conjecture for {$A_g\otimes\Bbb C$}},
  J. Amer. Math. Soc. \textbf{16} (2003), no.~4, 887--900. \MR{1992828}

\bibitem[KW15]{kramerweissauer}
T.~Kr\"amer and R.~Weissauer, \emph{On the {T}annaka group attached to the
  theta divisor of a generic principally polarized abelian variety}, Math. Z.
  \textbf{281} (2015), no.~3-4, 723--745. \MR{3421638}

\bibitem[KY]{KaremakerYu}
Valentijn Karemaker and Chia-Fu Yu, \emph{Supersingular {E}kedahl--{O}ort
  strata and {O}ort's conjecture}, arxiv.org/abs/2406.19748.

\bibitem[KYY21]{KaremakerYobukoYu}
Valentijn Karemaker, Fuetaro Yobuko, and Chia-Fu Yu, \emph{Mass formula and
  {O}ort's conjecture for supersingular abelian threefolds}, Adv. Math.
  \textbf{386} (2021), Paper No. 107812, 52. \MR{4267516}

\bibitem[Lan]{Landesman}
Aaron Landesman, \emph{Exercises on the moduli spaces and the {T}orelli map},
  https://swc-math.github.io/aws/2024/2024LandesmanProblems.pdf.

\bibitem[Lan23]{lange}
Herbert Lange, \emph{Abelian varieties over the complex numbers---a graduate
  course}, Grundlehren Text Editions, Springer, Cham, [2023] \copyright 2023.
  \MR{4573077}

\bibitem[LMPT19a]{LMPT2}
Wanlin Li, Elena Mantovan, Rachel Pries, and Yunqing Tang, \emph{Newton
  polygons arising from special families of cyclic covers of the projective
  line}, Res. Number Theory \textbf{5} (2019), no.~1, Paper No. 12, 31.
  \MR{3897613}

\bibitem[LMPT19b]{LMPT1}
\bysame, \emph{Newton polygons of cyclic covers of the projective line branched
  at three points}, Research directions in number theory---{W}omen in {N}umbers
  {IV}, Assoc. Women Math. Ser., vol.~19, Springer, Cham, [2019] \copyright
  2019, pp.~115--132. \MR{4069381}

\bibitem[LMPT22]{LMPT3}
\bysame, \emph{Newton polygon stratification of the {T}orelli locus in unitary
  {S}himura varieties}, Int. Math. Res. Not. IMRN (2022), no.~9, 6464--6511.
  \MR{4411461}

\bibitem[LMS]{LMS2}
Yuxin Lin, Elena Mantovan, and Deepesh Singhal, \emph{Non-$\mu$-ordinary smooth
  cyclic covers of $\mathbb{P}^1$}, https://arxiv.org/abs/2406.09632.

\bibitem[LMS24]{LMS}
\bysame, \emph{Abelian covers of {$\Bbb{P}^1$} of {$p$}-ordinary
  {E}kedahl-{O}ort type}, Int. Math. Res. Not. IMRN (2024), no.~23,
  14369--14392. \MR{4838305}

\bibitem[LO98a]{LO}
Ke-Zheng Li and Frans Oort, \emph{Moduli of supersingular abelian varieties},
  Lecture Notes in Mathematics, vol. 1680, Springer-Verlag, Berlin, 1998.
  \MR{1611305 (99e:14052)}

\bibitem[LO98b]{lioort}
\bysame, \emph{Moduli of supersingular abelian varieties}, Lecture Notes in
  Mathematics, vol. 1680, Springer-Verlag, Berlin, 1998.

\bibitem[Loo95]{looijenga}
Eduard Looijenga, \emph{On the tautological ring of {${\mathcal M}_g$}},
  Invent. Math. \textbf{121} (1995), no.~2, 411--419. \MR{1346214 (96g:14021)}

\bibitem[Man61]{manin61}
Ju.~I. Manin, \emph{The {H}asse-{W}itt matrix of an algebraic curve}, Izv.
  Akad. Nauk SSSR Ser. Mat. \textbf{25} (1961), 153--172. \MR{124324}

\bibitem[Man63]{maninthesis}
\bysame, \emph{Theory of commutative formal groups over fields of finite
  characteristic}, Uspehi Mat. Nauk \textbf{18} (1963), no.~6 (114), 3--90.
  \MR{0157972 (28 \#1200)}

\bibitem[Mat52]{matsusakagenerating}
Teruhisa Matsusaka, \emph{On a generating curve of an {A}belian variety},
  Natur. Sci. Rep. Ochanomizu Univ. \textbf{3} (1952), 1--4. \MR{49604}

\bibitem[Mat59]{matsusaka}
\bysame, \emph{On a characterization of a {J}acobian variety}, Mem. Coll. Sci.
  Univ. Kyoto Ser. A. Math. \textbf{32} (1959), 1--19. \MR{108497}

\bibitem[MFK94]{mumfordfogarty}
D.~Mumford, J.~Fogarty, and F.~Kirwan, \emph{Geometric invariant theory}, third
  ed., Ergebnisse der Mathematik und ihrer Grenzgebiete (2) [Results in
  Mathematics and Related Areas (2)], vol.~34, Springer-Verlag, Berlin, 1994.
  \MR{1304906}

\bibitem[Mil]{milneAV}
James Milne, \emph{Abelian varieties},
  https://www.jmilne.org/math/CourseNotes/AV.pdf.

\bibitem[Mil72]{miller}
Leonhard Miller, \emph{Curves with invertible {H}asse-{W}itt-matrix}, Math.
  Ann. \textbf{197} (1972), 123--127. \MR{314849}

\bibitem[Mil86a]{milnechapterAV}
J.~S. Milne, \emph{Abelian varieties}, Arithmetic geometry ({S}torrs, {C}onn.,
  1984), Springer, New York, 1986, pp.~103--150. \MR{861974}

\bibitem[Mil86b]{milneJacobian}
\bysame, \emph{Jacobian varieties}, Arithmetic geometry ({S}torrs, {C}onn.,
  1984), Springer, New York, 1986, pp.~167--212. \MR{861976}

\bibitem[Mir95]{miranda}
Rick Miranda, \emph{Algebraic curves and {R}iemann surfaces}, Graduate Studies
  in Mathematics, vol.~5, American Mathematical Society, Providence, RI, 1995.
  \MR{1326604}

\bibitem[MO13]{moonenoort}
Ben Moonen and Frans Oort, \emph{The {T}orelli locus and special subvarieties},
  Handbook of moduli. {V}ol. {II}, Adv. Lect. Math. (ALM), vol.~25, Int. Press,
  Somerville, MA, 2013, pp.~549--594. \MR{3184184}

\bibitem[Moo01]{M:group}
Ben Moonen, \emph{Group schemes with additional structures and {W}eyl group
  cosets}, Moduli of abelian varieties ({T}exel {I}sland, 1999), Progr. Math.,
  vol. 195, Birkh\"auser, Basel, 2001, pp.~255--298. \MR{1827024 (2002c:14074)}

\bibitem[Moo10]{moonen}
\bysame, \emph{Special subvarieties arising from families of cyclic covers of
  the projective line}, Doc. Math. \textbf{15} (2010), 793--819. \MR{2735989
  (2012a:14071)}

\bibitem[Moo22]{moonendiscrete}
\bysame, \emph{Computing discrete invariants of varieties in positive
  characteristic: {I}. {E}kedahl-{O}ort types of curves}, J. Pure Appl. Algebra
  \textbf{226} (2022), no.~11, Paper No. 107100, 19. \MR{4412228}

\bibitem[Mum65]{mumfordPicard}
David Mumford, \emph{Picard groups of moduli problems}, Arithmetical
  {A}lgebraic {G}eometry ({P}roc. {C}onf. {P}urdue {U}niv., 1963), Harper \&
  Row, New York, 1965, pp.~33--81. \MR{201443}

\bibitem[Mum75]{mumfordbook}
\bysame, \emph{Curves and their {J}acobians}, University of Michigan Press, Ann
  Arbor, MI, 1975. \MR{419430}

\bibitem[Mum08]{mumfordAVbook}
\bysame, \emph{Abelian varieties}, Tata Institute of Fundamental Research
  Studies in Mathematics, vol.~5, Tata Institute of Fundamental Research,
  Bombay; by Hindustan Book Agency, New Delhi, 2008, With appendices by C. P.
  Ramanujam and Yuri Manin, Corrected reprint of the second (1974) edition.
  \MR{2514037}

\bibitem[MW04]{moonenwedhorn}
Ben Moonen and Torsten Wedhorn, \emph{Discrete invariants of varieties in
  positive characteristic}, Int. Math. Res. Not. (2004), no.~72, 3855--3903.
  \MR{2104263}

\bibitem[MZ20]{masserzannier}
David Masser and Umberto Zannier, \emph{Abelian varieties isogenous to no
  {J}acobian}, Ann. of Math. (2) \textbf{191} (2020), no.~2, 635--674.
  \MR{4076633}

\bibitem[NO80]{normanoort}
Peter Norman and Frans Oort, \emph{Moduli of abelian varieties}, Ann. of Math.
  (2) \textbf{112} (1980), no.~3, 413--439. \MR{595202}

\bibitem[Oda69]{Oda}
Tadao Oda, \emph{The first de {R}ham cohomology group and {D}ieudonn\'e
  modules}, Ann. Sci. \'Ecole Norm. Sup. (4) \textbf{2} (1969), 63--135.
  \MR{0241435 (39 \#2775)}

\bibitem[Oor74]{O:sub}
Frans Oort, \emph{Subvarieties of moduli spaces}, Invent. Math. \textbf{24}
  (1974), 95--119. \MR{0424813 (54 \#12771)}

\bibitem[Oor91]{O:hypsup}
\bysame, \emph{Hyperelliptic supersingular curves}, Arithmetic algebraic
  geometry ({T}exel, 1989), Progr. Math., vol.~89, Birkh\"auser Boston, Boston,
  MA, 1991, pp.~247--284. \MR{1085262 (92c:14043)}

\bibitem[Oor00]{OortNPformalgroups}
\bysame, \emph{Newton polygons and formal groups: conjectures by {M}anin and
  {G}rothendieck}, Ann. of Math. (2) \textbf{152} (2000), no.~1, 183--206.
  \MR{1792294}

\bibitem[Oor01a]{OortNPstrata}
\bysame, \emph{Newton polygon strata in the moduli space of abelian varieties},
  Moduli of abelian varieties ({T}exel {I}sland, 1999), Progr. Math., vol. 195,
  Birkh\"auser, Basel, 2001, pp.~417--440. \MR{1827028}

\bibitem[Oor01b]{O:strat}
\bysame, \emph{A stratification of a moduli space of abelian varieties}, Moduli
  of abelian varieties (Texel Island, 1999), Progr. Math., vol. 195,
  Birkh\"auser, Basel, 2001, pp.~345--416. \MR{2002b:14055}

\bibitem[Oor05]{oortpadova05}
\bysame, \emph{Abelian varieties isogenous to a {J}acobian; in problems from
  the {W}orkshop on {A}utomorphisms of {C}urves}, Rend. Sem. Mat. Univ. Padova
  \textbf{113} (2005), 129--177. \MR{2168985}

\bibitem[Oor19]{oortcolemanconj}
\bysame, \emph{Special subvarieties in the {T}orelli locus}, Open problems in
  arithmetic algebraic geometry, Adv. Lect. Math. (ALM), vol.~46, Int. Press,
  Somerville, MA, [2019] \copyright 2019, pp.~71--93. \MR{3971182}

\bibitem[OP19]{ozmanpriesPrym}
Ekin Ozman and Rachel Pries, \emph{Ordinary and almost ordinary {P}rym
  varieties}, Asian J. Math. \textbf{23} (2019), no.~3, 455--477. \MR{3983755}

\bibitem[OPW22]{POW}
Ekin Ozman, Rachel Pries, and Colin Weir, \emph{The boundary of the {$p$}-rank
  0 stratum of the moduli space of cyclic covers of the projective line},
  Nagoya Math. J. \textbf{248} (2022), 865--887. \MR{4508269}

\bibitem[Pie]{piepertheta}
Andreas Pieper, \emph{Theta nullvalues of supersingular abelian varieties},
  https://arxiv.org/pdf/2208.12492.

\bibitem[Pie22]{pieperg2}
\bysame, \emph{Constructing all genus 2 curves with supersingular {J}acobian},
  Res. Number Theory \textbf{8} (2022), no.~2, Paper No. 32, 26. \MR{4423681}

\bibitem[PR17]{paulhusrojas17}
Jennifer Paulhus and Anita~M. Rojas, \emph{Completely decomposable {J}acobian
  varieties in new genera}, Exp. Math. \textbf{26} (2017), no.~4, 430--445.
  \MR{3684576}

\bibitem[Pri08]{Pr:sg}
Rachel Pries, \emph{A short guide to {$p$}-torsion of abelian varieties in
  characteristic {$p$}}, Computational arithmetic geometry, Contemp. Math.,
  vol. 463, Amer. Math. Soc., Providence, RI, 2008, math.NT/0609658,
  pp.~121--129. \MR{MR2459994 (2009m:11085)}

\bibitem[Pri09]{Pr:large}
\bysame, \emph{The {$p$}-torsion of curves with large {$p$}-rank}, Int. J.
  Number Theory \textbf{5} (2009), no.~6, 1103--1116. \MR{MR2569747}

\bibitem[Pri19]{priesCurrent}
\bysame, \emph{Current results on {N}ewton polygons of curves}, Open problems
  in arithmetic algebraic geometry, Adv. Lect. Math. (ALM), vol.~46, Int.
  Press, Somerville, MA, [2019] \copyright 2019, pp.~179--207. \MR{3971184}

\bibitem[Pri25]{Pssg=4}
\bysame, \emph{Some cases of {O}ort's conjecture about {N}ewton polygons of
  curves}, Nagoya Math. J. \textbf{257} (2025), 93--103. \MR{4897411}

\bibitem[PS]{paulhussutherland}
Jennifer Paulhus and Drew Sutherland, \emph{Completely decomposable modular
  {J}acobians}, https://arxiv.org/html/2502.16007v1.

\bibitem[PU21]{PriesUlmerBT1}
Rachel Pries and Douglas Ulmer, \emph{On {$BT_1$} group schemes and {F}ermat
  curves}, New York J. Math. \textbf{27} (2021), 705--739. \MR{4250272}

\bibitem[PU22]{priesulmerPAMS}
\bysame, \emph{Every {$BT_1$} group scheme appears in a {J}acobian}, Proc.
  Amer. Math. Soc. \textbf{150} (2022), no.~2, 525--537. \MR{4356165}

\bibitem[PW15]{PW12}
Rachel Pries and Colin Weir, \emph{The {E}kedahl-{O}ort type of {J}acobians of
  {H}ermitian curves}, Asian J. Math. \textbf{19} (2015), no.~5, 845--869.
  \MR{3431681}

\bibitem[PZ12]{prieszhu}
Rachel Pries and Hui~June Zhu, \emph{The {$p$}-rank stratification of
  {A}rtin-{S}chreier curves}, Ann. Inst. Fourier (Grenoble) \textbf{62} (2012),
  no.~2, 707--726. \MR{2985514}

\bibitem[Ran81]{ran}
Ziv Ran, \emph{On subvarieties of abelian varieties}, Invent. Math. \textbf{62}
  (1981), no.~3, 459--479. \MR{604839}

\bibitem[Re01]{Re}
Riccardo Re, \emph{The rank of the {C}artier operator and linear systems on
  curves}, J. Algebra \textbf{236} (2001), no.~1, 80--92. \MR{1808346}

\bibitem[Roh09]{Rohde}
Jan~Christian Rohde, \emph{Cyclic coverings, {C}alabi-{Y}au manifolds and
  complex multiplication}, Lecture Notes in Mathematics, vol. 1975,
  Springer-Verlag, Berlin, 2009. \MR{2510071}

\bibitem[RR96]{rapoport-richartz}
M.~Rapoport and M.~Richartz, \emph{On the classification and specialization of
  {$F$}-isocrystals with additional structure}, Compositio Math. \textbf{103}
  (1996), no.~2, 153--181. \MR{1411570}

\bibitem[Sch85]{Schoof}
Ren\'{e} Schoof, \emph{Elliptic curves over finite fields and the computation
  of square roots mod {$p$}}, Math. Comp. \textbf{44} (1985), no.~170,
  483--494. \MR{777280}

\bibitem[SD74]{swinnertondyer}
H.~P.~F. Swinnerton-Dyer, \emph{Analytic theory of abelian varieties}, London
  Mathematical Society Lecture Note Series, vol. No. 14, Cambridge University
  Press, London-New York, 1974. \MR{366934}

\bibitem[Ser68]{Se:lf}
J.-P. Serre, \emph{Corps locaux}, Hermann, 1968.

\bibitem[Ser83]{serre82}
Jean-Pierre Serre, \emph{Nombres de points des courbes alg\'{e}briques sur
  {${\bf F}\sb{q}$}}, Seminar on number theory, 1982--1983 ({T}alence,
  1982/1983), Univ. Bordeaux I, Talence, 1983, pp.~Exp. No. 22, 8. \MR{750323}

\bibitem[Shi64]{Shimurapurely}
Goro Shimura, \emph{On purely transcendental fields automorphic functions of
  several variable}, Osaka Math. J. \textbf{1} (1964), no.~1, 1--14.
  \MR{176113}

\bibitem[Shi86]{shiota}
Takahiro Shiota, \emph{Characterization of {J}acobian varieties in terms of
  soliton equations}, Invent. Math. \textbf{83} (1986), no.~2, 333--382.
  \MR{818357}

\bibitem[Sil09]{aec}
Joseph~H. Silverman, \emph{The arithmetic of elliptic curves}, second ed.,
  Graduate Texts in Mathematics, vol. 106, Springer, Dordrecht, 2009.

\bibitem[Sti09]{sti09}
Henning Stichtenoth, \emph{Algebraic function fields and codes}, second ed.,
  Graduate Texts in Mathematics, vol. 254, Springer-Verlag, Berlin, 2009.
  \MR{2464941 (2010d:14034)}

\bibitem[Str]{strengss}
Marco Streng, \emph{Explicit supersingular cyclic curves},
  https://arxiv.org/pdf/2501.14902.

\bibitem[Sub75]{subrao}
Dor{\'e} Subrao, \emph{The {$p$}-rank of {A}rtin-{S}chreier curves},
  Manuscripta Math. \textbf{16} (1975), no.~2, 169--193. \MR{0376693}

\bibitem[SV87]{stohrvoloch}
Karl-Otto St\"{o}hr and Jos\'{e}~Felipe Voloch, \emph{A formula for the
  {C}artier operator on plane algebraic curves}, J. Reine Angew. Math.
  \textbf{377} (1987), 49--64. \MR{887399}

\bibitem[Tat66]{tate:endo}
John Tate, \emph{Endomorphisms of abelian varieties over finite fields},
  Invent. Math. \textbf{2} (1966), 134--144. \MR{0206004 (34 \#5829)}

\bibitem[Tsi12]{tsimermannotisog}
Jacob Tsimerman, \emph{The existence of an abelian variety over
  {$\overline{\Bbb Q}$} isogenous to no {J}acobian}, Ann. of Math. (2)
  \textbf{176} (2012), no.~1, 637--650. \MR{2925392}

\bibitem[Tsi18]{tsimermanandreoort}
\bysame, \emph{The {A}ndr\'e-{O}ort conjecture for {$\mathcal A_g$}}, Ann. of
  Math. (2) \textbf{187} (2018), no.~2, 379--390. \MR{3744855}

\bibitem[vdG99]{V:cycles}
Gerard van~der Geer, \emph{Cycles on the moduli space of abelian varieties},
  Moduli of curves and abelian varieties, Aspects Math., vol. E33, Friedr.
  Vieweg, Braunschweig, 1999, pp.~65--89. \MR{1722539}

\bibitem[vdGO99]{geerOort99}
Gerard van~der Geer and Frans Oort, \emph{Moduli of abelian varieties: a short
  introduction and survey}, Moduli of curves and abelian varieties, Aspects
  Math., vol. E33, Friedr. Vieweg, Braunschweig, 1999, pp.~1--21. \MR{1722536}

\bibitem[vdGvdV92]{VdGVdV92}
Gerard van~der Geer and Marcel van~der Vlugt, \emph{Reed-{M}uller codes and
  supersingular curves. {I}}, Compositio Math. \textbf{84} (1992), no.~3,
  333--367. \MR{1189892 (93k:14038)}

\bibitem[vdGvdV95]{VdGVdV}
\bysame, \emph{On the existence of supersingular curves of given genus}, J.
  Reine Angew. Math. \textbf{458} (1995), 53--61. \MR{1310953 (95k:11084)}

\bibitem[Vis89]{V:stack}
Angelo Vistoli, \emph{Intersection theory on algebraic stacks and on their
  moduli spaces}, Invent. Math. \textbf{97} (1989), no.~3, 613--670.
  \MR{MR1005008 (90k:14004)}

\bibitem[Vol88]{volochhyp2}
Jos\'e{}~Felipe Voloch, \emph{A note on algebraic curves in characteristic
  {$2$}}, Comm. Algebra \textbf{16} (1988), no.~4, 869--875. \MR{932639}

\bibitem[VW13]{viehmann-wedhorn}
Eva Viehmann and Torsten Wedhorn, \emph{Ekedahl-{O}ort and {N}ewton strata for
  {S}himura varieties of {PEL} type}, Math. Ann. \textbf{356} (2013), no.~4,
  1493--1550. \MR{3072810}

\bibitem[Wei48a]{weil}
Andr{\'e} Weil, \emph{Sur les courbes alg\'ebriques et les vari\'et\'es qui
  s'en d\'eduisent}, Actualit\'es Sci. Ind., no. 1041, Hermann et Cie., Paris,
  1948.

\bibitem[Wei48b]{weil2}
\bysame, \emph{Vari\'et\'es ab\'eliennes et courbes alg\'ebriques},
  Actualit\'es Sci. Ind., no. 1064, Hermann \& Cie., Paris, 1948.

\bibitem[Wel84]{welters84}
G.~E. Welters, \emph{A criterion for {J}acobi varieties}, Ann. of Math. (2)
  \textbf{120} (1984), no.~3, 497--504. \MR{769160}

\bibitem[Wew98]{wewersthesis}
Stefan Wewers, \emph{Construction of {H}urwitz spaces}, Dissertation, 1998.

\bibitem[Yui78]{yui78}
Noriko Yui, \emph{On the {J}acobian varieties of hyperelliptic curves over
  fields of characteristic {$p>2$}}, J. Algebra \textbf{52} (1978), no.~2,
  378--410. \MR{491717}

\bibitem[Yui80]{yuiFermat}
\bysame, \emph{On the {J}acobian variety of the {F}ermat curve}, J. Algebra
  \textbf{65} (1980), no.~1, 1--35. \MR{578793}

\bibitem[YY09]{YuYu}
Chia-Fu Yu and Jeng-Daw Yu, \emph{Mass formula for supersingular abelian
  surfaces}, J. Algebra \textbf{322} (2009), no.~10, 3733--3743. \MR{2568360}

\bibitem[Zaa99]{Zaal}
C.~G. Zaal, \emph{A complete surface in {$M_6$} in characteristic {$>2$}},
  Compositio Math. \textbf{119} (1999), no.~2, 209--212. \MR{1723129}

\bibitem[Zho20]{zhou20}
Zijian Zhou, \emph{Ekedahl-{O}ort strata on the moduli space of curves of genus
  four}, Rocky Mountain J. Math. \textbf{50} (2020), no.~2, 747--761.
  \MR{4104409}

\end{thebibliography}

\printindex

\end{document}